\documentclass[11pt,twoside]{article}
\usepackage{amsmath}
\usepackage{amssymb}
\usepackage{amsthm}
\usepackage{mathrsfs}
\usepackage{simplewick}
\usepackage{times}
\usepackage{color}
\usepackage{hep}
\usepackage{leftidx}
\usepackage{graphicx}
\usepackage{float}
\usepackage{enumerate}
\usepackage{titletoc}
\usepackage{epstopdf}
\usepackage{ulem}
\usepackage{appendix}
\usepackage{hyperref}

\allowdisplaybreaks

\DeclareSymbolFont{EulerExtension}{U}{euex}{m}{n}
\DeclareMathSymbol{\euintop}{\mathop}{EulerExtension}{"52}
\DeclareMathSymbol{\euointop}{\mathop}{EulerExtension}{"48}

\def\al{\alpha}
\def\be{\beta}
\def\de{\delta}
\def\De{\Delta}
\def\ep{\epsilon}
\def\et{\eta}
\def\ga{\gamma}
\def\Ga{\Gamma}
\def\ka{\kappa}
\def\la{\lambda}
\def\La{\Lambda}
\def\na{\nabla}
\def\om{\omega}
\def\Om{\Omega}
\def\si{\sigma}
\def\Si{\Sigma}
\def\th{\theta}
\def\ze{\zeta}
\def\ve{\varepsilon}
\def\ve0{\varepsilon_{0}}
\def\vp{\varphi}
\def\vt{\vartheta}

\def\chib{\underline{\chi}}

\def\chic{\check{\chi}}

\def\d{\textrm{d}}

\def\divg{\textrm{div}_g}
\def\divgs{{\textrm{div}}_{\slashed{g}}}

\def\ds{\slashed{\textrm{d}}}
\def\dse{\displaystyle}
\def\D{\mathscr{D}}
\def\Des{\slashed{\Delta}}

\def\Et{\tilde{E}}

\def\f{\frac}

\def\Ft{\tilde{F}}

\def\gs{\slashed{g}}

\def\H{\mathcal{H}}
\def\Hc{\check{H}}
\def\J{\mathcal{J}}

\def\K{\mathcal{K}}
\def\Kt{\tilde{K}}
\def\les{\lesssim}
\def\lot{\text{l.o.t.}}
\def\Lb{\underline{L}}

\def\Lc{\check{L}}
\def\Lie{\mathcal{L}}

\def\Lies{\slashed{\mathcal{L}}}

\def\Lr{\mathring{L}}

\def\M{\mathscr{M}}
\def\Mc{\check{\mathscr{M}}}
\def\nas{\slashed{\nabla}}

\def\N{\mathscr{N}}

\def\p{\partial}
\def\pis{\slashed{\pi}}

\def\R{\mathscr{R}}
\def\Rc{\check{\mathscr{R}}}
\def\t{\mathfrak{t}}
\def\tu{\bar{\mathfrak{t}}}

\def\trg{{\textrm{tr}}_{g}}
\def\trgs{{\textrm{tr}}_{\slashed{g}}}

\def\Tc{\check{T}}

\def\Tr{\mathring{T}}

\def\Zu{\bar{Z}}

\begin{document}
\footskip=0pt
\footnotesep=2pt
\let\oldsection\section
\renewcommand\section{\setcounter{equation}{0}\oldsection}
\renewcommand\thesection{\arabic{section}}
\renewcommand\theequation{\thesection.\arabic{equation}}
\newtheorem{claim}{\noindent Claim}[section]
\newtheorem{theorem}{\noindent Theorem}[section]
\newtheorem{lemma}{\noindent Lemma}[section]
\newtheorem{proposition}{\noindent Proposition}[section]
\newtheorem{definition}{\noindent Definition}[section]
\newtheorem{remark}{\noindent Remark}[section]
\newtheorem{corollary}{\noindent Corollary}[section]
\newtheorem{example}{\noindent Example}[section]

\title{On the critical exponent $p_c$ of the 3D quasilinear wave equation $-\big(1+(\p_t\phi)^p\big)\p_t^2\phi+\De\phi=0$ with short pulse initial data. II, shock formation}
\author{Lu Yu$^{1,*}$, \quad Yin Huicheng$^{1,}$
\footnote{Lu Yu (15850531017@163.com) and Yin Huicheng (huicheng@nju.edu.cn, 05407@njnu.edu.cn)
are supported by the NSFC (No.11731007, No.12126338).}\vspace{0.5cm}\\
\small School of Mathematical Sciences and Mathematical Institute, Nanjing Normal University,\\
\small Nanjing, 210023, China.}
\date{}
\maketitle

\centerline {\bf Abstract} \vskip 0.3 true cm
In the previous paper [Ding Bingbing, Lu Yu, Yin Huicheng, On the critical exponent $p_c$ of the 3D quasilinear wave equation $-\big(1+(\partial_t\phi)^p\big)\partial_t^2\phi+\Delta\phi=0$ with short pulse initial data. I, global existence, Preprint, 2022],
for the 3D quasilinear wave equation $-\big(1+(\p_t\phi)^p\big)\p_t^2\phi+\De\phi=0$ with short pulse initial data $(\phi,\p_t\phi)(1,x)=\big(\de^{2-\ve0}\phi_0(\f{r-1}{\de},\om),\de^{1-\ve0}\phi_1(\f{r-1}{\de},\om)\big)$, where $p\in\mathbb{N}$, $0<\ve0<1$,
under the outgoing constraint condition
$(\p_t+\p_r)^k\phi(1,x)=O(\de^{2-\ve0-k\max\{0,1-(1-\ve0)p\}})$ for $k=1,2$, the authors establish the global existence of smooth large
solution $\phi$ when $p>p_c$ with $p_c=\f{1}{1-\ve0}$. In
the present paper, under the  same outgoing constraint condition, when $1\leq p\leq p_c$, we will show that the
smooth solution $\phi$ blows up and further the outgoing shock
is formed in finite time.

\vskip 0.2 true cm {\bf Keywords:} Short pulse initial data, critical exponent, inverse foliation density, shock formation
\vskip 0.2 true cm {\bf Mathematical Subject Classification:} 35L05, 35L72
\vskip 0.4 true cm

\tableofcontents

\section{Introduction}\label{Section 1}

As a continuation of \cite{DLY}, motivated by \cite{M-Y}, we continue to study the following 3D quasilinear wave equation
\begin{equation}\label{main equation}
	-\big(1+(\p_t\phi)^p\big)\p_t^2\phi+\De\phi=0
\end{equation}
with the short pulse initial data
\begin{equation}\label{initial data}
	(\phi,\p_t\phi)(1,x)=\big(\de^{2-\ve0}\phi_0(\f{r-1}{\de},\om),\de^{1-\ve0}\phi_1(\f{r-1}{\de},\om)\big),
\end{equation}
where $p\in\mathbb{N}$, $0<\ve0<1$, $\de>0$, $x=(x^1,x^2,x^3)\in\mathbb{R}^3$, $t\geq 1$, $r=|x|=\sqrt{(x^1)^2+(x^2)^2+(x^3)^2}$, $\om=(\om_1,\om_2,\om_3)=\f{x}{r}\in\mathbb{S}^2$, and $(\phi_0,\phi_1)(s,\om)\in C_0^{\infty}\big((-1,0)\times\mathbb{S}^2\big)$.

In \cite{DLY}, under such an outgoing constraint condition
\begin{equation}\label{outgoing constraint condition}
	(\p_t+\p_r)^k\phi(1,x)=O(\de^{2-\ve0-k\max\{0,1-(1-\ve0)p\}}),\ k=1,2,
\end{equation}
the authors have established the global existence of smooth solution $\phi$ to \eqref{main equation}-\eqref{initial data}
when $p>p_c=\f{1}{1-\ve0}$. As explained in \cite{DLY}, \eqref{outgoing constraint condition}
is very necessary in order to keep the local (or global) smallness of $\p_t\phi$ and the strict hyperbolicity of \eqref{main equation}.
On the other hand, when \eqref{initial data} is given, $(\p_t+\p_r)^k\phi(1,x)=O(\de^{2-\ve0-k})$
generally holds from \eqref{main equation}.
This will lead to the over-determination of \eqref{outgoing constraint condition} for arbitrary $(\phi_0,\phi_1)$. Thus the choice of $(\phi_0,\phi_1)$ in \eqref{initial data} together with \eqref{outgoing constraint condition}
will be somewhat restricted (see Appendix \ref{Section A} below).

For $p=2$ and $\ve0=\f12$ in \eqref{main equation}-\eqref{initial data}, when the following incoming constraint condition is posed
\begin{equation}\label{incoming constraint condition}
	(\p_t-\p_r)^k\phi(1,x)=O(\de^{\f32}),\ k=1,2,
\end{equation}
then under the suitable assumption of $(\phi_0,\phi_1)$, it is shown in \cite{M-Y} that the incoming shock will be
formed before the time $t=2$.
The short pulse initial data \eqref{initial data} with $\ve0=\f12$ are firstly introduced by D. Christodoulou in \cite{Ch2}.
For such short pulse data and by the short pulse method, the authors in monumental papers \cite{Ch2} and \cite{K-R} showed
the formation of black holes in vacuum spacetime for the 3D Einstein general relativity equations.

Our main result in the paper is
\begin{theorem}\label{main theorem}
Assume the condition \eqref{outgoing constraint condition} holds.
When $1\leq p\leq p_c=\f{1}{1-\ve0}$, for small $\de>0$, the smooth solution $\phi$ of \eqref{main equation} with \eqref{initial data} will blow up and further form the outgoing shock before the time $t^*=1+\de^{1-(1-\ve0)p}$ under the following assumption of $(\phi_0,\phi_1)$: there exists a point $(s_0,\om_0)\in(-1,0)\times\mathbb{S}^2$ such that
\begin{equation}\label{shock-formation assumption of data}
	\begin{split}
		&\phi_1^{p-1}(s_0,\om_0)\p_s\phi_1(s_0,\om_0)>\f{2}{p}\qquad \text{for $1\le p<p_c$},\\
        &\phi_1^{p-1}(s_0,\om_0)\p_s\phi_1(s_0,\om_0)>\f{(p-1)2^p}{(2^{p-1}-1)p}\qquad \text{for $p=p_c$}.
    \end{split}
\end{equation}
\end{theorem}

\begin{remark}
The assumption \eqref{shock-formation assumption of data} can be fulfilled
due to the arbitrariness of $\phi_0$ and the suitable choice of $\phi_1$ (see \eqref{phi_1} below).
\end{remark}

\begin{remark}
When equation \eqref{main equation} admits the small initial data
\begin{align}\label{Y-1}
\phi(1,x)=\de \vp_0(x), \quad \p_t\phi(1,x)=\de \vp_1(x),
\end{align}
where $\big(\vp_0(x),\vp_1(x)\big)\in C_0^{\infty}(\mathbb R^3)$, $\de>0$ is sufficiently small,
\eqref{main equation} with \eqref{Y-1} has a global smooth solution $\phi$ for $p\ge 2$
(see \cite{K-P} or Chapter 6 of
\cite{H}).
If $p=1$ holds in \eqref{main equation}, then the smooth solution $\phi$
of \eqref{main equation} with \eqref{Y-1} will blow up in finite time as long as $\big(\vp_0(x),\vp_1(x)\big)\not\equiv 0$
(see \cite{Ali1}, \cite{Ali2} and \cite{Ding1}).
\end{remark}

\begin{remark}\label{Remark equivalent outgoing constraint condition}
Due to the special forms of \eqref{main equation} and \eqref{initial data}, then \eqref{outgoing constraint condition} is
actually equivalent to
\begin{equation}\label{equivalent outgoing constraint condition}
(\p_t+c\p_r)^k\Om^{\al}\p^{\be}\phi(1,x)=O(\de^{2-\ve0-k\max\{0,1-(1-\ve0)p\}-|\be|}),\ k=1,2,\ \al\in\mathbb{N}_0^3,\ \be\in\mathbb{N}_0^4,
\end{equation}
where $\Om\in\{x^i\p_j-x^j\p_i:1\leq i<j\leq 3\}$ stands for the derivatives on $\mathbb{S}^2$, $\p\in\{\p_t,\p_{x^1},\p_{x^2},\p_{x^3}\}=\{\p_t,\p_1,\p_2,\p_3\}$, and $c=\big(1+(\p_t\phi)^p\big)^{-\f12}$ is the wave speed with $c \sim 1$ (see \eqref{L^infty estimates of c}).
\end{remark}

\begin{remark}\label{Short pulse initial data}
Consider the semilinear symmetric hyperbolic system
\begin{equation}\label{YHC-5}
\begin{split}
\p_tu+\sum_{j=1}^nA_j(t,x)\p_ju=G(t,x,u)
\end{split}
\end{equation}
with the short pulse initial data
\begin{equation}\label{YHC-6}
\begin{split}
u(0,x)=\de^pu_0(x,\f{\vp(x)}{\de}),
\end{split}
\end{equation}
where $u=(u_1,\dots,u_m)^T$, $A_j(t,x)$ are smooth $m\times m$ symmetric matrix,
$G(t,x,u)$ is a smooth $1\times m$ column vector, $p\in\Bbb N_0$,
$\nabla_{t,x}\vp\not=0$, $u_0\in C^{\infty}$ is a rapidly decreasing function on its arguments.
The authors in \cite{Alt1}-\!\cite{Alt2} establish the local existence of \eqref{YHC-5}
and further give the corresponding correctors by the techniques of asymptotic analysis.
On the other hand, for the quasilinear form of \eqref{YHC-5} with \eqref{YHC-6}, the authors in \cite{Hunter}
and \cite{Majda} have also discussed some  formal asymptotic correctors. We point out that
there are no blow-up results of \eqref{YHC-5} with \eqref{YHC-6} shown in the references above.
\end{remark}

We now comment on the proof of Theorem \ref{main theorem}.
As in \cite{M-Y} or\cite{DLY},
strongly motivated by the geometric methods of D. Christodoulou \cite{Ch1},
we will introduce the inverse foliation density $\mu=-\displaystyle\frac{1}{\big(1+(\p_t\phi)^p\big)\p_tu}$,
where the  optical function $u$ satisfies
$-\big(1+(\p_t\phi)^p\big)(\p_tu)^2+\displaystyle\sum_{i=1}^3(\p_iu)^2=0$ with the initial data $u(1,x)=1-r$ and $\p_tu>0$.
Under the suitably weighted bootstrap assumptions of $\p\phi$ with the precise
smallness powers of $\de$, it follows from involved analysis and computations that
the first order transport equation $L\mu=-\f{p}{2t^p}\de^{-[1-(1-\ve0)p]}\phi_1^{p-1}(s,\omega)\p_s\phi_1(s,\omega)
+O\big(\de^{(1-\ve0)p-[1-(1-\ve0)p]}\big)$
is eventually obtained, where $L$ is a vector field approximating $\p_t+\p_r$.
By the integration from $t=1$ to $t\le t^*=1+\de^{1-(1-\ve0)p}$ along the characteristics of $L$,
one gets $\mu=1-\f{p}{2}\de^{-[1-(1-\ve0)p]}\phi_1^{p-1}(s,\omega)\p_s\phi_1(s,\omega)\f{1}{p-1}\left(1-\f{1}{t^{p-1}}\right)
+O\big(\de^{(1-\ve0)p}\big)$ for $1<p\le p_c$ and
$\mu=1-\f{p}{2}\de^{-[1-(1-\ve0)p]}\phi_1^{p-1}(s,\omega)\p_s\phi_1(s,\omega)\ln t
+O\big(\de^{(1-\ve0)p}\big)$ for $p=1$, which derives  $\mu\to 0$  and further the shock forms before $t^*$
under the assumption \eqref{shock-formation assumption of data}.
To prove the bootstrap assumptions of $\p\phi$,  we will use the energy method as in \cite{Ch2},
\cite{M-Y} and \cite{Sp}. However, there are still some new features
in the present paper as follows.

$\bullet$ Through analyzing  the special structures of short pulse initial data and equation \eqref{main equation},
we can choose a class of $(\phi_0,\phi_1)$ to fulfill \eqref{outgoing constraint condition} (see Appendix \ref{Section A}).
Moreover, for $1\le p<p_c$,  the loss of smallness order $\de^{1-(1-\ve0)p}$ of $\phi(1,x)$ along the directional derivative
$\p_t+\p_r$ is
naturally derived (see \eqref{outgoing constraint condition}). Note that in \cite{M-Y}, the authors
look for $(\phi_0,\phi_1)$ to satisfy \eqref{incoming constraint condition}
by solving the resulting ordinary differential equations (see Lemma 1.1 of \cite{M-Y}).

$\bullet$ We directly treat the linearized equation of \eqref{main equation} under the Lorentzian
spacetime metric $g=(g_{\al\be})=\mathrm{diag}(-c^2,1,1,1)$ with $c=\big(1+(\p_t\phi)^p\big)^{-\f{1}{2}}$,
while the authors in \cite{M-Y} introduce another new metric $\tilde g$ to homogenize the resulting linearized equation
and this leads to more involved computations. On the other hand, due to our direct treatment methods for the
linearized equation of $\sum_{\alpha,\beta=0}^ng^{\al\be}\p_{\al\be}^2\phi=0$, the general quasilinear
$n$-dimensional wave equations with short pulse initial data can be investigated by the analogous manners (one can also see \cite{Ding2}-\!\cite{Ding3}).

$\bullet$ Due to the over-determination of \eqref{outgoing constraint condition} and further the loss of smallness with
respect to the orders of $\de$ along $\p_t+\p_r$  when $1\le p<p_c$
(this is slightly different from the no loss condition of smallness in \eqref{incoming constraint condition} of \cite{M-Y}
because of  the corresponding $p=2$, $\ve0=\f12$ and $\max\{0,1-(1-\ve0)p\}=0$ in \eqref{outgoing constraint condition}), we
require to deliberately  choose the weighted bootstrap assumptions of $\p\phi$
and estimate the optimal powers of $\delta$ for all the related quantities
so that the shock formation of equation \eqref{main equation} can be shown when $1\le p\le p_c$.

$\bullet$ To close the bootstrap assumptions of $\p\phi$ up to the $N$-order derivatives ($N\in\Bbb N$ is a suitably
large integer),   we need to obtain the
uniform positive constant $M$ in bootstrap assumptions (see \eqref{bootstrap assumptions}). For this purpose, our
ingredient is:  at first,
acting the second order operator $L\Lb$ ($\Lb$ is a first order differential operator approximate to $\p_t-\p_r$) on the
corresponding quantities and make full use of the special form of \eqref{main equation} to
obtain the uniform positive constant $C$ (independent of $M$) in the bootstrap assumptions
up to $(N-2)$-order derivatives of $\p\phi$. Secondly, based on this,
the bootstrap assumptions for the $(N-1)$-order and $N$-order derivatives of $\p\phi$ can be closed by the energy estimates
since the related constants only depend on $C$ rather than $M$.
In \cite{M-Y}, the bootstrap assumptions  up to $N$-orders are treated together. If we take the same methods in \cite{M-Y},
it will  cause difficulties for us to close the bootstrap assumptions
because some estimated constants  in the energy estimates will depend on the given a priori constant $M$
and one can not get the uniform constant $C$ (independent of $M$) in the bootstrap assumptions.

Our paper is organized as follows.
In Section \ref{Section 2}, we give  some preliminaries as in \cite{DLY}.
In Section \ref{Section 3}, the bootstrap assumptions are listed. Meanwhile,  the lower order $L^\infty$ estimates of some
quantities and the behavior of $\mu$ near blow-up time are derived.
In Section \ref{Section 4}, under the bootstrap assumptions, we will establish the higher order
derivative $L^{\infty}$ estimates and close the
partial bootstrap assumptions.
In Section \ref{Section 5}, we first take the energy estimates for the linearized covariant
wave equation and define some suitable higher order weighted energies and fluxes as in \cite{Ch2} and \cite{M-Y}-\!\cite{Sp}.
In addition, the commuted covariant wave equations are
derived.
In Section \ref{Section 6}, we first establish the non-top order $L^2$ estimates without derivative loss.
Then the top order $L^2$ estimates for $\p\chi$ and $\p^2\mu$ are investigated, where $\chi$ is the second
fundamental form of metric $g$ with respect to $L$.
In Section \ref{Section 7}, we treat the estimates of the error terms and derive the top order energy estimates
on the resulting covariant wave equations.
In Section \ref{Section 8}, collecting all the estimates in the previous sections, we complete the bootstrap argument.
From this, the mechanism of shock formation is shown and the proof of the main theorem is also completed.

\vskip 0.2 true cm
As in \cite{DLY}, throughout the whole paper, without special mentions, the following conventions are used:
\begin{itemize}
  \item Greek letters $\{\al,\be,\ga,\cdots\}$ corresponding to the spacetime coordinates are chosen in \{0,1,2,3\}; Latin letters $\{i,j,k,\cdots\}$ corresponding to the spatial coordinates are chosen in \{1,2,3\}; Capital letters $\{A,B,C,\cdots\}$ corresponding to the sphere coordinates are chosen in \{1,2\}.

  \item We use the Einstein summation convention to sum over the repeated upper and lower indices.

  \item The convention $f\lesssim h$ means that there exists a generic positive constant $C$ independent of the parameter $\de$ and the variables $(t,x)$ such that $f\leq Ch$.

  \item If $\xi$ is a $(0,2)$-type spacetime tensor, $\La$ is a one-form, $U$ and $V$ are vector fields, then
  the contraction of $\xi$ with respect to $U$ and $V$ is defined as $\xi_{UV}=\xi_{\al\be}U^{\al}V^{\be}$,
  and the contraction of $\La$ with respect to $U$ is defined as $\La_{U}=\La_{\al}U^{\al}$.

  \item The restriction of quantity $\xi$ (including the metric $g$ or any $(m,n)$-type
  spacetime tensor field) on the sphere is represented by $\slashed{\xi}$. But if $\xi$ is already defined on the sphere, it is still represented by $\xi$.

  \item $\Lie_V\xi$ stands for the Lie derivative of $\xi$ with respect to vector field $V$, and $\Lies_V\xi$ is the restriction of $\Lie_V\xi$ on the sphere.

  \item $c=\big(1+(\p_t\phi)^p\big)^{-\frac{1}{2}}$ is the wave speed.

  \item $\vp=(\vp_0,\vp_1,\vp_2,\vp_3)=(\p_t\phi, \p_{x^1}\phi, \p_{x^2}\phi, \p_{x^3}\phi)$.

  \item $t^*=1+\de^{1-(1-\ve0)p}$.

  \item For $p=1$, define $\dse\f{1}{p-1}\big(1-\f{1}{t^{p-1}}\big)$ as $\dse\lim_{p\to 1}\f{1}{p-1}\big(1-\f{1}{t^{p-1}}\big)=\ln t$
  for $t>1$.
\end{itemize}

\section{Some preliminaries}\label{Section 2}

In this section, we give some preliminaries on the Lorentzian geometry as in \cite{DLY}, which have been also
illustrated in \cite{Ch1}, \cite{M-Y} and \cite{Sp}. However, for readers' convenience, we still give the details.

By \eqref{main equation}, it is natural to introduce such an inverse spacetime metric
\begin{equation}\label{g^-1}
	\begin{split}
		g^{-1}=(g^{\al\be})=\mathrm{diag}(-c^{-2},1,1,1),
	\end{split}
\end{equation}
and the corresponding spacetime metric
\begin{equation}\label{g}
	\begin{split}
		g=(g_{\al\be})=\mathrm{diag}(-c^2,1,1,1).
	\end{split}
\end{equation}
In this case, \eqref{main equation} can be rewritten as
\begin{equation}\label{quasilinear wave equation}
	g^{\al\be}\p_{\al\be}^2\phi=0.
\end{equation}
The related Christoffel symbols of $g$ are defined by
\begin{equation}\label{Christoffel symbol}
	\begin{split}
		&\Ga_{\al\be\ga}=\f12(\p_{\al}g_{\be\ga}+\p_{\ga}g_{\al\be}-\p_{\be}g_{\al\ga}),\\
		&\Ga_{\al\be}^{\ga}=g^{\ga\la}\Ga_{\al\la\be},\\
		&\Ga^{\ga}=g^{\al\be}\Ga_{\al\be}^{\ga}.
	\end{split}
\end{equation}

One can introduce the following optical function as in \cite{Ch1}.
\begin{definition}\label{Definition optical function}
A $C^1$ function $u(t,x)$ is called the optical function of \eqref{quasilinear wave equation}, if $u(t,x)$ satisfies the eikonal equation
\begin{equation}\label{eikonal equation}
	g^{\al\be}\p_{\al}u\p_{\be}u=0
\end{equation}
with the initial data $u(1,x)=1-r$ and the condition $\p_tu>0$.
\end{definition}

According to the definition of optical function, define the inverse foliation density $\mu$ as in \cite{Ch1} and \cite{Sp}:
\begin{equation}\label{inverse foliation density}
	\mu=-\f{1}{g^{\al\be}\p_{\al}u\p_{\be}t}\big(=\f{1}{c^{-2}\p_tu}\big).
\end{equation}

In fact, $\mu^{-1}$ measures the foliation density of the outgoing
characteristic surfaces
and indicates that the  shock will be formed when $\mu\rightarrow 0+$ with the development of time $t$.

Note that
\begin{equation}\label{Lr}
	\Lr=-\na u=-g^{\al\be}\p_{\al}u\p_{\be}
\end{equation}
is a tangent vector field of the outgoing light cone $\{u=C\}$ and $\Lr t=\mu^{-1}$ holds. One naturally rescale $\Lr$ as
\begin{equation}\label{L}
	L=\mu\Lr.
\end{equation}
To obtain a tangent vector field $\Lb$ of the incoming light cone, let
\begin{equation}\label{Tr}
	\Tr=c^{-1}(\p_t-L).
\end{equation}
Set
\begin{align}
	T&=c^{-1}\mu\Tr,\label{T}\\
	\Lb&=c^{-2}\mu L+2T.\label{Lb}
\end{align}
Then $L$ and $\Lb$ are two vector fields in the null frame. About the other vector fields $\{X_1,X_2\}$ in the null frame,
we utilize the vector field $L$ to construct them. To this end, one extends the local coordinates $\{\th^1,\th^2\}$ on $\mathbb{S}^2$ as follows
\begin{equation*}
	L\vt^{A}=0,\ \vt^{A}|_{t=1}=\th^{A},
\end{equation*}
here and below $A=1,2$. Subsequently, let
\begin{equation}\label{X}
	X_1=\f{\p}{\p\vt^1},\ X_2=\f{\p}{\p\vt^2}.
\end{equation}

A direct computation yields
\begin{lemma}\label{Lemma null frame}
	$\{L,\Lb,X_1,X_2\}$ constitutes a null frame with respect to the metric $g$ in \eqref{g}, and admits the following identities
	\begin{equation*}
		g(L,L)=g(\Lb,\Lb)=g(L,X_A)=g(\Lb,X_A)=0,\ g(L,\Lb)=-2\mu.
	\end{equation*}
	In addition,
	\begin{equation*}
		g(L,T)=-\mu,\ g(T,T)=c^{-2}\mu^2.
	\end{equation*}
	Moreover,
	\begin{equation*}
		Lt=1,\ Lu=0,\ Tt=0,\ Tu=1,\ \Lb t=c^{-2}\mu,\ \Lb u=2,\ X_At=X_Au=0.
	\end{equation*}
\end{lemma}

\begin{remark}\label{Remark T and L on t=1}
On $t=1$, by $u=1-r$, then $T=-\p_r$. Therefore, in terms of \eqref{Tr}, we have $L=\p_t+c\p_r$
for $t=1$.
\end{remark}

\begin{lemma}\label{Lemma componets of gs}
According to \cite[Lemma 3.47]{Sp}, the components of $\gs$ satisfy
	\begin{equation*}
		\begin{split}
			&\gs^{\al\be}=g^{\al\be}+\f12\mu^{-1}(L^{\al}\Lb^{\be}+\Lb^{\al}L^{\be}),\\
			&\gs^{\al\be}=\gs^{AB}X_A^{\al}X_B^{\be},
		\end{split}
	\end{equation*}
	and for any smooth function $\Psi$,
	\begin{equation*}
		\gs^{\al\be}\p_{\al}\Psi\p_{\be}\Psi=|\ds\Psi|^2:=\gs^{AB}\ds_A\Psi\ds_B\Psi,
	\end{equation*}
	where $\ds$ stands for the restriction of differential operator $\d$ on the sphere and $\ds_Af=X_Af$ for any smooth function $f$.
\end{lemma}
\vskip 0.1 true cm

We now perform the change of coordinates as in \cite{DLY}:
$(t,x^1,x^2,x^3)\rightarrow(\t,u,\vt^1,\vt^2)$ with
\begin{equation}\label{change of coordinates}
	\left\{
	\begin{aligned}
		&\t=t,\\
		&u=u(t,x),\\
		&\vt^1=\vt^1(t,x),\\
		&\vt^2=\vt^2(t,x).
	\end{aligned}
	\right.
\end{equation}
For notational convenience,  the following domains are introduced
\begin{definition}\label{Definition domains}
	\begin{align*}
		\Si_\t&=\{(\t',u',\vt):\ \t'=\t\},\\
		\Si_\t^u&=\{(\t',u',\vt):\ \t'=\t,0\leq u'\leq u\},\\
		C_u&=\{(\t',u',\vt):\ \t'\geq 1,u'=u\},\\
		C_u^\t&=\{(\t',u',\vt):\ 1\leq\t'\leq\t,u'=u\},\\
		S_{\t,u}&=\Si_\t\cap C_u,\\
		D^{\t,u}&=\{(\t',u',\vt):\ 1\leq\t'<\t,0\leq u'\leq u\}.
	\end{align*}
\end{definition}

Note that $\vt=(\vt^1,\vt^2)$ are the coordinates on sphere $S_{\t,u}$, then under the new coordinate system $(\t,u,\vt^1,\vt^2)$, one has $L=\f{\p}{\p\t}$, $T=\f{\p}{\p u}-\Xi$ with $\Xi=\Xi^AX_A$. In addition, it follows from direct computations that
\begin{lemma}\label{Lemma jacobian}
	In $D^{\t,u}$, the Jacobian determinant of map $(\t,u,\vt^1,\vt^2)\rightarrow(t,x^1,x^2,x^3)$ is
	\begin{equation*}
		\det\f{\p(t,x^1,x^2,x^3)}{\p(\t,u,\vt^1,\vt^2)}=c^{-1}\mu\sqrt{\det\gs}.
	\end{equation*}
\end{lemma}

\begin{remark}\label{Remark jacobian}
	Since the metric $\gs$ are locally regular, that is, $\det\gs>0$ (see Remark \ref{Remark sqrt det gs is bounded above and below}
below), it is clear that the coordinate transformation will fail to be a diffeomorphism when $\mu\rightarrow 0+$.
\end{remark}

We now give some integrations and $L^2$-norms, which will be utilized repeatedly in subsequent sections.
\begin{definition}\label{Definition norms}
	For any continuous function $f$, set
	\begin{align*}
		&\int_{S_{\t,u}}f=\int_{S_{\t,u}}f(\t,u,\vt)\sqrt{\det\gs(\t,u,\vt)}\d\vt,\ \|f\|_{L^2(S_{\t,u})}^2=\int_{S_{\t,u}}|f|^2,\\
		&\int_{C_u^\t}f=\int_{1}^\t\int_{S_{\t',u}}f(\t',u,\vt)\sqrt{\det\gs(\t',u,\vt)}\d\vt\d\t',\ \|f\|_{L^2(C_u^\t)}^2=\int_{C_u^\t}|f|^2,\\
		&\int_{\Si_\t^u}f=\int_{0}^u\int_{S_{\t,u'}}f(\t,u',\vt)\sqrt{\det\gs(\t,u',\vt)}\d\vt\d u',\ \|f\|_{L^2(\Si_\t^u)}^2=\int_{\Si_\t^u}|f|^2,\\
		&\int_{D^{\t,u}}f=\int_{1}^\t\int_{0}^u\int_{S_{\t',u'}}f(\t',u',\vt)\sqrt{\det\gs(\t',u',\vt)}\d\vt\d u'\d\t',\ \|f\|_{L^2(D^{\t,u})}^2=\int_{D^{\t,u}}|f|^2.
	\end{align*}
\end{definition}

\vskip 0.1 true cm

Let $\D$ and $\nas$ be the Levi-Civita connections of $g$ and $\gs$ respectively.
Under the null frame $\{L,\Lb,X_1,X_2\}$, define the second fundamental forms $\chi$, $\chib$ and $\si$ as
\begin{equation}\label{second fundamental forms}
	\chi_{AB}=g(\D_AL,X_B),\ \chib_{AB}=g(\D_A\Lb,X_B),\ \si_{AB}=g(\D_A\Tr,X_B).
\end{equation}
And the torsion one-forms $\ze$ and $\et$ are denoted by
\begin{equation}\label{torsion one forms}
	\ze_A=g(\D_AL,T),\ \et_A=-g(\D_AT,L).
\end{equation}

As in \cite{M-Y}, direct computations yield
\begin{lemma}\label{Lemma components of second fundamental forms and torsion one forms}
	The components of the second fundamental forms and torsion one-forms satisfy the following relation
	\begin{equation*}
		\begin{split}
			&\chib_{AB}=-c^{-2}\mu\chi_{AB},\ \si_{AB}=-c^{-1}\chi_{AB},\\
			&\ze_A=-c^{-1}\mu\ds_Ac,\ \et_A=c\ds_A(c^{-1}\mu).
		\end{split}
	\end{equation*}
\end{lemma}

\begin{lemma}\label{Lemma connection coefficients}
	For the connection coefficients of the related frames, it holds that
	\begin{equation*}
		\begin{split}
			&\D_LL=\mu^{-1}L\mu L,\ \D_TL=\et^AX_A-c^{-1}L(c^{-1}\mu)L,\ \D_AL=-\mu^{-1}\ze_AL+\chi_A^BX_B,\\
			&\D_LT=-\ze^AX_A-c^{-1}L(c^{-1}\mu)L,\ \D_AT=\mu^{-1}\et_AT-c^{-2}\mu\chi_A^BX_B,\\
			&\D_TT=c^{-3}\mu[Tc+L(c^{-1}\mu)]L+\{c^{-1}[Tc+L(c^{-1}\mu)]+T\ln(c^{-1}\mu)\}T-c^{-1}\mu\ds ^A(c^{-1}\mu)X_A,\\
			&\D_LX_A=\D_AL,\ \D_TX_A=\D_AT,\ \D_AX_B=\nas_AX_B+\mu^{-1}\chi_{AB}T,\ \D_A\Lb=\mu^{-1}\et_A\Lb-c^{-2}\mu\chi_A^BX_B,\\
			&\D_{\Lb}L=-L(c^{-2}\mu)L+2\et^AX_A,\ \D_L\Lb=-2\ze^AX_A,\ \D_{\Lb}\Lb=[\mu^{-1}\Lb\mu+L(c^{-2}\mu)]\Lb-2\mu\ds^A(c^{-2}\mu)X_A,
		\end{split}
	\end{equation*}
	where $\ze^A=\gs^{AB}\ze_B$, $\et^A=\gs^{AB}\et_B$ and $\chi_A^B=\gs^{BC}\chi_{AC}$.
\end{lemma}

\vskip 0.1 true cm

On $\t=1$, one has $L^0=1$, $L^i=\f{x^i}{r}+O(\de^{(1-\ve0)p})$, $\Tr^i=-\f{x^i}{r}$ and $\chi_{AB}=\f{c}{r}\gs_{AB}$ with $r=1-u$. Therefore, for $\t\geq 1$, we define the ``error vectors" with the components being
	\begin{equation*}
		\begin{split}
			\Lc^0&=0,\\
			\Lc^i&=L^i-\f{x^i}{\rho},\\
			\Tc^i&=\Tr^i+\f{x^i}{\rho},\\
			\chic_{AB}&=\chi_{AB}-\f{1}{\rho}\gs_{AB},
		\end{split}
	\end{equation*}
	here and below $\rho=\t-u$, $\trgs\chi=\gs^{AB}\chi_{AB}$. We point out that
these error vectors will admit good smallness orders of $\de$ in subsequent estimates.
In addition, it follows from direct computation that the error vectors satisfy
	\begin{equation}\label{YHC-1}
		\begin{split}
			\Lc^i&=-c\Tc^i+\rho^{-1}(c-1)x^i,\\
			\trgs\chic&=\trgs\chi-2\rho^{-1},\\
			|\chic|^2&=|\chi|^2-2\rho^{-1}\trgs\chi+2\rho^{-2}.
		\end{split}
	\end{equation}

Let
\begin{equation}\label{v_i}
	v_i=g(\Om_i,\Tr)=\ep_{ijk}x^j\Tc^k,
\end{equation}
then
\begin{equation}\label{R_i}
	R_i=\Om_i-v_i\Tr
\end{equation}
are the rotation vector fields of $S_{\t,u}$, where $\Om_i=\ep_{ij}{}^kx^j\p_k,\ i=1,2,3$, $\ep_{ij}{}^k=\ep_{ijk}=-1$ when $ijk$ is $123$'s odd permutation, and $\ep_{ij}{}^k=\ep_{ijk}=1$ when $ijk$ is $123$'s even permutation.

\vskip 0.1 true cm

The Riemann curvature tensor $\R$ of $g$ is defined as
\begin{equation}\label{Riemann curvature tensor}
	\R_{WXYZ}=-g(\D_W{\D_X}Y-\D_X{\D_W}Y-{\D_{[W,X]}}Y,Z).
\end{equation}
Due to
\begin{equation*}
	\R_{\mu\nu\al\be}=\p_{\nu}\Ga_{\mu\be\al}-\p_{\mu}\Ga_{\nu\be\al}+g^{\ka\la}(\Ga_{\nu\ka\al}\Ga_{\mu\la\be}-\Ga_{\mu\ka\al}\Ga_{\nu\la\be}),
\end{equation*}
the only non-vanishing component of $\R$ with respect to $g$ in \eqref{g} is
\begin{equation}\label{R_0i0j}
	\R_{0i0j}=c\p_{ij}^2c.
\end{equation}
The only non-vanishing component under the frame $\{L,T,X_1,X_2\}$ is
\begin{equation}\label{R_LALB}
	\R_{LALB}=-cTc\mu^{-1}\chi_{AB}+\Rc_{LALB},
\end{equation}
where
\begin{equation}\label{Rc_LALB}
	\begin{split}
		\Rc_{LALB}=c\nas_{AB}^2c=-\f12pc^4\vp_0^{p-1}\nas_{AB}^2\vp_0-\f32pc^3\vp_0^{p-1}\ds_Ac\ds_B\vp_0-\f12p(p-1)c^4\vp_0^{p-2}\ds_A\vp_0\ds_B\vp_0.
	\end{split}	
\end{equation}
Here we point out that $\Rc_{LALB}$ will admit the higher smallness orders of $\de$ than $\R_{LALB}$.

For any smooth function $\Psi$, one can denote its energy-momentum tensor with respect to $g$ by
\begin{equation}\label{energy-momentum tensor}
	Q_{\al\be}=Q_{\al\be}[\Psi]=\p_{\al}\Psi\p_{\be}\Psi-\f12g_{\al\be}g^{\ka\la}\p_{\ka}\Psi\p_{\la}\Psi.
\end{equation}
\begin{lemma}\label{Lemma components of energy-momentum tensor}
	The components of energy-momentum tensor in terms of the null frame $\{L,\Lb,X_1,X_2\}$ can be computed as follows
	\begin{equation}\label{components of energy-momentum tensor}
		\begin{split}
			&Q_{LL}=(L\Psi)^2,\ Q_{\Lb\Lb}=(\Lb\Psi)^2,\ Q_{L\Lb}=\mu|\ds\Psi|^2,\\
			&Q_{LA}=L\Psi\ds_A\Psi,\ Q_{\Lb A}=\Lb\Psi\ds_A\Psi,\ Q_{AB}=\ds_A\Psi\ds_B\Psi-\f12\gs_{AB}(|\ds\Psi|^2-\mu^{-1}L\Psi\Lb\Psi).
		\end{split}
	\end{equation}
\end{lemma}

For any vector field $V$, denote its deformation tensor with respect to $g$ by
\begin{equation}\label{deformation tensor}
	{}^{(V)}\pi_{\al\be}=g(\D_{\al}V,\p_{\be})+g(\D_{\be}V,\p_{\al}).
\end{equation}
Moreover, for any two vector fields $X$ and $Y$, one has
\begin{equation*}
	{}^{(V)}\pi_{XY}={}^{(V)}\pi_{\al\be}X^{\al}Y^{\be}=g(\D_XV,Y)+g(\D_YV,X).
\end{equation*}

\begin{lemma}\label{Lemma components of the deformation tensor}
	The components of the deformation tensor under the corresponding frames and the metric $g$ in \eqref{g}
can be derived as follows
	\begin{itemize}
		\item[(1)] when $V=L$,
		\begin{equation}\label{L pi}
			\begin{split}
				&{}^{(L)}\pi_{LL}=0,\ {}^{(L)}\pi_{LT}=-L\mu,\ {}^{(L)}\pi_{TT}=2c^{-1}\mu L(c^{-1}\mu),\\
				&{}^{(L)}\pis_{LA}=0,\ {}^{(L)}\pis_{TA}=c^2\ds_A(c^{-2}\mu),\ {}^{(L)}\pis_{AB}=2\chi_{AB},\\
				&{}^{(L)}\pi_{L\Lb}=-2L\mu,\ {}^{(L)}\pi_{\Lb\Lb}=4\mu L(c^{-2}\mu),\ {}^{(L)}\pis_{\Lb A}=2c^2\ds_A(c^{-2}\mu).
			\end{split}
		\end{equation}
		\item[(2)] when $V=\rho L$,
		\begin{equation}\label{rho L pi}
			\begin{split}
				&{}^{(\rho L)}\pi_{LL}=0,\ {}^{(\rho L)}\pi_{\Lb\Lb}=4\rho\mu L(c^{-2}\mu)-4c^{-2}\mu^2+8\mu,\ {}^{(\rho L)}\pi_{L\Lb}=-2\rho L\mu-2\mu,\\
				&{}^{(\rho L)}\pis_{LA}=0,\ {}^{(\rho L)}\pis_{\Lb A}=2\rho c^2\ds_A(c^{-2}\mu),\ {}^{(\rho L)}\pis_{AB}=2\rho\chi_{AB}.
			\end{split}
		\end{equation}
		\item[(3)] when $V=\Lb$,
		\begin{equation}\label{Lb pi}
			\begin{split}
				&{}^{(\Lb)}\pi_{LL}=0,\ {}^{(\Lb)}\pi_{\Lb\Lb}=0,\ {}^{(\Lb)}\pi_{L\Lb}=-2\Lb\mu-2\mu L(c^{-2}\mu),\\
				&{}^{(\Lb)}\pis_{LA}=-2c^2\ds_A(c^{-2}\mu),\ {}^{(\Lb)}\pis_{\Lb A}=-2\mu\ds_A(c^{-2}\mu),\ {}^{(\Lb)}\pis_{AB}=-2c^{-2}\mu\chi_{AB}.
			\end{split}
		\end{equation}
		\item[(4)] when $V=T$,
		\begin{equation}\label{T pi}
			\begin{split}
				&{}^{(T)}\pi_{LL}=0,\ {}^{(T)}\pi_{LT}=-T\mu,\ {}^{(T)}\pi_{TT}=T(c^{-2}\mu^2),\\
				&{}^{(T)}\pis_{LA}=-c^2\ds_A(c^{-2}\mu),\ {}^{(T)}\pis_{TA}=0,\ {}^{(T)}\pis_{AB}=-2c^{-2}\mu\chi_{AB},\\
				&{}^{(T)}\pi_{L\Lb}=-2T\mu,\ {}^{(T)}\pi_{\Lb\Lb}=4\mu T(c^{-2}\mu),\ {}^{(T)}\pis_{\Lb A}=-\mu\ds_A(c^{-2}\mu).
			\end{split}
		\end{equation}
		\item[(5)] when $V=R_i$,
		\begin{equation}\label{Ri pi}
			\begin{split}
				&{}^{(R_i)}\pi_{LL}=0,\ {}^{(R_i)}\pi_{LT}=-R_i\mu,\ {}^{(R_i)}\pi_{TT}=2c^{-1}\mu R_i(c^{-1}\mu),\\
				&{}^{(R_i)}\pis_{LA}=-R_i^B\check{\chi}_{AB}+\ep_{ijk}\check{L}^j\ds_Ax^k-v_i\ds_Ac,\\
				&{}^{(R_i)}\pis_{TA}=c^{-2}\mu R_i^B\check{\chi}_{AB}-\rho^{-1}c^{-2}(c-1)\mu\gs_{AB}R_i^B+c^{-1}\mu\ep_{ijk}\check{T}^j\ds_Ax^k+v_i\ds_A(c^{-1}\mu),\\
				&{}^{(R_i)}\pis_{AB}=2c^{-1}v_i\chi_{AB},\ {}^{(R_i)}\pi_{L\Lb}=-2R_i\mu,\ {}^{(R_i)}\pi_{\Lb\Lb}=4\mu R_i(c^{-2}\mu),\\
				&{}^{(R_i)}\pis_{\Lb A}=c^{-2}\mu R_i^B\check{\chi}_{AB}-2\rho^{-1}c^{-2}(c-1)\mu\gs_{AB}R_i^B-3c^{-2}\mu v_i\ds_Ac+2c^{-1}v_i\ds_A\mu\\
				&\quad\quad\quad\quad\ +c^{-2}\mu\ep_{ijk}\check{L}^j\ds_Ax^k+2c^{-1}\mu\ep_{ijk}\check{T}^j\ds_Ax^k.
			\end{split}
		\end{equation}
	\end{itemize}
\end{lemma}

\vskip 0.1 true cm

\begin{lemma}\label{Lemma properties of Lie derivatives}
For any tensor field $\xi$, according to Lemmas 7.5, 8.6 and Corollary 4.13 in \cite{Sp}, then
\begin{itemize}
	\item[(1)] $X[T(Y_1,\cdots,Y_p)]=(\Lie_XT)(Y_1,\cdots,Y_p)+\sum_{i=1}^pT(Y_1,\cdots,\Lie_XY_i,\cdots,Y_p)$.
	
	\item[(2)] $\Lies_Z\gs_{AB}={}^{(Z)}\pis_{AB},\ \Lies_Z\gs^{AB}=-{}^{(Z)}\pis^{AB},\ Z\in\{L,\rho L,T,\Lb,R_1,R_2,R_3\}$.
	
	\item[(3)] $\Lies_Z\ds f=\ds Zf,\ Z\in\{L,\rho L,T,R_1,R_2,R_3\}$, $\Lies_{\Lb}\ds_Af=\ds_A\Lb f+[\Lb,X_A]f=\ds_A\Lb f-\ds_A(c^{-2}\mu)Lf$.
	
	\item[(4)] $[\Lies_X,\Lies_Y]=\Lies_{[X,Y]}$.
\end{itemize}
\end{lemma}

\begin{lemma}\label{Lemma commutators}
According to Lemmas 8.9, 8.11 in \cite{Sp}, then

(1) For any vector field $Z\in\{L,\rho L,T,R_1,R_2,R_3\}$ and symmetric $(0,2)$-type tensor field $\xi$ on $S_{\t,u}$,
\begin{equation}\label{[nas,Lies]}
	([\nas_A,\Lies_Z]\xi)_{BC}=(\check{\nas}_A{}^{(Z)}\pis_B^D)\xi_{CD}+(\check{\nas}_A{}^{(Z)}\pis_C^D)\xi_{BD},
\end{equation}
where
\begin{equation*}
	\check{\nas}_A{}^{(Z)}\pis_{BC}=\frac{1}{2}(\nas_A{}^{(Z)}\pis_{BC}+\nas_B{}^{(Z)}\pis_{AC}-\nas_C{}^{(Z)}\pis_{AB}).
\end{equation*}

(2) For any vector field $Z\in\{L,\rho L,T,R_1,R_2,R_3\}$ and smooth function $f$,
\begin{equation}\label{[nas^2,Lies]}
	\big([\nas^2,\Lies_Z]f\big)_{AB}=(\check{\nas}_A{}^{(Z)}\pis_B^C)\ds cf
\end{equation}
and
\begin{equation}\label{[Des,Z]}
	[\Des,Z]f={}^{(Z)}\pis^{AB}\nas_{AB}^2f+\check\nas_A{}^{(Z)}\pis^{AB}\ds_Bf,
\end{equation}
in particular,
\begin{equation}\label{particular [Des,Z]}
	\begin{split}
		[L,\Des]f&=-2\chic^{AB}\nas_{AB}^2f-2\rho^{-1}\Des f-2\check{\nas}_A\chic^{AB}\ds_Bf,\\
		[T,\Des]f&=-2c^{-2}\mu\chic^{AB}\nas_{AB}^2f-2\rho^{-1}c^{-2}\mu\Des f-2\check{\nas}_A(c^{-2}\mu\chi^{AB})\ds_Bf.
	\end{split}
\end{equation}
	
(3) The following commutator relations hold from Lemma \ref{Lemma connection coefficients} that
\begin{equation}\label{[X,Y]}
    \begin{split}
		&[L,R_i]={}^{(R_i)}\pis_L,\\
		&[L,T]={}^{(T)}\pis_L,\\
		&[T,R_i]={}^{(R_i)}\pis_T,\\
		&[L,\Lb]=L(c^{-2}\mu)L-2(\ze^A+\et^A)X_A.
	\end{split}
\end{equation}
\end{lemma}

\begin{lemma}\label{Lemma relation of equivalence}
As in \cite{M-Y}, for any smooth function $f$, one has
\begin{equation*}
	|R_if| \sim |\ds f| \sim |\nas f|.
\end{equation*}
Moreover, for any $S_{\t,u}$-tangential one-from  $\xi$ (or trace-free symmetric $S_{\t,u}$-tangential $(0,2)$-type tensor field),
we have
\begin{equation*}
	|\Lies_{R_i}\xi| \sim |\xi|+|\nas\xi|.
\end{equation*}
\end{lemma}

\begin{remark}\label{Remark sqrt det gs is bounded above and below}
As in \cite{M-Y}, the formula
	\begin{equation*}
		L\big(\ln\sqrt{\det\gs}\big)=\trgs\chi	
	\end{equation*}
illustrates that $\sqrt{\det\gs}$ is bounded and never vanishes along each null generator when
the bounded estimate of $\chi$ is established.
\end{remark}

\vskip 0.1 true cm

We now look for the equation of $\vp=(\vp_0,\vp_1,\vp_2,\vp_3)=(\p_t\phi,\p_1\phi,\p_2\phi,\p_3\phi)$ under
the action of the covariant wave operator $\Box_g=g^{\al\be}\D_{\al\be}^2$ with the help of metrics and Christoffel symbols.
Taking the derivative on two sides of \eqref{quasilinear wave equation} with respect to the variable $x^{\ga}$ derives
\begin{equation*}
	-c^{-2}\p_t^2\vp_{\ga}+\De\vp_{\ga}=p\vp_0^{p-1}\p_t\vp_0\p_t\vp_{\ga}.
\end{equation*}
Then it follows from direct computation that
\begin{equation}\label{covariant wave equations}
	\mu\Box_g\vp_\ga=F_\ga,
\end{equation}
where
\begin{equation}\label{F_ga}
	F_\ga=\frac{1}{2}p\mu\vp_0^{p-1}L\vp_0L\vp_{\ga}+\frac{1}{2}pc^2\vp_0^{p-1}L\vp_0T\vp_{\ga}+\frac{1}{2}pc^2\vp_0^{p-1}T\vp_0L\vp_{\ga}-\frac{1}{2}pc^2\mu\vp_0^{p-1}\ds_A\vp_0\ds^A\vp_{\ga}.
\end{equation}
On the other hand, due to
\begin{equation*}
	\mu\Box_g\vp_\ga=-(L+\frac{1}{2}\trgs\chi)\Lb\vp_\ga+2c^{-1}\mu\ds_Ac\ds^A\vp_\ga+\mu\Des\vp_\ga+\frac{1}{2}c^{-2}\mu\trgs\chi L\vp_\ga,
\end{equation*}
this yields
\begin{equation}\label{transport equation of Lb vp_ga}
	(L+\frac{1}{2}\trgs\chi)\Lb\vp_\ga=\mu\Des\vp_\ga+H_\ga-F_\ga,
\end{equation}
where
\begin{equation}\label{H_ga}
	H_\ga=\frac{1}{2}c^{-2}\mu\trgs\chi L\vp_\ga+2c^{-1}\mu\ds_Ac\ds^A\vp_\ga.
\end{equation}
For convenience, define
\begin{equation}\label{Hc_ga}
	\Hc_\ga=2c^{-1}\mu\ds_Ac\ds^A\vp_\ga.
\end{equation}

At the last of this section, we give the structure equations of $\mu$, $\chi$ and $L^i$ as in \cite{DLY} and \cite{M-Y}.
\begin{lemma}\label{Lemma structure equations}
	It holds that
\begin{align}
	L\mu&=c^{-1}Lc\mu-cTc,\label{transport equation of mu}\\
	\Lies_L\chi_{AB}&=c^{-1}Lc\chi_{AB}+\chi_A^C\chi_{BC}-\check{\R}_{LALB},\label{transport equation of chi}\\
	\Lies_T\chi_{AB}&=c^{-1}Tc\chi_{AB}-c^{-2}\mu\chi_A^C\chi_{BC}+c\nas_{AB}^2(c^{-1}\mu),\label{transport equation of chi along T}\\
	(\divgs\chi)_A&=\ds_A\trgs\chi+c^{-1}\ds^Bc\chi_{AB}-c^{-1}\ds_Ac\trgs\chi,\label{elliptic equation of chi}\\
	LL^i&=c^{-1}LcL^i-c\ds_Ac\ds^Ax^i,\label{transport equation of L^i}\\
	TL^i&=c^{-1}TcL^i+c\ds_A(c^{-1}\mu)\ds^Ax^i,\label{transport equation of L^i along T}\\
	\ds_AL^i&=\chi_{AB}\ds^Bx^i.\label{structure equation of L^i}
\end{align}
\end{lemma}

\section{Bootstrap assumptions and  $L^{\infty}$ estimates of lower order derivatives}\label{Section 3}

By virtue of the initial data \eqref{initial data} and Remarks \ref{Remark equivalent outgoing constraint condition}, \ref{Remark T and L on t=1}, we make the following bootstrap assumptions in $D^{\t,u}$
\begin{equation}\label{bootstrap assumptions}
	\de^{l+s[1-(1-\ve0)p]}\|Z^{\al}\vp_{\ga}\|_{L^{\infty}(\Si_\t^u)}\leq M\de^{1-\ve0},\ \ga=0,1,2,3,
\end{equation}
where $1\leq\t\leq t^*$, $|\al|\leq N$, $N$ is a large positive integer, $M$ is some positive number which is suitably chosen (at least double bounds of the corresponding quantities on $\t=1$), $Z\in\{\rho L,T,R_1,R_2,R_3\}$, $l$ (or $s$) is the
number of $T$ (or $\rho L$) included in $Z^{\al}$ with $s\leq 2$.

In terms of the definition of $c$ and the Leibnizian rule, we have from \eqref{bootstrap assumptions} that
\begin{equation}\label{L^infty estimates of c}
	\|c-1\|_{L^\infty(\Si_\t^u)}+\de^{l+s[1-(1-\ve0)p]}\|Z^{\al}c\|_{L^\infty(\Si_\t^u)}\les M^p\de^{(1-\ve0)p},
\end{equation}
where $1\leq|\al|\leq N$ and $1\leq\t\leq t^*$.

Next, we derive $L^{\infty}$ estimates of  some lower order derivatives.
\begin{proposition}\label{Proposition lower order L^infty estimates}
Under the assumptions \eqref{bootstrap assumptions}, for any vector field $Z\in\{\rho L,T,R_1,R_2,R_3\}$, when $\de>0$ is small, it holds
 that for $|\al|\leq 1$ and $1\leq\t\leq t^*$,
\begin{equation}\label{lower order L^infty estimates}
	\begin{split}
		&\de^{l+s[1-(1-\ve0)p]}\|\Lies_{Z}^{\leq\al}\chic\|_{L^\infty(\Si_\t^u)}\les M^p\de^{(1-\ve0)p},\\
		&\de^{l+s[1-(1-\ve0)p]}\|Z^{\leq\al+1}\mu\|_{L^\infty(\Si_\t^u)}\les M^p,\\
		&\de^{l+s[1-(1-\ve0)p]}\left(\|Z^{\leq\al+1}\Lc^j\|_{L^\infty(\Si_\t^u)}+\|Z^{\leq\al+1}\Tc^j\|_{L^\infty(\Si_\t^u)}+\|Z^{\leq\al+1}v_j\|_{L^\infty(\Si_\t^u)}\right)\les M^p\de^{(1-\ve0)p},\\
		&\de^{l+s[1-(1-\ve0)p]}\left(\|\Lies_{Z}^{\leq\al+1}\ds x^j\|_{L^\infty(\Si_\t^u)}+\|\Lies_{Z}^{\leq\al+1}R_j\|_{L^\infty(\Si_\t^u)}\right)\les 1,\\
		&\de^{l+s[1-(1-\ve0)p]}\left(\|\Lies_{Z}^{\leq\al}{}^{(R_j)}\pis\|_{L^\infty(\Si_\t^u)}+\|\Lies_{Z}^{\leq\al}{}^{(R_j)}\pis_L\|_{L^\infty(\Si_\t^u)}+\|\Lies_{Z}^{\leq\al}{}^{(R_j)}\pis_T\|_{L^\infty(\Si_\t^u)}\right)\les M^p\de^{(1-\ve0)p},\\
		&\de^{l+s[1-(1-\ve0)p]}\left(\|\Lies_Z^{\leq\al}{}^{(T)}\pis\|_{L^\infty(\Si_\t^u)}+\|\Lies_Z^{\leq\al}{}^{(T)}\pis_L\|_{L^\infty(\Si_\t^u)}\right)\les M^p,
	\end{split}
\end{equation}
where $l$ (or $s$) is the number of $T$ (or $\rho L$) appearing in the string of $Z$ with $s\leq 2$.
\end{proposition}
\begin{proof}
Taking the $j$-th component on both sides of $\p_j=\Tr^j\Tr+\ds^Ax^jX_A$ yields $1=|\Tr^j|^2+|\ds x^j|^2$,
which immediately gives $|\ds x^j|\les 1$ and hence $|\gs|\les 1$ holds in terms of $\gs_{AB}=g_{ij}\ds_Ax^i\ds_Bx^j$.

\vskip 0.2 true cm

\textbf{Part 1. Estimate of $\chic$}

\vskip 0.1 true cm	

To estimate $\chic$, by substituting $\chi_{AB}=\chic_{AB}+\rho^{-1}\gs_{AB}$ into \eqref{transport equation of chi}
and in view of $\Lies_L\gs_{AB}=2\chi_{AB}$, one has
\begin{equation}\label{transport equation of chic}
	\Lies_L\chic_{AB}=c^{-1}Lc\chic_{AB}+\chic_A^C\chic_{BC}+\rho^{-1}c^{-1}Lc\gs_{AB}-\Rc_{LALB}
\end{equation}
and hence,
\begin{equation}\label{weighted transport equation of chic}
	L(\rho^4|\chic|^2)=2\rho^4\big(c^{-1}Lc|\chic|^2-\chic_A^B\chic_B^C\chic_C^A+\rho^{-1}c^{-1}Lc\trgs\chic-\chic^{AB}\Rc_{LALB}\big).
\end{equation}
Then by \eqref{Rc_LALB} and \eqref{L^infty estimates of c}, we obtain
\begin{equation}\label{weighted transport inequality of chic}
	\big|L(\rho^2|\chic|)\big|\les M^p\de^{(1-\ve0)p-[1-(1-\ve0)p]}\rho^2|\chic|+\rho^2|\chic|^2+M^p\de^{(1-\ve0)p-[1-(1-\ve0)p]}.
\end{equation}
Due to
\begin{equation}\label{chic on t=1}
	\chic_{AB}|_{\t=1}=(c-1)\f1r\gs_{AB}|_{\t=1}\les M^p\de^{(1-\ve0)p},
\end{equation}
then integrating \eqref{weighted transport inequality of chic} along integral curves of $L$ from $1$ to $\t\leq t^*$ yields
\begin{equation*}
	|\chic|\les M^p\de^{(1-\ve0)p}.
\end{equation*}

\vskip 0.2 true cm

\textbf{Part 2. Estimates of $\mu$ and $Z\mu$}

\vskip 0.1 true cm	

To estimate $\mu$, integrating the transport equation \eqref{transport equation of mu} along integral curves of $L$ from $1$ to $\t$ yields
\begin{equation}\label{explicit formula of mu}
	\mu(\t)=\mathrm{e}^{\int_1^\t c^{-1}Lc(\t')\d\t'}\mu(1)-\int_1^\t\mathrm{e}^{\int_{\t'}^\t c^{-1}Lc(\t'')\d\t''}cTc(\t')\d\t'.
\end{equation}
Since $\mu=c$ on $\t=1$, together with \eqref{L^infty estimates of c} and $\t\leq t^*$, we have
\begin{equation*}
	|\mu|\les M^p,
\end{equation*}
which also implies the analogous estimate of $L\mu$ in terms of \eqref{transport equation of mu}.

To estimate $R_i\mu$ and $T\mu$, we will act $\ds$ and $T$ to the transport equation \eqref{transport equation of mu} respectively.
For $R_i\mu$, one has
\begin{equation}\label{transport equation of ds mu}
	\Lies_L\ds\mu=\ds L\mu=c^{-1}Lc\ds\mu+\ds(c^{-1}Lc)\mu-\ds(cTc).
\end{equation}
By \eqref{L^infty estimates of c} and the estimate of $\mu$, we have
\begin{equation}\label{transport inequality of ds mu}
	|\Lies_L\ds\mu|\les M^p\de^{(1-\ve0)p-[1-(1-\ve0)p]}|\ds\mu|+M^p\de^{(1-\ve0)p-1}.
\end{equation}
Then integrating \eqref{transport inequality of ds mu} along integral curves of $L$ from $1$ to $\t$ yields
\begin{equation*}
	|\ds\mu|\les M^p
\end{equation*}
and hence $|R_i\mu|\les M^p$. For $T\mu$, one has from \eqref{[X,Y]} that
\begin{equation}\label{transport equation of T mu}
	LT\mu=TL\mu+[L,T]\mu=c^{-1}LcT\mu+\big[T(c^{-1}Lc)\mu-T(cTc)-c^2\ds^A(c^{-2}\mu)\ds_A\mu\big].
\end{equation}
By \eqref{L^infty estimates of c} and the estimate of $\ds\mu$, we have
\begin{equation}\label{transport inequality of T mu}
	|LT\mu|\les M^p\de^{(1-\ve0)p-[1-(1-\ve0)p]}|T\mu|+M^p\de^{(1-\ve0)p-2}.
\end{equation}
Then integrating \eqref{transport inequality of T mu} along integral curves of $L$ from $1$ to $\t$ yields
\begin{equation*}
	|T\mu|\les M^p\de^{-1}.
\end{equation*}

\vskip 0.2 true cm

\textbf{Part 3. Estimates of $\Lc^j$ and $Z\Lc^j$}

\vskip 0.1 true cm	

To estimate $\Lc^j$, by substituting $L^j=\Lc^j+\rho^{-1}x^j$ into \eqref{transport equation of L^i}, one has
\begin{equation}\label{transport equation of Lc^j}
    L\Lc^j=(c^{-1}Lc-\rho^{-1})\Lc^j+\rho^{-1}c^{-1}Lcx^j-c\ds_Ac\ds^Ax^j
\end{equation}
and hence,
\begin{equation}\label{weighted transport equation of Lc^j}
	L(\rho\Lc^j)=\rho c^{-1}Lc\Lc^j+c^{-1}Lcx^j-\rho c\ds_Ac\ds^Ax^j.
\end{equation}
Together with \eqref{L^infty estimates of c}, this yelds
\begin{equation}\label{weighted transport inequality of Lc^j}
	\big|L(\rho\Lc^j)\big|\les M^p\de^{(1-\ve0)p-[1-(1-\ve0)p]}|\rho\Lc^j|+M^p\de^{(1-\ve0)p-[1-(1-\ve0)p]}.	
\end{equation}
Then integrating \eqref{weighted transport inequality of Lc^j} along integral curves of $L$ from $1$ to $\t$ yields
\begin{equation*}
	|\Lc^j|\les M^p\de^{(1-\ve0)p},
\end{equation*}
which also implies the analogous estimate of $L\Lc^j$ in terms of \eqref{transport equation of Lc^j}.

To estimate $R_i\Lc^j$ and $T\Lc^j$, by substituting $L^j=\Lc^j+\rho^{-1}x^j$ into \eqref{structure equation of L^i}
and \eqref{transport equation of L^i along T}, we arrive at
\begin{align}
	\ds_A\Lc^j&=\chic_{AB}\ds^Bx^j,\label{structure equation of Lc^j}\\
	T\Lc^j&=(c^{-1}Tc+\rho^{-1}c^{-2}\mu)\Lc^j+(\rho^{-1}c^{-1}Tc+\rho^{-2}c^{-2}\mu-\rho^{-2})x^j+c\ds_A(c^{-1}\mu)\ds^Ax^j,\label{transport equation of Lc^j along T}	
\end{align}
which immediately give the estimates of $R_i\Lc^j$ and $T\Lc^j$.

\vskip 0.2 true cm

\textbf{Part 4. Estimates of the other quantities}

\vskip 0.1 true cm	

The estimates of $\Tc^j$, $v_j$, $Z\Tc^j$ and $Zv_j$ follow immediately from the first identity in \eqref{YHC-1} and \eqref{v_i}.
Then the estimates of $\Lies_Z\ds x^j$ can be obtained in view of $\Lies_Z\ds x^j=\ds Z^j$ and \eqref{R_i}.
In addition, $R_j$ and $\Lies_ZR_j$ are bounded due to $R_j^A:=g(R_j,X^A)=\sum_{k=1}^{3}R_j^k\ds^Ax^k$ and \eqref{R_i}.

On the other hand, the first order derivatives of $\chic$ and the second order derivatives of $\mu$, $\Lc^j$ are bounded
by acting the related derivatives on \eqref{transport equation of chic}, \eqref{transport equation of mu} and \eqref{transport equation of Lc^j}
as in Part 1,2,3. Analogously, the second order derivatives of $\Tc^j$, $v_j$, $\ds x^j$ and $R_j$
can be bounded.

Finally, the estimates of the deformation tensors and their first order derivatives are bounded by \eqref{Ri pi} and \eqref{T pi},
since the corresponding quantities have been estimated in the above.
\end{proof}

We now close bootstrap assumptions \eqref{bootstrap assumptions} for $\vp_\ga$ and $Z\vp_\ga$.

\begin{proposition}\label{Proposition close bootstrap assumptions for vp_ga and Z vp_ga}
For sufficiently small $\de>0$, it holds that
\begin{equation}\label{close bootstrap assumptions for vp_ga and Z vp_ga}
	\|\vp_\ga\|_{L^\infty(\Si_\t^u)}+\|R_i\vp_\ga\|_{L^\infty(\Si_\t^u)}
+\de^{1-(1-\ve0)p}\|L\vp_\ga\|_{L^\infty(\Si_\t^u)}+\de\|T\vp_\ga\|_{L^\infty(\Si_\t^u)}\les\de^{1-\ve0}.
\end{equation}
\end{proposition}
\begin{proof}
\vskip 0.2 true cm

\textbf{Part 1. Estimates of $T\vp_\ga$ and $\vp_\ga$}

\vskip 0.1 true cm	

Recall \eqref{F_ga} and \eqref{H_ga}, it follows from \eqref{bootstrap assumptions} that
\begin{equation*}
	|F_\ga|\les M^{p+1}\de^{-\ve0+(1-\ve0)p-[1-(1-\ve0)p]},\ |H_\ga|\les M^{p+1}\de^{-\ve0+(1-\ve0)p}.
\end{equation*}
Then by \eqref{transport equation of Lb vp_ga}, one has
\begin{equation*}
	|L\Lb\vp_\ga+\f12\trgs\chi\Lb\vp_\ga|\les M^{p+1}\de^{-\ve0+(1-\ve0)p-[1-(1-\ve0)p]}.
\end{equation*}
Therefore, by virtue of \eqref{YHC-1} and the estimate of $\chic$ in Proposition \ref{Proposition lower order L^infty estimates}, we obtain
\begin{equation}\label{transport inequality of rho Lb vp_ga}
	|L(\rho\Lb\vp_\ga)|\les M^{p+1}\de^{-\ve0+(1-\ve0)p-[1-(1-\ve0)p]}.
\end{equation}
Integrating \eqref{transport inequality of rho Lb vp_ga} along integral curves of $L$ from $1$ to $\t$ yields
that for small $\de>0$,
\begin{equation*}
    |\Lb\vp_\ga|\les M^{p+1}\de^{-\ve0+(1-\ve0)p}\les\de^{-\ve0}
\end{equation*}	
and hence,
\begin{equation*}
	|T\vp_\ga|\les\de^{-\ve0}.
\end{equation*}	
Then integrating along integral curves of $T$ from $0$ to $u$ yields
\begin{equation*}
	|\vp_\ga|\les\de^{1-\ve0}.
\end{equation*}

\vskip 0.2 true cm

\textbf{Part 2. Estimate of $L\vp_\ga$}

\vskip 0.1 true cm	

To estimate $L\vp_\ga$, by \eqref{transport equation of Lb vp_ga} and \eqref{[X,Y]}, one has
\begin{equation*}
	\begin{split}
		\Lb L\vp_\ga
		&=L\Lb\vp_\ga-[L,\Lb]\vp_\ga\\
		&=-\f12\trgs\chi\Lb\vp_\ga+\mu\Des\vp_\ga+H_\ga-F_\ga-L(c^{-2}\mu)L\vp_\ga+2c^2\ds^A(c^{-2}\mu)\ds_A\vp_\ga.
	\end{split}
\end{equation*}
Then by Proposition \ref{Proposition lower order L^infty estimates}, we have
\begin{equation*}
	|\Lb L\vp_\ga|\les\de^{-\ve0}+M^{p+1}\de^{-\ve0+(1-\ve0)p-[1-(1-\ve0)p]}
\end{equation*}	
and hence,
\begin{equation*}
	|TL\vp_\ga|\les\de^{-\ve0}+M^{p+1}\de^{-\ve0+(1-\ve0)p-[1-(1-\ve0)p]}.
\end{equation*}	
Integrating along integral curves of $T$ from $0$ to $u$ yields that for small $\de>0$,
\begin{equation*}
	|L\vp_\ga|\les\de^{1-\ve0}+M^{p+1}\de^{1-\ve0-[1-(1-\ve0)p]+(1-\ve0)p}\les\de^{1-\ve0-[1-(1-\ve0)p]}.
\end{equation*}

\vskip 0.2 true cm

\textbf{Part 3. Estimate of $R_i\vp_\ga$}

\vskip 0.1 true cm	

To estimate $R_i\vp_\ga$, by acting $R_i$ on \eqref{transport equation of Lb vp_ga}, one has
\begin{equation*}
	\begin{split}
		L\Lb R_i\vp_\ga
		&=R_iL\Lb\vp_\ga+[L\Lb,R_i]\vp_\ga\\
		&=-\f12\trgs\chi\Lb R_i\vp_\ga+\f12\trgs\chi[\Lb,R_i]\vp_\ga-\f12R_i\trgs\chi\Lb\vp_\ga+R_i(\mu\Des\vp_\ga+H_\ga-F_\ga)+[L\Lb,R_i]\vp_\ga.
	\end{split}
\end{equation*}
Then by \eqref{[X,Y]} and Proposition \ref{Proposition lower order L^infty estimates}, we have
\begin{equation*}
	|L\Lb R_i\vp_\ga|\les|\Lb R_i\vp_\ga|+M^{p+1}\de^{-\ve0+(1-\ve0)p-[1-(1-\ve0)p]}.
\end{equation*}	
Integrating along integral curves of $L$ from $1$ to $\t$ yields
\begin{equation*}
	|\Lb R_i\vp_\ga|\les M^{p+1}\de^{-\ve0+(1-\ve0)p}
\end{equation*}	
and hence,
\begin{equation*}
	|TR_i\vp_\ga|\les M^{p+1}\de^{-\ve0+(1-\ve0)p}\les\de^{-\ve0}.
\end{equation*}	
Then integrating along integral curves of $T$ from $0$ to $u$ yields
\begin{equation*}
	|R_i\vp_\ga|\les\de^{1-\ve0}.
\end{equation*}
\end{proof}

At the end of this section, we give the key estimates of $\mu$ and its derivatives, which
will imply the basic behavior of $\mu$ especially near the blow-up time as in \cite{M-Y}.

\begin{proposition}\label{Proposition behavior of mu}
For sufficiently small $\de>0$, it holds that
\begin{equation}\label{key estimate of mu}
	\mu(\t,u,\vt)=1-\f{1}{p-1}\big(\f{1}{\t^{p-1}}-1\big)L\mu(1,u,\vt)+O\big(M^p\de^{(1-\ve0)p}\big).
\end{equation}
Moreover, when $\mu(\t,u,\vt)<\f{1}{10}$, then
\begin{equation}\label{key estimate of L mu}
	L\mu(\t,u,\vt)\les-\de^{-[1-(1-\ve0)p]}
\end{equation}
and
\begin{equation}\label{key estimate of T mu}
	\big(\mu^{-1}T\mu\big)_+(\t,u,\vt)\les\f{1}{\sqrt{t^*-\t}}M^{\f{p}{2}}\de^{-1+\f12[1-(1-\ve0)p]}.
\end{equation}
\end{proposition}
\begin{proof}
\vskip 0.2 true cm

\textbf{Part 1. Estimate of $\mu$}

\vskip 0.1 true cm	

According to \eqref{transport equation of mu}, \eqref{transport inequality of rho Lb vp_ga} and Newton-Leibniz formula, one has
\begin{equation}\label{close tL mu}
	\begin{split}
		&\quad\t^pL\mu(\t,u,\vt)\\
		&=\f12pc^4(\t\vp_0)^{p-1}\t T\vp_0(\t,u,\vt)+O\big(M^p\de^{(1-\ve0)p-[1-(1-\ve0)p]}\big)\\
		&=\f12p\big[1+O\big(\de^{(1-\ve0)p}\big)\big]\big[\vp_0(1,u,\vt)+O\big(M^{p+1}\de^{1-\ve0+(1-\ve0)p}\big)\big]^{p-1}\big[T\vp_0(1,u,\vt)+O\big(M^{p+1}\de^{-\ve0+(1-\ve0)p}\big)\big]\\
		&\quad+O\big(M^p\de^{(1-\ve0)p-[1-(1-\ve0)p]}\big)\\
		&=L\mu(1,u,\vt)+O\big(M^p\de^{(1-\ve0)p-[1-(1-\ve0)p]}\big).
	\end{split}
\end{equation}
Then integrating along integral curves of $L$ from $1$ to $\t$ yields
\begin{equation*}
	\begin{split}
		&\quad\mu(\t,u,\vt)-\mu(1,u,\vt)\\
		&=\int_1^\t\f{L\mu(1,u,\vt)}{\t'^p}\d\t'+\int_1^\t\f{O\big(M^p\de^{(1-\ve0)p-[1-(1-\ve0)p]}\big)}{\t'^p}\d\t'\\
		&=-\f{1}{p-1}\big(\f{1}{\t^{p-1}}-1\big)L\mu(1,u,\vt)+O\big(M^p\de^{(1-\ve0)p-[1-(1-\ve0)p]}\cdot\de^{1-(1-\ve0)p}\big)\\
		&=-\f{1}{p-1}\big(\f{1}{\t^{p-1}}-1\big)L\mu(1,u,\vt)+O\big(M^p\de^{(1-\ve0)p}\big).
	\end{split}
\end{equation*}

\vskip 0.2 true cm

\textbf{Part 2. Estimate of $L\mu$}

\vskip 0.1 true cm	

If $\mu(\t,u,\vt)<\f{1}{10}$, in view of \eqref{key estimate of mu}, we claim that
\begin{equation*}
	\f{1}{p-1}\big(\f{1}{\t^{p-1}}-1\big)L\mu(1,u,\vt)\geq\f{1}{2}.
\end{equation*}
Otherwise, for sufficiently small $\de>0$, we would have $\mu(\t,u,\vt)>1-\f{1}{2}+O\big(M^p\de^{(1-\ve0)p}\big)>\f{1}{10}$, which is a contradiction.
Therefore,
\begin{equation*}
	L\mu(1,u,\vt)\leq-\f{p-1}{2}\f{\t^{p-1}}{\t^{p-1}-1}.
\end{equation*}
Then by \eqref{close tL mu}, one has
\begin{equation*}
	\begin{split}
		\t^pL\mu(\t,u,\vt)
		&\leq -\f{p-1}{2}\f{\t^{p-1}}{\t^{p-1}-1}+O\big(M^p\de^{(1-\ve0)p-[1-(1-\ve0)p]}\big)\\
		&\leq-\f{p-1}{2}\f{1}{1+2+\cdots+2^{p-2}}\de^{-[1-(1-\ve0)p]}+O\big(M^p\de^{(1-\ve0)p-[1-(1-\ve0)p]}\big)\\
		&\les-\de^{-[1-(1-\ve0)p]}.
	\end{split}
\end{equation*}

\vskip 0.2 true cm

\textbf{Part 3. Estimate of $T\mu$}

\vskip 0.1 true cm	

By the methods in \cite{M-Y}, we start to estimate $T\mu$.
Suppose $u^*$ is a maximum point of $T(\ln\mu)(u)$ on $[0,\de]$, then $T^2(\ln\mu)(u^*)=0$. Therefore, at the point $(\t,u^*,\vt)$, we have $\mu^{-1}T^2\mu-\mu^{-2}(T\mu)^2=0$. Hence,
\begin{equation}\label{preliminary estimate of mu^-1 T mu}
	\big(\mu^{-1}T\mu\big)_+(\t,u,\vt)\leq\sqrt{\f{\|T^2\mu\|_{L^\infty(\Si_\t^u)}}{\dse\inf_{u\in[0,\de]}\mu(u)}}.
\end{equation}
For $T^2\mu$, by Proposition \ref{Proposition lower order L^infty estimates}, one has
\begin{equation*}
	\|T^2\mu\|_{L^\infty(\Si_\t^u)}\les M^p\de^{-2}.
\end{equation*}
For $\inf\mu$, by \eqref{key estimate of L mu}, integrating along integral curves of $L$ from $\t$ to $t^*$ yields
\begin{align*}
	\mu(\t,u,\vt)
	&=\mu(t^*,u,\vt)-\int_{\t}^{t^*}L\mu(\t',u,\vt)\d\t'\\
	&\geq-\int_{\t}^{t^*}L\mu(\t',u,\vt)\d\t'\\
	&\gtrsim(t^*-t)\de^{-[1-(1-\ve0)p]}.
\end{align*}
Together with \eqref{preliminary estimate of mu^-1 T mu}, this completes the proof of Proposition \ref{Proposition behavior of mu}.
\end{proof}

\section{$L^\infty$ estimates of higher order derivatives}\label{Section 4}

Since our aim is to close the bootstrap assumptions \eqref{bootstrap assumptions}, the results obtained in
Section \ref{Section 3} are far from enough. For this purpose, we require to derive the uniform  $L^{\infty}$
estimates of higher order derivatives.

\begin{proposition}\label{Proposition higher order L^infty estimates}
Under the assumptions \eqref{bootstrap assumptions}, for any vector field $Z\in\{\rho L,T,R_1,R_2,R_3\}$, when $\de>0$ is small,
it holds that for $|\al|\leq N-2$,
\begin{equation}\label{higher order L^infty estimates}
	\begin{split}
		&\de^{l+s[1-(1-\ve0)p]}\|\Lies_{Z}^{\leq\al}\chic\|_{L^\infty(\Si_\t^u)}\les M^p\de^{(1-\ve0)p},\\
		&\de^{l+s[1-(1-\ve0)p]}\|Z^{\leq\al+1}\mu\|_{L^\infty(\Si_\t^u)}\les M^p,\\
		&\de^{l+s[1-(1-\ve0)p]}\left(\|Z^{\leq\al+1}\Lc^j\|_{L^\infty(\Si_\t^u)}+\|Z^{\leq\al+1}\Tc^j\|_{L^\infty(\Si_\t^u)}+\|Z^{\leq\al+1}v_j\|_{L^\infty(\Si_\t^u)}\right)\les M^p\de^{(1-\ve0)p},\\
		&\de^{l+s[1-(1-\ve0)p]}\left(\|\Lies_{Z}^{\leq\al+1}\ds x^j\|_{L^\infty(\Si_\t^u)}+\|\Lies_{Z}^{\leq\al+1}R_j\|_{L^\infty(\Si_\t^u)}\right)\les 1,\\
		&\de^{l+s[1-(1-\ve0)p]}\left(\|\Lies_{Z}^{\leq\al}{}^{(R_j)}\pis\|_{L^\infty(\Si_\t^u)}+\|\Lies_{Z}^{\leq\al}{}^{(R_j)}\pis_L\|_{L^\infty(\Si_\t^u)}+\|\Lies_{Z}^{\leq\al}{}^{(R_j)}\pis_T\|_{L^\infty(\Si_\t^u)}\right)\les M^p\de^{(1-\ve0)p},\\
		&\de^{l+s[1-(1-\ve0)p]}\left(\|\Lies_Z^{\leq\al}{}^{(T)}\pis\|_{L^\infty(\Si_\t^u)}+\|\Lies_Z^{\leq\al}{}^{(T)}\pis_L\|_{L^\infty(\Si_\t^u)}\right)\les M^p,
	\end{split}
\end{equation}
where $l$ (or $s$) is the number of $T$ (or $\rho L$) appearing in the string of $Z$ with $s\leq2$.
\end{proposition}
\begin{proof}
We will prove this proposition by the induction method with respect to the index $\al$. Note that we have proved \eqref{higher order L^infty estimates} for the case $|\al|\leq 1$ in Section \ref{Section 3}. For any $2\leq|\al|\leq N-2$, assume that \eqref{higher order L^infty estimates} holds up to the order $|\al|-1$, one needs to show that \eqref{higher order L^infty estimates} is also true for the order $|\al|$.
	
We first prove \eqref{higher order L^infty estimates} in the cases of $Z\in\{R_1,R_2,R_3\}$.
	
\vskip 0.2 true cm

\textbf{Part 1. Estimates of $\Lies_{R_i}^{\al}\chic$ and $\Lies_{R_i}^{\al}{}^{(R_j)}\pis_L$}

\vskip 0.1 true cm	
	
By \eqref{[nas,Lies]}, we commute $\Lies_{R_i}^{\al}$ with \eqref{transport equation of chic} to obtain
\begin{equation}\label{commute Lies Ri al with L chic}
	\begin{split}
		\Lies_L\Lies_{R_i}^{\al}\chic_{AB}
		&=\Lies_{R_i}^{\al}(\Lies_L\chic_{AB})+\sum_{\be_1+\be_2=\al-1}\Lies_{R_i}^{\be_1}\Lies_{[L,R_i]}\Lies_{R_i}^{\be_2}\chic_{AB}\\
		&=\Lies_{R_i}^{\al}(c^{-1}Lc\chic_{AB}+\chic_A^C\chic_{BC}+\rho^{-1}c^{-1}Lc\gs_{AB}-\Rc_{LALB})\\
		&\quad+\sum_{\be_1+\be_2=\al-1}\Lies_{R_i}^{\be_1}\big[{}^{(R_i)}\pis_L^C(\nas_C\Lies_{R_i}^{\be_2}\chic_{AB})+\Lies_{R_i}^{\be_2}\chic_{BC}(\nas_A{}^{(R_i)}\pis_L^C)+\Lies_{R_i}^{\be_2}\chic_{AC}(\nas_B{}^{(R_i)}\pis_L^C)\big].
	\end{split}
\end{equation}
Recalling the expression of ${}^{(R_i)}\pis_{LA}$ in \eqref{Ri pi} and utilizing \eqref{L^infty estimates of c}, together with the induction hypothesis, yield
\begin{equation}\label{estimate of commute Lies Ri al with L chic}
	\begin{split}
		|\Lies_L\Lies_{R_i}^{\al}\chic|
		&\les M^p\de^{(1-\ve0)p-[1-(1-\ve0)p]}|\Lies_{R_i}^{\al}\chic|+M^p\de^{(1-\ve0)p-[1-(1-\ve0)p]}.
	\end{split}
\end{equation}
As in \eqref{weighted transport equation of chic},
\begin{equation}\label{weighted transport equation of |Lies Ri al chic|}
	L(\rho^4|\Lies_{R_i}^{\al}\chic|^2)=-4\rho^4\chic_A^B\cdot\Lies_{R_i}^{\al}\chic_B^C\cdot\Lies_{R_i}^{\al}\chic_C^A+2\rho^4\Lies_{R_i}^{\al}\chic^{AB}\cdot\Lies_L\Lies_{R_i}^{\al}\chic_{AB}.
\end{equation}
Combining \eqref{estimate of commute Lies Ri al with L chic} and \eqref{weighted transport equation of |Lies Ri al chic|} derives
\begin{equation*}
	|L(\rho^2|\Lies_{R_i}^{\al}\chic|)|\les M^p\de^{(1-\ve0)p-[1-(1-\ve0)p]}\cdot\rho^2|\Lies_{R_i}^{\al}\chic|+M^p\de^{(1-\ve0)p-[1-(1-\ve0)p]}.
\end{equation*}
It follows from Gronwall inequality that
\begin{equation*}
	|\Lies_{R_i}^{\al}\chic|\les M^p\de^{(1-\ve0)p},
\end{equation*}
which also gives the analogous estimate of $\Lies_{R_i}^{\al}{}^{(R_j)}\pis_L$.

\vskip 0.2 true cm
	
\textbf{Part 2. Estimates of $R_i^{\al+1}\mu$ and $\Lies_{R_i}^{\al}{}^{(R_j)}\pis_T$}
	
\vskip 0.1 true cm	
	
Similarly to the treatment for $\Lies_{R_i}^\al\chic$, we commute $R_i^{\al+1}$ with \eqref{transport equation of mu} to obtain
\begin{equation}\label{commute Ri al+1 with L mu}
	\begin{split}
		LR_i^{\al+1}\mu
		&=R_i^{\al+1}L\mu+\sum_{\be_1+\be_2=\al}R_i^{\be_1}[L,R_i]R_i^{\be_2}\mu\\
		&=R_i^{\al+1}(c^{-1}Lc\mu-cTc)+\sum_{\be_1+\be_2=\al}\Lies_{R_i}^{\be_1}({}^{(R_i)}\pis_L^A\ds_AR_i^{\be_2}\mu).
	\end{split}
\end{equation}
By \eqref{L^infty estimates of c} and the estimate of $\Lies_{R_i}^{\al}{}^{(R_j)}\pis_L$ in Part 1, together with the induction hypothesis, one has
\begin{equation*}
	\begin{split}
		|LR_i^{\al+1}\mu|\les M^p\de^{(1-\ve0)p-[1-(1-\ve0)p]}|R_i^{\al+1}\mu|+M^p\de^{(1-\ve0)p-1}.
	\end{split}
\end{equation*}
Together with Gronwall inequality, this yields
\begin{equation*}
	|R_i^{\al+1}\mu|\les M^p,
\end{equation*}
which also gives the corresponding estimate of $\Lies_{R_i}^{\al}{}^{(R_j)}\pis_T$.

\vskip 0.2 true cm
	
\textbf{Part 3. Estimates of $R_i^{\al+1}\Lc^j$, $R_i^{\al+1}\Tc^j$ and $R_i^{\al+1}v_j$}
	
\vskip 0.1 true cm	
	
By \eqref{structure equation of Lc^j}, one has
\begin{equation*}
	\begin{split}
		R_i^{\al+1}\Lc^j
		&=\Lies_{R_i}^{\al}(R_i^A\ds_A\Lc^j)\\
		&=\Lies_{R_i}^{\al}(R_i^A\chic_{AB}\ds^Bx^j).
	\end{split}
\end{equation*}
Thanks to the estimate of $\Lies_{R_i}^{\al}\chic$ in Part 1, together with the induction hypothesis, we arrive at
\begin{equation*}
	|R_i^{\al+1}\Lc^j|\les M^p\de^{(1-\ve0)p}.
\end{equation*}
Then the estimates of $R_i^{\al+1}\Tc^j$ and $R_i^{\al+1}v_j$ follow immediately from
\eqref{YHC-1}  and \eqref{v_i}.

\vskip 0.2 true cm
	
\textbf{Part 4. Estimates of the other quantities}
	
\vskip 0.1 true cm	
	
The estimates of $\Lies_{R_i}^{\al+1}\ds x^j$ and $\Lies_{R_i}^{\al+1}R_j$ follow directly from \eqref{R_i} and Part 3,
while the estimates of the other deformation tensors follow from \eqref{Ri pi}, \eqref{T pi} and the above procedure.
	
In addition, it follows from \eqref{transport equation of chi along T} that the formula of $\Lies_T\chic$ can be obtained. After taking the Lie derivatives for this formula with respect to the rotation vector, we may
estimate $\Lies_{R_i}^{\al-1}\Lies_T\chic$ directly by applying the estimates above and \eqref{bootstrap assumptions}.
Therefore, the bounds of $\Lies_Z^\al\chic$ are obtained when $Z\in\{T,R_1,R_2,R_3\}$ and there is only one $T$
to appear in $Z^\al$. Similarly, $Z^{\al+1}\Lc^j$ can be estimated by \eqref{transport equation of Lc^j along T}.
By use of \eqref{transport equation of mu} again, the bounds of $Z^{\al+1}\mu$ are also obtained.
On the other hand, the estimates of other left quantities in \eqref{higher order L^infty estimates} follow immediately.
Then by the induction method
with respect to the number of $T$, \eqref{higher order L^infty estimates} can be derived for $Z\in\{T,R_1,R_2,R_3\}$.
	
Finally, when the derivatives in $Z^\al$ involve $\rho L$, the transport equations \eqref{transport equation of mu}, \eqref{transport equation of chic}, \eqref{transport equation of Lc^j} can be utilized to derive \eqref{higher order L^infty estimates}.
\end{proof}

At the end of this section, we close bootstrap assumptions \eqref{bootstrap assumptions} for higher order derivatives.

\begin{proposition}\label{Proposition close bootstrap assumptions for higher order derivatives}
For sufficiently small $\de>0$, it holds  that for $|\al|\leq N-2$,
\begin{equation}\label{close bootstrap assumptions for higher order derivatives}
	\de^{l+s[1-(1-\ve0)p]}\|Z^{\al}\vp_{\ga}\|_{L^{\infty}(\Si_\t^u)}\les\de^{1-\ve0},\ \ga=0,1,2,3,
\end{equation}
where $Z\in\{\rho L,T,R_1,R_2,R_3\}$, $l$ (or $s$) is the number of $T$ (or $\rho L$) included in $Z^{\al}$ with $s\leq 2$.
\end{proposition}
\begin{proof}
We will prove this proposition by the induction method with respect to the index $\al$.
Note that we have proved \eqref{close bootstrap assumptions for higher order derivatives}
for the case $|\al|\leq 1$ in Section \ref{Section 3}. For any $2\leq|\al|\leq N-2$,
assume that \eqref{close bootstrap assumptions for higher order derivatives} holds up to the order $|\al|-1$,
one requires to show \eqref{close bootstrap assumptions for higher order derivatives} for the order $|\al|$.

We first prove \eqref{close bootstrap assumptions for higher order derivatives}
when $Z^{\alpha}$ only involves the rotation vectors, that is, $Z\in\{R_1,R_2,R_3\}$.
To estimate $R_i^{\al}\vp_\ga$, by
commuting $R_i^{\al}$ with \eqref{transport equation of Lb vp_ga}, one has
\begin{equation*}
	\begin{split}
		L\Lb R_i^{\al}\vp_\ga
		&=R_i^{\al}L\Lb\vp_\ga+[L\Lb,R_i^{\al}]\vp_\ga\\
		&=-\f12\trgs\chi\Lb R_i^{\al}\vp_\ga+\f12\trgs\chi[\Lb,R_i^{\al}]\vp_\ga-\f12R_i^{\al}\trgs\chi\Lb\vp_\ga+R_i^{\al}(\mu\Des\vp_\ga+H_\ga-F_\ga)+[L\Lb,R_i^{\al}]\vp_\ga.
	\end{split}
\end{equation*}
Together with \eqref{[X,Y]}, Proposition \ref{Proposition higher order L^infty estimates} and \eqref{bootstrap assumptions}, this yields
\begin{equation*}
	|L\Lb R_i^{\al}\vp_\ga|\les|\Lb R_i^{\al}\vp_\ga|+M^{p+1}\de^{-\ve0+(1-\ve0)p-[1-(1-\ve0)p]}.
\end{equation*}	
Integrating along integral curves of $L$ from $1$ to $\t$ derives
\begin{equation*}
	|\Lb R_i^{\al}\vp_\ga|\les M^{p+1}\de^{-\ve0+(1-\ve0)p}
\end{equation*}	
and hence,
\begin{equation*}
	|TR_i^{\al}\vp_\ga|\les M^{p+1}\de^{-\ve0+(1-\ve0)p}\les\de^{-\ve0}.
\end{equation*}	
Then we have that by integrating along integral curves of $T$ from $0$ to $u$,
\begin{equation*}
	|R_i^{\al}\vp_\ga|\les\de^{1-\ve0}.
\end{equation*}
 	
When $Z^{\al}$ involves $T$ or $\rho L$, in view of \eqref{[X,Y]} and the induction hypothesis, we only need to
estimate $\Lb Z^{\al-1}\vp_\ga$ or $Z^{\al-1}L\vp_\ga$ by commuting $Z^{\al-1}$ with \eqref{transport equation of Lb vp_ga}. This completes the proof of \eqref{close bootstrap assumptions for higher order derivatives}.
\end{proof}

\begin{remark}\label{Remark close L^infty estimates up to order N-4}
It is pointed out that we have closed the bootstrap assumptions \eqref{bootstrap assumptions} for the orders up to $|\al|\leq N-2$.
In this case, Proposition \ref{Proposition higher order L^infty estimates} holds for the orders up to $|\al|\leq N-4$,
where the constant $M$ has been removed.
\end{remark}

\section{Energy estimates for the covariant wave equation}\label{Section 5}

In order to close the estimates of the $(N-1)^{\text{th}}$ and $N^{\text{th}}$ order derivatives of $\vp_{\ga}$
in \eqref{bootstrap assumptions}, we now focus on the energy estimates for \eqref{covariant wave equations}.
As in \cite{Ch2} and \cite{M-Y}-\!\cite{Sp},
the strategy is to deal with linearized equation first and then to act the vector fields $Z^{\al+1}$
on \eqref{covariant wave equations} so that the higher order energies are derived. This procedure is
divided into the following six steps.

\vskip 0.2 true cm

\textbf{Step 1. Deriving the divergence form}

\vskip 0.1 true cm

For the linearized covariant wave equation of \eqref{covariant wave equations}, that is
\begin{equation}\label{linearized covariant wave equation}
	\mu\Box_g\Psi=\Phi,
\end{equation}
where $\Psi$ and its derivatives vanish on $C_0^\t$. By \eqref{energy-momentum tensor}, one has
\begin{equation*}
	\Box_g\Psi\cdot\p_{\be}\Psi=\D^{\al}Q_{\al\be}.
\end{equation*}
In addition, for any vector field $V$, it follows from \eqref{deformation tensor} that
\begin{equation}\label{key formula}
	\Box_g\Psi\cdot V\Psi=\D_{\al}{}^{(V)}J^{\al}-\frac{1}{2}Q^{\al\be}[\Psi]{}^{(V)}\pi_{\al\be},
\end{equation}
where ${}^{(V)}J^{\al}=Q_{\be}^{\al}V^{\be}$ with $Q_\be^\al=g^{\al\ga}Q_{\be\ga}$ and $Q^{\al\be}=g^{\al\ka}g^{\be\la}Q_{\ka\la}$.

\vskip 0.2 true cm

\textbf{Step 2. Integrating by parts on domain $D^{\t,u}$}

\vskip 0.1 true cm

Under the coordinate $\{\t,u,\vt^1,\vt^2\}$, one has
\begin{equation}\label{J under the optical coordinate}
	{}^{(V)}J={}^{(V)}J^\t\frac{\p}{\p\t}+{}^{(V)}J^u\frac{\p}{\p u}+{}^{(V)}\slashed{J}^A\frac{\p}{\p\vt^A}.
\end{equation}
Let $N=\p_t=L+c^2\mu^{-1}T$ be the normal vector. Then taking the inner product with $N$ on both
sides of \eqref{J under the optical coordinate} yields
\begin{equation*}
	{}^{(V)}J_N={}^{(V)}J^\t g(L,N)+{}^{(V)}J^ug(T,N)+{}^{(V)}\slashed{J}^Ag(X_A,N)=-c^2{}^{(V)}J^\t.
\end{equation*}
Hence, ${}^{(V)}J^\t=-c^{-2}{}^{(V)}J_N$. Similarly, ${}^{(V)}J^u=-\mu^{-1}{}^{(V)}J_L$. Therefore,
it follows from $\sqrt{|\det g|}=\mu\sqrt{\det\gs}$ that
\begin{equation*}
	\begin{split}
		\D_{\al}{}^{(V)}J^{\al}
		&=\frac{1}{\sqrt{|\det g|}}\big[\frac{\p}{\p\t}(\sqrt{|\det g|}{}^{(V)}J^\t)+\frac{\p}{\p u}(\sqrt{|\det g|}{}^{(V)}J^u)+\frac{\p}{\p\vt^A}(\sqrt{|\det g|}{}^{(V)}\slashed{J}^A)\big]\\
		&=\frac{1}{\sqrt{|\det g|}}\big[\frac{\p}{\p\t}(-c^{-2}\mu{}^{(V)}J_N\sqrt{\det\gs})+\frac{\p}{\p u}(-{}^{(V)}J_L\sqrt{\det\gs})+\frac{\p}{\p\vt^A}(\sqrt{|\det g|}{}^{(V)}\slashed{J}^A)\big].
	\end{split}
\end{equation*}
Integrating over $D^{\t,u}$ to obtain
\begin{equation}\label{divergence theorem}
	-\int_{D^{\t,u}}\mu\D_{\al}{}^{(V)}J^{\al}=\int_{\Si_\t^u}c^{-2}\mu{}^{(V)}J_N-\int_{\Si_1^u}c^{-2}\mu{}^{(V)}J_N+\int_{C_u^\t}{}^{(V)}J_L,
\end{equation}
where
\begin{equation*}
	{}^{(V)}J_N={}^{(V)}J^{\al}N_{\al}=Q_{\be}^{\al}V^{\be}N_{\al}=Q_{VN},\ {}^{(V)}J_L=Q_{VL}.
\end{equation*}

\vskip 0.2 true cm

\textbf{Step 3. Establishing energy identity}

\vskip 0.1 true cm

Choosing two vector fields $V_1=L,V_2=\Lb$ as multipliers, by $T=\frac{1}{2}(\Lb-c^{-2}\mu L)$ and \eqref{components of energy-momentum tensor}, then
\begin{equation*}
	\begin{split}
		Q_{V_1N}&=\frac{1}{2}\big[(L\Psi)^2+c^2|\ds\Psi|^2\big],\ Q_{V_1L}=(L\Psi)^2,\\
		Q_{V_2N}&=\frac{1}{2}\big[c^2\mu^{-1}(\Lb\Psi)^2+\mu|\ds\Psi|^2\big],\ Q_{V_2L}=\mu|\ds\Psi|^2.
	\end{split}
\end{equation*}
Therefore, by \eqref{divergence theorem} and \eqref{key formula}, we have the following energy identity
\begin{equation}\label{energy identity}
	E_i[\Psi](\t,u)-E_i[\Psi](1,u)+F_i[\Psi](\t,u)=-\int_{D^{\t,u}}\Phi\cdot V_i\Psi-\int_{D^{\t,u}}\frac{1}{2}\mu Q^{\al\be}[\Psi]^{(V_i)}\pi_{\al\be},\ i=1,2,
\end{equation}
where the energies $E_i[\Psi](t,u)$ and fluxes $F_i[\Psi](t,u)$ are defined as follows
\begin{equation}\label{energy and flux}
	\begin{split}
		E_1[\Psi](\t,u)
		&=\int_{\Si_\t^u}\frac{1}{2}\big[c^{-2}\mu(L\Psi)^2+\mu|\ds\Psi|^2\big],\ F_1[\Psi](\t,u)=\int_{C_u^\t}(L\Psi)^2,\\
		E_2[\Psi](\t,u)
		&=\int_{\Si_\t^u}\frac{1}{2}\big[(\Lb\Psi)^2+c^{-2}\mu^2|\ds\Psi|^2\big],\ F_2[\Psi](\t,u)=\int_{C_u^\t}\mu|\ds\Psi|^2.
	\end{split}
\end{equation}

\vskip 0.2 true cm

\textbf{Step 4. Deriving error estimates and energy inequality}

\vskip 0.1 true cm
We now deal with the second integral in the right hand side of $\eqref{energy identity}$. By Lemma \ref{Lemma componets of gs},
one has
\begin{equation*}
	\begin{split}
		&\quad-\frac{1}{2}\mu Q^{\al\be}[\Psi]{}^{(V)}\pi_{\al\be}\\
		&=-\frac{1}{2}\mu\big[-\frac{1}{2}\mu^{-1}(L^{\al}\Lb^{\ka}+\Lb^{\al}L^{\ka})+\gs^{AB}X_A^{\al}X_B^{\ka}\big]\big[-\frac{1}{2}\mu^{-1}(L^{\be}\Lb^{\la}+\Lb^{\be}L^{\la})+\gs^{CD}X_C^{\be}X_D^{\la}\big]Q_{\ka\la}{}^{(V)}\pi_{\al\be}\\
		&=-\frac{1}{8}\mu^{-1}(Q_{LL}{}^{(V)}\pi_{\Lb\Lb}+Q_{\Lb\Lb}{}^{(V)}\pi_{LL})-\frac{1}{4}\mu^{-1}Q_{L\Lb}{}^{(V)}\pi_{L\Lb}+\frac{1}{2}(Q_L^A{}^{(V)}\pis_{\Lb A}+Q_{\Lb}^A{}^{(V)}\pis_{LA})-\frac{1}{2}\mu Q^{AB}{}^{(V)}\pis_{AB}.
	\end{split}
\end{equation*}
Then by \eqref{components of energy-momentum tensor}, \eqref{L pi} and \eqref{Lb pi}, we obtain
\begin{equation}\label{Q*pi 1}
	\begin{split}
		-\f12\mu Q^{\al\be}[\Psi]{}^{(V_1)}\pi_{\al\be}
		&=-\f12L(c^{-2}\mu)(L\Psi)^2-\f12\trgs\chi L\Psi\Lb\Psi+c^2\ds_A(c^{-2}\mu)L\Psi\ds^A\Psi\\
		&\quad+\f12(L\mu+\mu\trgs\chic)|\ds\Psi|^2-\mu\chic_{AB}\ds^A\Psi\ds^B\Psi
	\end{split}
\end{equation}
and
\begin{equation}\label{Q*pi 2}
	\begin{split}
		-\f12\mu Q^{\al\be}[\Psi]{}^{(V_2)}\pi_{\al\be}
		&=\f12c^{-2}\mu\trgs\chi L\Psi\Lb\Psi-\mu\ds_A(c^{-2}\mu)L\Psi\ds^A\Psi-c^2\ds_A(c^{-2}\mu)\Lb\Psi\ds^A\Psi\\
		&\quad+\f12\big[\Lb\mu+\mu L(c^{-2}\mu)-c^{-2}\mu^2\trgs\chic\big]|\ds\Psi|^2+c^{-2}\mu^2\chic_{AB}\ds^A\Psi\ds^B\Psi.
	\end{split}
\end{equation}
Applying the results in Sections \ref{Section 3} and \ref{Section 4} to estimate all the terms in \eqref{Q*pi 1} and \eqref{Q*pi 2} yields
\begin{equation}\label{integral of Q*pi 1}
 	\begin{split}
 		&\quad\de^{1-(1-\ve0)p}\int_{D^{\t,u}}-\f12c^{-2}\mu Q^{\al\be}[\Psi]{}^{(V_1)}\pi_{\al\be}\\
 		&\les\de^{-[1-(1-\ve0)p]}\int_0^u\de^{1-(1-\ve0)p}F_1(\t,u')\d u'\\
 		&\quad+\de^{-1}\int_0^u\de^{1-(1-\ve0)p}F_1(\t,u')\d u'+\de^{1-(1-\ve0)p}\int_1^\t\de E_2(\t',u)\d\t'\\
 		&\quad+\de^{-1}\int_0^u\de^{1-(1-\ve0)p}F_1(\t,u')\d u'+\de^{1+[1-(1-\ve0)p]}K(\t,u)+\de\int_1^\t\de^{1-(1-\ve0)p}E_1(\t',u)\d\t'\\
 		&\quad-K(\t,u)+\de^{-[1-(1-\ve0)p]}\int_1^\t\de^{1-(1-\ve0)p}E_1(\t',u)\d\t'\\
 		&\quad+\de^{(1-\ve0)p}\int_1^\t\de^{1-(1-\ve0)p}E_1(\t',u)\d\t'
 	\end{split}
\end{equation}
and
\begin{equation}\label{integral of Q*pi 2}
 	\begin{split}
 		&\quad\de\int_{D^{\t,u}}-\f12c^{-2}\mu Q^{\al\be}[\Psi]{}^{(V_2)}\pi_{\al\be}\\
 		&\les\de^{-[1-(1-\ve0)p]}\int_1^\t\de^{1-(1-\ve0)p}E_1(\t',u)\d\t'+\de\int_1^\t\de E_2(\t',u)\d\t'\\
 		&\quad+\de^{1-[1-(1-\ve0)p]}\int_1^\t\de^{1-(1-\ve0)p}E_1(\t',u)\d\t'\\
 		&\quad+\de^{-[1-(1-\ve0)p]}\int_1^\t\de E_2(\t',u)\d\t'+\de^{1+[1-(1-\ve0)p]}K(\t,u)+\de\int_1^\t\de^{1-(1-\ve0)p}E_1(\t',u)\d\t'\\
 		&\quad+\int_1^\t\de^{-\f12[1-(1-\ve0)p]}\f{1}{\sqrt{t^*-\t'}}\cdot\de^{1-(1-\ve0)p}E_1(\t',u)\d\t'+\de^{-[1-(1-\ve0)p]}\int_1^\t\de^{1-(1-\ve0)p}E_1(\t',u)\d\t'\\
 		&\quad+\de^{(1-\ve0)2p}\int_1^\t\de^{1-(1-\ve0)p}E_1(\t',u)\d\t',
 	\end{split}
\end{equation}
where we have used $D^{\t,u}=\big(D^{\t,u}\cap\{\mu\geq\f{1}{10}\}\big)\cup\big(D^{\t,u}\cap\{\mu<\f{1}{10}\}\big)$,
\eqref{key estimate of L mu} and \eqref{key estimate of T mu}. In addition, the term $K(\t,u)=\int_{D^{\t,u}\cap\{\mu<\f{1}{10}\}}|\ds\Psi|^2$
in \eqref{integral of Q*pi 1}  plays a key role in controlling $\ds\Psi$ in the error terms.

Substituting \eqref{integral of Q*pi 1} and \eqref{integral of Q*pi 2} into \eqref{energy identity} and utilizing the Gronwall inequality, we obtain
\begin{equation}\label{energy inequality}
 	\begin{split}
 		&\quad\de^{1-(1-\ve0)p}E_1(\t,u)+\de^{1-(1-\ve0)p}F_1(\t,u)+\de E_2(\t,u)+\de F_2(\t,u)+K(\t,u)\\
 		&\les\de^{1-(1-\ve0)p}E_1(1,u)+\de E_2(1,u)+\de^{1-(1-\ve0)p}\int_{D^{\t,u}}|\Phi|\cdot|L\Psi|+\de\int_{D^{\t,u}}|\Phi|\cdot|\Lb\Psi|.
 	\end{split}
\end{equation}
Motivated by \cite{Ch2} and \cite{M-Y}-\!\cite{Sp}, define the higher order weighted energy and flux as follows
\begin{equation}\label{higher order energy and flux}
	\begin{split}
		\Et_{i,\leq m+1}(\t,u)&=\sum_{\ga=0}^3\sum_{|\al|\leq m}\de^{2l+2s[1-(1-\ve0)p]}E_i[Z^\al\vp_{\ga}](\t,u),\ i=1,2,\\
		\Ft_{i,\leq m+1}(\t,u)&=\sum_{\ga=0}^3\sum_{|\al|\leq m}\de^{2l+2s[1-(1-\ve0)p]}F_i[Z^\al\vp_{\ga}](\t,u),\ i=1,2,\\
		\Kt_{\leq m+1}(\t,u)&=\sum_{\ga=0}^3\sum_{|\al|\leq m}\de^{2l+2s[1-(1-\ve0)p]}K[Z^\al\vp_{\ga}](\t,u),\\
		\Et^b_{i,\leq m+1}(\t,u)&=\sup_{1\leq\t'\leq\t}\left\{\mu_{\min}^{2b_{m+1}}(\t')\Et_{i,\leq m+1}(\t',u)\right\},\ i=1,2,\\
		\Ft^b_{i,\leq m+1}(\t,u)&=\sup_{1\leq\t'\leq\t}\left\{\mu_{\min}^{2b_{m+1}}(\t')\Ft_{i,\leq m+1}(\t',u)\right\},\ i=1,2,\\
		\Kt^b_{\leq m+1}(\t,u)&=\sup_{1\leq\t'\leq\t}\left\{\mu_{\min}^{2b_{m+1}}(\t')\Kt_{\leq m+1}(\t',u)\right\},
	\end{split}
\end{equation}
where $l$ (or $s$) is the number of $T$ (or $\rho L$) included in $Z^\al$ with $s\leq 2$, $\mu_{\min}(\t')
=\dse\min_{(u,\vt)}\mu(\t',u,\vt)$ and the indices $\{b_k\}$ will be determined in Section \ref{Section 8.1} below.

\vskip 0.2 true cm

\textbf{Step 5. Obtaining commuted covariant wave equation}

\vskip 0.1 true cm

Choosing $\Psi=\Psi_{\ga}^{|\al|+1}=Z_{|\al|+1}\cdots Z_1\vp_{\ga}$ and
$\Phi=\Phi_{\ga}^{|\al|+1}=\mu\Box_g\Psi_{\ga}^{|\al|+1}$ in \eqref{energy inequality}.
By commuting vector fields $Z_k$ $(k=1,\cdots,|\al|+1)$ with \eqref{linearized covariant wave equation},
the induction argument gives
\begin{equation}\label{higher order commuted covariant wave equation}
	\begin{split}
		\Phi_{\ga}^{|\al|+1}
		&=\underbrace{\sum_{j=0}^{|\al|-1}\big(Z_{|\al|+1}+\leftidx{^{(Z_{|\al|+1})}}\la\big)\cdots\big(Z_{j+2}+\leftidx{^{(Z_{j+2})}}\la\big)\big(\mu\divg{}^{(Z_{j+1})}C_{\ga}^{j}\big)}_{\text{vanishes when $|\al|=0$}}\\
		&\quad+\mu\divg{}^{(Z_{|\al|+1})}C_{\ga}^{|\al|}+\big(Z_{|\al|+1}+\leftidx{^{(Z_{|\al|+1})}}\la\big)\cdots\big(Z_{1}+\leftidx{^{(Z_1)}}\la\big)\Phi_\ga^0,
	\end{split}
\end{equation}
where
\begin{equation*}
	\begin{split}
		\divg{}^{(Z)}C_{\ga}^{j}&=\D_{\be}\big[\big({}^{(Z)}\pi^{\be\nu}-\frac{1}{2}g^{\be\nu}\trg{}^{(Z)}\pi\big)\p_{\nu}\Psi_{\ga}^{j}\big],\\
		{}^{(Z)}\la&=\frac{1}{2}\trg{}^{(Z)}\pi-\mu^{-1}Z\mu,\\
		\Psi_{\ga}^0&=\vp_{\ga},\ \Phi_{\ga}^0=\mu\Box_g\vp_{\ga}
	\end{split}
\end{equation*}
with $\trg{}^{(Z)}\pi=g^{\al\be}{}^{(Z)}\pi_{\al\be}=-\f12\mu^{-1}{}^{(Z)}\pi_{L\underline L}+\f12\trgs{}^{(Z)}\pis$. Thus,
\begin{equation}\label{lamda}
	{}^{(\rho L)}\la=\rho\trgs\check\chi+3,\ {}^{(T)}\la=-c^{-2}\mu\trgs\chi,\ {}^{(R_i)}\la=c^{-1}v_i\trgs\chi.
\end{equation}
Under the frame $\{L,\Lb,X_1,X_2\}$, $\mu\divg{}^{(Z)}C_{\ga}^{j}$ can be written as
\begin{equation}\label{div C gamma alpha}
	\mu\divg{}^{(Z)}C_{\ga}^{j}=-\frac{1}{2}\big[\Lb+L(c^{-2}\mu)-c^{-2}\mu\trgs\chi\big]{}^{(Z)}C_{\ga,L}^{j}-\frac{1}{2}(L+\trgs\chi){}^{(Z)}C_{\ga,\Lb}^{j}+\nas^A(\mu{}^{(Z)}\slashed{C}_{\ga,A}^{j}),
\end{equation}
where
\begin{equation}\label{C gamma alpha}
\begin{split}
{}^{(Z)}C_{\ga,L}^{j}&=g({}^{(Z)}C_{\ga}^{j},L)=-\frac{1}{2}\trgs{}^{(Z)}\pis L\Psi_{\ga}^{j}+{}^{(Z)}\pis_{LA}\ds^A\Psi_{\ga}^{j},\\
{}^{(Z)}C_{\ga,\Lb}^{j}&=g({}^{(Z)}C_{\ga}^{j},\Lb)=-\frac{1}{2}\mu^{-1}{}^{(Z)}\pi_{\Lb\Lb}L\Psi_{\ga}^{j}
-\frac{1}{2}\trgs{}^{(Z)}\pis\Lb\Psi_{\ga}^{j}+{}^{(Z)}\pis_{\Lb A}\ds^A\Psi_{\ga}^{j},\\
\mu{}^{(Z)}\slashed{C}_{\ga,A}^{j}&= g(\mu{}^{(Z)}C_{\ga}^{j},X_A)=-\frac{1}{2}{}^{(Z)}\pis_{\Lb A}L\Psi_{\ga}^{j}-\frac{1}{2}{}^{(Z)}\pis_{LA}\Lb\Psi_{\ga}^{j}+\frac{1}{2}{}^{(Z)}\pi_{L\Lb}\ds_A\Psi_{\ga}^{j}\\
&\qquad\qquad\qquad\qquad\quad+\mu({}^{(Z)}\pis_{AB}-\frac{1}{2}\trgs{}^{(Z)}\pis\gs_{AB})\ds^B\Psi_{\ga}^{j}.\\
\end{split}
\end{equation}

Note that if one substitutes \eqref{C gamma alpha} into \eqref{div C gamma alpha} directly, then a lengthy and tedious equality for $\mu\divg{}^{(Z)}C_{\ga}^{j}$ is obtained. To treat these terms more convenient, we will
divide $\mu\divg{}^{(Z)}C_{\ga}^{j}$ into the following three parts as in \cite{DLY} and \cite{M-Y}:
\begin{equation*}
	\mu\divg{}^{(Z)}C_{\ga}^{j}={}^{(Z)}\mathscr{N}_1^{j}+{}^{(Z)}\mathscr{N}_2^{j}+{}^{(Z)}\mathscr{N}_3^{j},
\end{equation*}
where
\begin{equation}\label{N1 pi Psi}
	\begin{split}
		{}^{(Z)}\mathscr{N}_1^{j}
		=&\big[\frac{1}{4}L(\mu^{-1}{}^{(Z)}\pi_{\Lb\Lb})+\frac{1}{4}\Lb(\trgs{}^{(Z)}\pis)-\frac{1}{2}\nas^A{}^{(Z)}\pis_{\Lb A}\big]L\Psi_{\ga}^{j}-\big[\frac{1}{2}\Lies_{\Lb}{}^{(Z)}\pis_{LA}\\
		&+\frac{1}{2}\Lies_L{}^{(Z)}\pis_{\Lb A}-\frac{1}{2}\ds_A{}^{(Z)}\pi_{L\Lb}-\nas^B(\mu{}^{(Z)}\slashed\pi_{AB}-\f12\mu\trgs{}^{(Z)}\pis\slashed g_{AB})\big]\ds^A\Psi_{\ga}^{j}\quad\\
		&+\big[\frac{1}{4}L(\trgs{}^{(Z)}\pis)-\frac{1}{2}\nas^A{}^{(Z)}\pis_{LA}\big]\Lb \Psi_{\ga}^{j},
	\end{split}
\end{equation}
\begin{equation}\label{N2 pi Psi}
	\begin{split}
		{}^{(Z)}\mathscr{N}_2^{j}
		=&\frac{1}{2}\trgs{}^{(Z)}\pis(L+\frac{1}{2}\trgs\chi)\Lb\Psi_{\ga}^{j}+\frac{1}{4}\mu^{-1}{}^{(Z)}\pi_{\Lb\Lb}L^2\Psi_{\ga}^{j}-{}^{(Z)}\pis_{\Lb A}\ds^AL\Psi_{\ga}^{j}\\
		&-{}^{(Z)}\pis_{LA}\ds^A\Lb\Psi_{\ga}^{j}+\frac{1}{2}{}^{(Z)}\pi_{L\Lb}\Des\Psi_{\ga}^{j}+\mu({}^{(Z)}\slashed\pi_{AB}-\f12\trgs{}^{(Z)}\pis\slashed g_{AB})\nas_{AB}^2\Psi_{\ga}^{j},
	\end{split}
\end{equation}
\begin{equation}\label{N3 pi Psi}
	\begin{split}
		{}^{(Z)}\mathscr{N}_3^{j}
		=&\big[\frac{1}{4}\mu^{-1}\trgs\chi{}^{(Z)}\pi_{\Lb\Lb}-\frac{1}{4}c^{-2}\mu\trgs\chi\trgs{}^{(Z)}\pis
		+\frac{1}{2}{}^{(Z)}\pis_{LA}\ds^A(c^{-2}\mu)\big]L\Psi_{\ga}^{j}\ \quad\quad\\
		&+\big[\frac{1}{2}c^2\ds_A(c^{-2}\mu)\trgs{}^{(Z)}\pis-\frac{1}{2}L(c^{-2}\mu){}^{(Z)}\pis_{LA}-\frac{1}{2}\trgs\chi{}^{(Z)}\pis_{\Lb A}\\
		&+\frac{1}{2}c^{-2}\mu\trgs\chi{}^{(Z)}\pis_{LA}-c^{-2}\mu{}^{(Z)}\pis_L^B\chi_{AB}+{}^{(Z)}\pis_{\Lb}^B\chi_{AB}\big]\ds^A\Psi_{\ga}^{j}.
	\end{split}
\end{equation}
One sees that ${}^{(Z)}\mathscr{N}_1^{j}$ collects the products of the first order derivatives of
the deformation tensor and the first order derivatives of $\Psi_\ga^{j}$, and ${}^{(Z)}\mathscr{N}_2^{j}$ contains the terms which are the products of the deformation tensor and the second order derivatives of $\Psi_\ga^{j}$ except the first term which has the good smallness with respect to $\de$ due to \eqref{transport equation of Lb vp_ga}. In addition, ${}^{(Z)}\mathscr{N}_3^{j}$ is the collections of the products of the deformation tensor and the first order derivatives of $\Psi_\ga^{j}$ which can be treated much easier.

\vskip 0.1 true cm

\textbf{Step 6. Displaying and analyzing expressions of ${}^{(Z)}\N_k^{j},\ k=1,2,3$}

\vskip 0.1 true cm

By substituting the components of the deformation tensor, we next derive the expression of ${}^{(Z)}\mathscr{N}_1^{j}$.
\begin{itemize}
	\item
	\textbf{The case of $Z=\rho L$}
	
	By \eqref{L pi}, one has
	\begin{equation}\label{N1 pi Psi rho L}
		\begin{split}
			{}^{(\rho L)}\mathscr{N}_1^{j}
			=&\Big\{\rho L^2(c^{-2}\mu)+\underline{\frac{1}{2}\Lb(\rho\trgs\chi)-\rho\nas^A\big[c^2\ds_A(c^{-2}\mu)\big]}\Big\}L\Psi_\ga^j-\Big\{\Lies_L[\rho c^2\ds_A(c^{-2}\mu)]\\
			&+\ds_A(\rho L\mu+\mu)\uwave{-2\rho\nas^B(\mu\check\chi_{AB}-\f12\mu\trgs\check\chi\slashed g_{AB})}\Big\}\ds^A\Psi_\ga^j+\frac{1}{2}L(\rho\trgs\check\chi)\Lb\Psi_\ga^j.
		\end{split}
	\end{equation}
	
	For the terms with underline in \eqref{N1 pi Psi rho L}, we point out that there exists an important cancelation.
	In fact, by \eqref{transport equation of chi along T}, we have
\begin{equation}\label{transport equation of trgs chi along T}
	T\trgs\chi=\Des\mu+\mathcal{I}
\end{equation}
with
\begin{equation}\label{I}
	\mathcal{I}=c^{-2}\mu|\chi|^2-c^{-1}L(c^{-1}\mu)\trgs\chi-\nas^A(c^{-1}\mu\ds_Ac)+c^{-2}\mu|\ds c|^2-c^{-1}\ds_Ac\ds^A\mu.
\end{equation}
It follows from \eqref{transport equation of trgs chi along T} that
	\begin{equation}\label{important cancelation 1}
		\begin{split}
			\frac{1}{2}\Lb(\rho\trgs\chi)-\rho\nas^A\big[c^2\ds_A(c^{-2}\mu)\big]=\rho\mathcal{I}-\trgs\chi+\frac{1}{2}c^{-2}\mu L(\rho\trgs\chi)-\rho\nas^A\big[c^2\mu\ds_A(c^{-2})\big],
		\end{split}
	\end{equation}
	which implies that the troublesome term $\Des\mu$ in the underline part of \eqref{N1 pi Psi rho L} has been eliminated.
	
	For the term with wavy line in \eqref{N1 pi Psi rho L}, it follows from \eqref{elliptic equation of chi} that
	\begin{equation*}
		\begin{split}
			-2\rho\nas^B(\mu\check\chi_{AB}-\f12\mu\trgs\check\chi\slashed g_{AB})=&-\rho\mu\uwave{\ds_A\trgs\chi}+2\rho\ze^B\check\chi_{AB}-2\rho\ze_A\trgs\check\chi-2\zeta_A\\
&-2\rho\ds^B\mu(\check\chi_{AB}-\f12\trgs\check\chi\slashed g_{AB}).\\
		\end{split}
	\end{equation*}
		
	Therefore, we arrive at
	\begin{equation}\label{principal term of N1 rho L}
		{}^{(\rho L)}\mathscr{N}_1^{j}=\rho\mu\ds\trgs\chi\cdot\ds\Psi_\ga^j+\lot,
	\end{equation}
	where and below ``$\lot$" stands for the lower order derivative term.
	\item
	\textbf{The case of $Z=T$}
	
	By \eqref{T pi}, we have
	\begin{equation}\label{N1 pi Psi T}
		\begin{split}
			{}^{(T)}\mathscr{N}_1^{j}
			=&\Big\{L T(c^{-2}\mu)\underbrace{-\frac{1}{2}\Lb(c^{-2}\mu\trgs\chi)+\frac{1}{2}\nas^A\big[\mu\ds_A(c^{-2}\mu)\big]}\Big\}L\Psi_\ga^j+\Big\{\underline{\frac{1}{2}\Lies_{\Lb}\big[c^2\ds_A(c^{-2}\mu)\big]}\\
			&+\frac{1}{2}\Lies_L\big[\mu\ds_A(c^{-2}\mu)\big]\underline{-\ds_AT\mu}\uwave{-2\nas^B\big(c^{-2}\mu^2(\check\chi_{AB}-\f12\trgs\check\chi\slashed g_{AB})\big)}\Big\}\ds^A\Psi_\ga^j\\
			&+\Big\{-\frac{1}{2}L(c^{-2}\mu\trgs\chi)+\underbrace{\frac{1}{2}\nas^A\big[c^2\ds_A(c^{-2}\mu)\big]}\Big\}\Lb\Psi_\ga^j.
		\end{split}
	\end{equation}
	
	For the terms with underline in \eqref{N1 pi Psi T}, one has
	\begin{equation}\label{important cancelation 2}
		\begin{split}
			&\frac{1}{2}\Lies_{\Lb}\big[c^2\ds_A(c^{-2}\mu)\big]-\ds_AT\mu =\frac{1}{2}c^{-2}\mu\ds_AL\mu+\frac{1}{2}\Lies_{\Lb}\big[c^2\mu\ds_A(c^{-2})\big],
		\end{split}
	\end{equation}
	which implies that the troublesome term $\ds_AT\mu$ in the underline part of \eqref{N1 pi Psi T} has been eliminated.
	
	For the terms with braces  in \eqref{N1 pi Psi T}, we have
	\begin{equation*}
		\begin{split}
			&-\frac{1}{2}\Lb(c^{-2}\mu\trgs\chi)+\frac{1}{2}\nas^A\big[\mu\ds_A(c^{-2}\mu)\big]\\
			=&-\frac{1}{2}c^{-2}\mu\underbrace{\Des\mu}-c^{-2}\mu \mathcal{I}-\frac{1}{2}c^{-4}\mu^2L\trgs\chi-\frac{1}{2}\Lb(c^{-2}\mu)\trgs\chi+\frac{1}{2}\ds(c^{-2}\mu)\cdot\ds\mu+\frac{1}{2}\nas^A\big[\mu^2\ds_A(c^{-2})\big]
		\end{split}
	\end{equation*}
	and
	\begin{equation*}
		\frac{1}{2}\nas^A\big[c^2\ds_A(c^{-2}\mu)\big]=\frac{1}{2}\underbrace{\Des\mu}+\frac{1}{2}\nas^A\big[c^2\mu\ds_A(c^{-2})\big].
	\end{equation*}
	
	For the terms with wavy line  in \eqref{N1 pi Psi T},
	\begin{equation*}
		\begin{split}
			&-2\nas^B\big(c^{-2}\mu^2(\check\chi_{AB}-\f12\trgs\check\chi\slashed g_{AB})\big)\\
			=&-c^{-2}\mu^2\uwave{\ds_A\trgs\chi}+2c^{-2}\mu\ze^B\check\chi_{AB}-2c^{-2}\mu\ze_A\trgs\check\chi-2c^{-2}\mu\rho^{-1}\zeta_A-2\ds^B(c^{-2}\mu^2)(\check\chi_{AB}-\f12\trgs\check\chi\slashed g_{AB}).
		\end{split}
	\end{equation*}
	
	Therefore, we arrive at
	\begin{equation}\label{principal term of N1 T}
		{}^{(T)}\mathscr{N}_1^{j}=(\Des\mu)T\Psi_\ga^j-c^{-2}\mu^2\ds\trgs\chi\cdot\ds\Psi_\ga^j+\lot.
	\end{equation}

	\item
	\textbf{The case of $Z=R_i$}
	
	By \eqref{Ri pi}, we have
	\begin{equation}\label{N1 pi Psi Ri}
		\begin{split}
			{}^{(R_i)}\mathscr{N}_1^{j}
			=&\Big\{LR_i(c^{-2}\mu)+\underline{\frac{1}{2}\Lb(c^{-1}v_i\trgs\chi)}\underbrace{-\frac{1}{2}\nas^A(c^{-2}\mu R_i^B\chic_{AB})}\underline{-\nas^A(c^{-1}v_i\ds_A\mu)}\\
			&-\frac{1}{2}\nas^A\big[-2c^{-2}(c-1)\mu\rho^{-1}\gs_{AB}R_i^B-3c^{-2}\mu v_i\ds_Ac+c^{-2}\mu\ep_{ijk}\check{L}^j\ds_Ax^k\\
			&+2c^{-1}\mu\ep_{ijk}\check{T}^j\ds_Ax^k\big]\Big\}L\Psi_\ga^j+\Big\{\underline{\frac{1}{2}\Lies_{\Lb}(R_i^B\chic_{AB}}-\ep_{ijk}\check{L}^j\ds_Ax^k+v_i\ds_Ac)\\
			&-\frac{1}{2}\Lies_L{}^{(R_i)}\pis_{\Lb A}\underline{-\ds_AR_i\mu}+\uwave{\nas^B\big(2c^{-1}\mu v_i(\check\chi_{AB}-\f12\trgs\check\chi\slashed g_{AB})\big)}\Big\}\ds^A\Psi_\ga^j\\
			&+\Big\{\frac{1}{2}L(c^{-1}v_i\trgs\chi)+\underbrace{\frac{1}{2}\nas^A(R_i^B\chic_{AB}}-\ep_{ijk}\check{L}^j\ds_Ax^k+v_i\ds_Ac)\Big\}\Lb\Psi_\ga^j.
		\end{split}
	\end{equation}
	
	For the terms with underline in \eqref{N1 pi Psi Ri}, one has that by \eqref{transport equation of chi along T},
	\begin{equation*}
		\begin{split}
			\frac{1}{2}\Lb(c^{-1}v_i\trgs\chi)-\nas^A(c^{-1}v_i\ds_A\mu)=c^{-1}v_i\mathcal{I}+\frac{1}{2}c^{-3}\mu v_iL\trgs\chi+\frac{1}{2}\Lb(c^{-1}v_i)\trgs\chi-\ds(c^{-1}v_i)\cdot\ds\mu=\lot
		\end{split}
	\end{equation*}
	and
	\begin{equation*}
		\begin{split}
			&\quad\frac{1}{2}\Lies_{\Lb}(R_i^B\chic_{AB})-\ds_AR_i\mu=R_i^B\Lies_T\chic_{AB}-\nabla_A(R_i^B\ds_B\mu)+\lot=\lot,
		\end{split}
	\end{equation*}
	which implies that the terms $\Des\mu$ and $\ds_AR_i\mu$ in the underline part of \eqref{N1 pi Psi Ri} have been
eliminated.
	
	For the terms with braces  in \eqref{N1 pi Psi Ri}, we have
	\begin{equation*}
		\begin{split}
			&-\frac{1}{2}\nas^A(c^{-2}\mu R_i^B\chic_{AB})=-\frac{1}{2}c^{-2}\mu\underbrace{R_i\trgs\chi}+\lot
		\end{split}
	\end{equation*}
	and
	\begin{equation*}
		\frac{1}{2}\nas^A(R_i^B\chic_{AB})=\frac{1}{2}\underbrace{R_i\trgs\chi}+\lot.
	\end{equation*}
	
	For the terms with wavy line  in \eqref{N1 pi Psi Ri},
	\begin{equation*}
		\begin{split}
			-\nas^B\big(2c^{-1}\mu v_i(\check\chi_{AB}-\f12\trgs\check\chi\slashed g_{AB})\big)=-c^{-1}\mu v_i\uwave{\ds_A\trgs\chi}+\lot.
		\end{split}
	\end{equation*}
	
	Therefore, we arrive at
	\begin{equation}\label{principal term of N1 Ri}
		{}^{(R_i)}\mathscr{N}_1^{j}=R_i^A(\ds_A\trgs\chi)T\Psi_\ga^j-c^{-1}\mu v_i\ds\trgs\chi\cdot\ds\Psi_\ga^j+\lot.
	\end{equation}
\end{itemize}

Similarly to the treatment for ${}^{(Z)}\mathscr{N}_1^{j}$, one has
\begin{equation}\label{N2 pi Psi rho L}
	\begin{split}
		{}^{(\rho L)}\mathscr{N}_2^{j}
		=&\rho\trgs\chi(L+\frac{1}{2}\trgs\chi)\Lb\Psi_\ga^j+\big[\rho L(c^{-2}\mu)-c^{-2}\mu+2\big]L^2\Psi_\ga^j-2\rho c^2\ds(c^{-2}\mu)\cdot\ds L\Psi_\ga^j\\
		&-(\rho L\mu+\mu)\Des\Psi_\ga^j+2\rho\mu(\check\chi^{AB}-\f12\trgs\check\chi\slashed g^{AB})\nas^2_{AB}\Psi_\ga^j,
	\end{split}
\end{equation}

\begin{equation}\label{N2 pi Psi T}
	\begin{split}
		{}^{(T)}\mathscr{N}_2^{j}
		=&-c^{-2}\mu\trgs\chi(L+\frac{1}{2}\trgs\chi)\Lb\Psi_\ga^j+T(c^{-2}\mu)L^2\Psi_\ga^j+\mu\ds(c^{-2}\mu)\cdot\ds L\Psi_\ga^j\\
		&+c^2\ds_A(c^{-2}\mu)\ds^A\Lb\Psi_\ga^j-T\mu\Des\Psi_\ga^j-2c^{-2}\mu^2(\check\chi^{AB}-\f12\trgs\check\chi\slashed g^{AB})\nas^2_{AB}\Psi_\ga^j,\qquad\qquad
	\end{split}
\end{equation}

\begin{equation}\label{N2 pi Psi Ri}
	\begin{split}
		{}^{(R_i)}\mathscr{N}_2^{j}
		=&c^{-1}v_i\trgs\chi(L+\frac{1}{2}\trgs\chi)\Lb\Psi_\ga^j-\big(c^{-2}\mu{}^{(R_i)}\pis_{LA}+2{}^{(R_i)}\pis_{TA}\big)\ds^AL\Psi_\ga^j-R_i\mu\Des\Psi_\ga^j\ \\
		&+{}^{(R_i)}\pis_{LA}\ds^A\Lb\Psi_\ga^j+R_i(c^{-2}\mu)L^2\Psi_\ga^j+2c^{-1}\mu v_i(\check\chi^{AB}-\f12\trgs\check\chi\slashed g^{AB})\nas^2_{AB}\Psi_\ga^j
	\end{split}
\end{equation}
and
\begin{equation}\label{N3 pi Psi rho L}
	\begin{split}
		{}^{(\rho L)}\mathscr{N}_3^{j}
		=&\trgs\chi\big\{\rho L(c^{-2}\mu)-2c^{-2}(\mu-1)+2(1-c^{-2})-\frac{1}{2}\rho c^{-2}\mu\trgs\check\chi\big\}L\Psi_\ga^j\qquad\qquad\qquad\\
		&+2\rho c^2\ds^B(c^{-2}\mu)\chi_{AB}\ds^A\Psi_\ga^j,
	\end{split}
\end{equation}

\begin{equation}\label{N3 pi Psi T}
	\begin{split}
		{}^{(T)}\mathscr{N}_3^{j}
		=&\big[T(c^{-2}\mu)\trgs\chi+\frac{1}{2}c^{-4}\mu^2(\trgs\chi)^2-\frac{1}{2}c^2|\ds(c^{-2}\mu)|^2\big]L\Psi_\ga^j+\big[\frac{1}{2}c^2L(c^{-2}\mu)\ds_A(c^{-2}\mu)\\
		&-\mu\trgs\chi\ds_A(c^{-2}\mu)\big]\ds^A\Psi_\ga^j,
	\end{split}
\end{equation}

\begin{equation}\label{N3 pi Psi Ri}
	\begin{split}
		{}^{(R_i)}\mathscr{N}_3^{j}
		=&\big[R_i(c^{-2}\mu)\trgs\chi-\frac{1}{2}c^{-3}\mu v_i(\trgs\chi)^2+\frac{1}{2}{}^{(R_i)}\pis_{LA}\ds^A(c^{-2}\mu)\big]L\Psi_\ga^j\\
		&+\big[cv_i\ds_A(c^{-2}\mu)\trgs\chi-\frac{1}{2}L(c^{-2}\mu){}^{(R_i)}\pis_{LA}+2{}^{(R_i)}\pis_{T}^B\check\chi_{AB}
-{}^{(R_i)}\pis_{TA}\trgs\check\chi\big]\ds^A\Psi_\ga^j.
	\end{split}
\end{equation}

After making the preparations for the $L^2$ estimates of the related quantities
(see Section \ref{Section 6} below), we can handle the error terms $\int_{D^{\t,u}}|\Phi|\cdot|L\Psi|$
and $\int_{D^{\t,u}}|\Phi|\cdot|\Lb\Psi|$ in \eqref{energy inequality} so that all the energy
estimates for $\vp_{\ga}$ and its derivatives in \eqref{bootstrap assumptions} are obtained.

\section{$L^2$ estimates of higher order derivatives}\label{Section 6}

In this section, we shall establish the higher order derivative $L^2$ estimates for some related quantities
so that the last two error terms of \eqref{energy inequality} can be absorbed by the left hand side.

\begin{lemma}\label{Lemma 1 in higher order L^2 estimates}
For any $\Psi\in C^1(D^{\t,u})$ which vanishes on $C_0$ and sufficiently small $\de>0$, one has
\begin{equation*}
	\int_{\Si_\t^u}\Psi^2\les\de^2\big\{E_1[\Psi](\t,u)+E_2[\Psi](\t,u)\big\}.
\end{equation*}
\end{lemma}
\begin{proof}
See the proof in Lemma 7.3 of \cite{M-Y}.
\end{proof}

\begin{lemma}\label{Lemma 2 in higher order L^2 estimates}
For any trace-free symmetric $(0,2)$-type tensor field $\xi$ on $S_{\t,u}$ or function $f\in C^2(D^{\t,u})$, it holds that
\begin{equation}\label{elliptic estimate 1}
	\int_{S_{\t,u}}\mu^2\big(|\nas\xi|^2+2\mathcal{G}|\xi|^2\big)\les\int_{S_{\t,u}}\big(\mu^2|\divgs\xi|^2+|\ds\mu|^2|\xi|^2\big),
\end{equation}
\begin{equation}\label{elliptic estimate 2}
	\int_{S_{\t,u}}\mu^2\big(|\nas^2f|^2+\mathcal{G}|\ds f|^2\big)\les\int_{S_{\t,u}}\big(\mu^2|\Des f|^2+|\ds\mu|^2|\ds f|^2\big),
\end{equation}
where the Gaussian curvature $\mathcal{G}$ of $\gs$ satisfies
\begin{equation}\label{estimate of Gaussian curvature}
	\mathcal{G}=1+O\big(\de^{(1-\ve0)p}\big).
\end{equation}
\end{lemma}
\begin{proof}
The elliptic estimates \eqref{elliptic estimate 1} and \eqref{elliptic estimate 2} hold immediately from \cite[(18.26), (18.15)]{Sp}. Moreover, by \cite[(2.29)]{M-Y}, one has
\begin{equation}\label{Gaussian curvature}
	\mathcal{G}=\f12c^{-2}\big[(\trgs\chi)^2-|\chi|^2\big].
\end{equation}
Together with \eqref{YHC-1}, it yields
\begin{equation*}
	\mathcal{G}=\f12c^{-2}\big[2\rho^{-2}+(\trgs\chic)^2+2\rho^{-1}\trgs\chic-|\chic|^2\big]=1+O\big(\de^{(1-\ve0)p}\big).
\end{equation*}
\end{proof}

\begin{lemma}\label{Lemma 3 in higher order L^2 estimates}
Let $\tu=\dse\inf_{\t'}\left\{1\leq\t'\leq\t<t^*:\mu_{\min}(\t')<\f{1}{10}\right\}$. For $a>1$ and sufficiently small $\de>0$, one has
\begin{itemize}
	\item[(1)] Integrability:
		\begin{equation}\label{integrability of mu}
			\int_{\tu}^{\t}\mu_{\min}^{-a}(\t')\d\t'\les\de^{1-(1-\ve0)p}\f{1}{a-1}\mu_{\min}^{-a+1}(\t).
		\end{equation}
	\item[(2)] Monotonicity:
		\begin{equation}\label{monotonicity of mu}
			\mu_{\min}^{-1}(\t')\les\mu_{\min}^{-1}(\t).
		\end{equation}
\end{itemize}			
\end{lemma}
\begin{proof}
(1) For $\t\geq\tu$, by the proof of \eqref{key estimate of L mu} and \eqref{key estimate of T mu}, $\mu_{\min}(\t)$ behaves exactly as $\de^{-[1-(1-\ve0)p]}(t^*-\t)$, which implies
\begin{equation*}
	\begin{split}
		\int_{\tu}^{\t}\mu_{\min}^{-a}(\t')\d\t'
		&\les(\de^{-[1-(1-\ve0)p]})^{-a}\int_{\tu}^{\t}(t^*-\t')^{-a}\d\t'\\
		&\leq(\de^{-[1-(1-\ve0)p]})^{-a}\f{1}{a-1}(t^*-\t)^{-a+1}\\
		&\les(\de^{-[1-(1-\ve0)p]})^{-a}\f{1}{a-1}(\de^{1-(1-\ve0)p})^{-a+1}\mu_{\min}^{-a+1}(\t)\\
		&=\de^{1-(1-\ve0)p}\f{1}{a-1}\mu_{\min}^{-a+1}(\t).
	\end{split}
\end{equation*}
	
(2) Similarly to the proof of \eqref{key estimate of L mu}, if $\mu(\t,u,\vt)\leq1-\f1a$, one
has $L\mu(\t,u,\vt)\les-\f1a\de^{-[1-(1-\ve0)p]}$, which means that if there is a $\tau\in[1,t^*)$
such that $\mu_{\min}(\tau)\leq1-\f1a$, then for all $\t\geq\tau$, $\mu_{\min}(\t)\leq1-\f1a$.
We now define a time $\tu_a$ such that it is the minimum of all the $\tau$ with $\mu_{\min}(\tau)\leq1-\f1a$.

If $\t'\geq\tu_a$, then $\mu_{\min}(\t')\leq1-\f1a$. Let $\mu_{\min}(\t')=\mu(\t',u_{\t'},\vt_{\t'})$. Since $\mu(\t,u_{\t'},\vt_{\t'})$ is decreasing in $\t$ for $\t\geq\t'\geq\tu_a$ (due to $L\mu(\t,u,\vt)\les-\f1a\de^{-[1-(1-\ve0)p]}<0$), one then has
\begin{equation*}
	\mu_{\min}^{-a}(\t')=\mu^{-a}(\t',u_{\t'},\vt_{\t'})\leq\mu^{-a}(\t,u_{\t'},\vt_{\t'})\leq\mu_{\min}^{-a}(\t).
\end{equation*}

If $\t'\leq\tu_a$, then $\mu_{\min}(\t')\geq 1-\f1a$. Since $\mu_{\min}(\t)\les 1$, we have
\begin{equation*}
	\mu_{\min}^{-a}(\t')\leq(1-\f1a)^{-a}\les1\les\mu_{\min}^{-a}(\t).
\end{equation*}
\end{proof}

\subsection{Non-top order derivative $L^2$ estimates}\label{Section 6.1}

Based on the preparations above, we are ready to derive the non-top order derivative $L^2$ estimates for the related quantities.
\begin{proposition}\label{Proposition non-top order L^2 estimates}
Under the assumptions \eqref{bootstrap assumptions}, when $\de>0$ is small, it holds  that for $|\al|\leq 2N-10$,
\begin{equation}\label{non-top order L^2 estimates}
	\begin{split}
		&\de^{l+s[1-(1-\ve0)p]}\|\Lies_Z^{\al}\chic\|_{L^2(\Si_\t^u)}\les\de^{(1-\ve0)p+\f12}+\Theta_M^1(\t,u),\\
		&\de^{l+s[1-(1-\ve0)p]}\|Z^{\al+1}\mu\|_{L^2(\Si_\t^u)}\les\de^{\f12}+\Theta_M^2(\t,u),\\
		&\de^{l+s[1-(1-\ve0)p]}\left(\|Z^{\al+1}\Lc^j\|_{L^2(\Si_\t^u)}+\|Z^{\al+1}\Tc^j\|_{L^2(\Si_\t^u)}+\|Z^{\al+1}v_j\|_{L^2(\Si_\t^u)}\right)\les\de^{(1-\ve0)p+\f12}+\Theta_M^1(\t,u),\\
		&\de^{l+s[1-(1-\ve0)p]}\left(\|\Lies_Z^{\al}\ds x^j\|_{L^2(\Si_\t^u)}+\|\Lies_Z^{\al}R_j\|_{L^2(\Si_\t^u)}\right)\les\de^{\f12}+\Theta_M^1(\t,u),\\	
		&\de^{l+s[1-(1-\ve0)p]}\left(\|\Lies_Z^{\al}{}^{(R_j)}\pis\|_{L^2(\Si_\t^u)}+\|\Lies_Z^{\al}{}^{(R_j)}\pis_L\|_{L^2(\Si_\t^u)}+\|\Lies_Z^{\al}{}^{(R_j)}\pis_T\|_{L^2(\Si_\t^u)}\right)\les\de^{(1-\ve0)p+\f12}+\Theta_M^1(\t,u),\\	
		&\de^{l+s[1-(1-\ve0)p]}\left(\|\Lies_Z^{\al}{}^{(T)}\pis\|_{L^2(\Si_\t^u)}
+\|\Lies_Z^{\al}{}^{(T)}\pis_L\|_{L^2(\Si_\t^u)}\right)\les\de^{\f12}+\Theta_M^2(\t,u),
	\end{split}
 \end{equation}
where $l$ (or $s$) is the number of $T$ (or $\rho L$) in the string of $Z$ with $s\leq 2$, and
\begin{equation*}
 	\begin{split}
 		&\Theta_M^1(\t,u)=\de^{(1-\ve0)(p-1)}\int_1^\t\left(\mu_{\min}^{-\f12}(\t')\sqrt{\Et_{1,\leq|\al|+2}(\t',u)}+\de^{(1-\ve0)p}\sqrt{\Et_{2,\leq|\al|+2}(\t',u)}\right)\d\t',\\
 		&\Theta_M^2(\t,u)=\de^{(1-\ve0)(p-1)}\int_1^\t\left(\mu_{\min}^{-\f12}(\t')\sqrt{\Et_{1,\leq|\al|+2}(\t',u)}+\sqrt{\Et_{2,\leq|\al|+2}(\t',u)}\right)\d\t'.
 	\end{split}
\end{equation*}
\end{proposition}
\begin{proof}
We will prove this proposition by the induction method with respect to the index $\al$. When $\al=0$, in view of
Proposition \ref{Proposition higher order L^infty estimates}, the corresponding $L^2$ estimates can be directly
obtained by the fact $\|1\|_{L^2(\Si_\t^u)}\les\de^{1/2}$ (similarly to the proof of \cite[Corollary 11.30.3]{Sp}).
For example, $\|\chic\|_{L^2(\Si_\t^u)}\les\|\chic\|_{L^{\infty}(\Si_\t^u)}\cdot\|1\|_{L^2(\Si_\t^u)}\les \de^{(1-\ve0)p+\f12}$.
For any $1\leq|\al|\leq 2N-10$, assume that \eqref{non-top order L^2 estimates} holds up to the order $|\al|-1$,
one needs to show \eqref{non-top order L^2 estimates} for the order $|\al|$. To this end,
we take the $L^2$ norms for the factors equipped with the higher order derivatives of
the related quantities, meanwhile applying Proposition \ref{Proposition higher order L^infty estimates}
for the corresponding $L^{\infty}$ coefficients in these terms.

As in the proof of Proposition \ref{Proposition higher order L^infty estimates}, we first prove \eqref{non-top order L^2 estimates}
when $Z$ only involves the rotation vector, that is, $Z\in\{R_1,R_2,R_3\}$.

\vskip 0.2 true cm

\textbf{Part 1. Estimates of $\Lies_{R_i}^{\al}\chic$ and $\Lies_{R_i}^{\al}{}^{(R_j)}\pis_L$}

\vskip 0.1 true cm	

By \eqref{commute Lies Ri al with L chic}, one has
 \begin{equation*}
 \begin{split}
 &\quad\Lies_L\Lies_{R_i}^{\al}\chic_{AB}\\
 &=\sum_{\be_1+\be_2=\al}\Lies_{R_i}^{\be_1}(c^{-1}Lc)\cdot\Lies_{R_i}^{\be_2}\chic_{AB}
 +\sum_{\be_1+\be_2+\be_3=\al}\Lies_{R_i}^{\be_1}\gs^{CD}\cdot\Lies_{R_i}^{\be_2}\chic_{AD}\cdot\Lies_{R_i}^{\be_3}\chic_{BC}
 +\Lies_{R_i}^{\al}(\rho^{-1}c^{-1}Lc\gs_{AB})\\
 &\quad -\Lies_{R_i}^{\al}\Rc_{LALB}+\sum_{\be_1+\be_2=\al-1}\Lies_{R_i}^{\be_1}\big[{}^{(R_i)}\pis_L^C(\nas_C\Lies_{R_i}^{\be_2}\chic_{AB})
 +\Lies_{R_i}^{\be_2}\chic_{BC}(\nas_A{}^{(R_i)}\pis_L^C)+\Lies_{R_i}^{\be_2}\chic_{AC}(\nas_B{}^{(R_i)}\pis_L^C)\big]\\
 &=I_1+I_2+I_3+I_4+I_5,
 \end{split}
 \end{equation*}
where we have neglected the unimportant coefficient constants.

For $I_1$, by Lemma \ref{Lemma 1 in higher order L^2 estimates} and \eqref{[X,Y]}, then
\begin{equation*}
 	\begin{split}
 		\|I_1\|_{L^2(\Si_\t^u)}
 		&\les\de^{(1-\ve0)p-[1-(1-\ve0)p]}\|\Lies_{R_i}^{\leq\al}\chic\|_{L^2(\Si_\t^u)}\\
 		&\quad+\de^{(1-\ve0)2p}\|\Lies_{R_i}^{\leq\al-1}{}^{(R_j)}\pis_L\|_{L^2(\Si_\t^u)}\\
 		&\quad+\de^{(1-\ve0)(2p-1)}\left(\mu_{\min}^{-\f12}(\t)\sqrt{\Et_{1,\leq|\al|+1}(\t,u)}+\de^{(1-\ve0)p}\sqrt{\Et_{2,\leq|\al|+1}(\t,u)}\right).
 	\end{split}
\end{equation*}

For $I_2$, we have
\begin{equation*}
 	\|I_2\|_{L^2(\Si_\t^u)}\les\de^{(1-\ve0)p}\|\Lies_{R_i}^{\leq\al}\chic\|_{L^2(\Si_\t^u)}
 +\de^{(1-\ve0)2p}\|\Lies_{R_i}^{\leq\al-1}{}^{(R_j)}\pis\|_{L^2(\Si_\t^u)}.
\end{equation*}

In addition,
\begin{equation*}
	\begin{split}
		\|I_3\|_{L^2(\Si_\t^u)}
		&\les\de^{(1-\ve0)p-[1-(1-\ve0)p]}\|\Lies_{R_i}^{\leq\al-1}{}^{(R_j)}\pis\|_{L^2(\Si_\t^u)}\\
		&\quad+\|R_i^{\leq\al}(-\f12pc^2\vp_0^{p-1}L\vp_0)\|_{L^2(\Si_\t^u)}\\
		&\les\de^{(1-\ve0)p-[1-(1-\ve0)p]}\|\Lies_{R_i}^{\leq\al-1}{}^{(R_j)}\pis\|_{L^2(\Si_\t^u)}\\
		&\quad+\de^{(1-\ve0)p}\|\Lies_{R_i}^{\leq\al-1}{}^{(R_j)}\pis_L\|_{L^2(\Si_\t^u)}\\
		&\quad+\de^{(1-\ve0)(p-1)}\left(\mu_{\min}^{-\f12}(\t)\sqrt{\Et_{1,\leq|\al|+1}(\t,u)}+\de^{(1-\ve0)p}\sqrt{\Et_{2,\leq|\al|+1}(\t,u)}\right),
	\end{split}
\end{equation*}
\begin{equation*}
 	\|I_4\|_{L^2(\Si_\t^u)}\les\de^{(1-\ve0)(p-1)}\|\ds R_i^{\leq\al+1}\vp_0\|_{L^2(\Si_\t^u)}\les\de^{(1-\ve0)(p-1)}\mu_{\min}^{-\f12}(\t)\sqrt{\Et_{1,\leq|\al|+2}(\t,u)}.
\end{equation*}

For $I_5$, by the expression of ${}^{(R_i)}\pis_{LA}$ in \eqref{Ri pi}, we arrive at
\begin{equation*}
 	\begin{split}
 		\|I_5\|_{L^2(\Si_\t^u)}
 		&\les\de^{(1-\ve0)p}\|\Lies_{R_i}^{\leq\al}\chic\|_{L^2(\Si_\t^u)}\\
 		&\quad+\de^{(1-\ve0)2p}\|\Lies_{R_i}^{\leq\al-1}{}^{(R_j)}\pis\|_{L^2(\Si_\t^u)}\\
 		&\quad+\de^{(1-\ve0)p}\left(\|\Lies_{R_i}^{\leq\al}\chic\|_{L^2(\Si_\t^u)}+\de^{(1-\ve0)p}\|\Lies_{R_i}^{\leq\al}R_j\|_{L^2(\Si_\t^u)}
 +\|R_i^{\leq\al}\Lc^j\|_{L^2(\Si_\t^u)}\right.\\
 		&\quad+\de^{(1-\ve0)p}\|\Lies_{R_i}^{\leq\al}\ds x^j\|_{L^2(\Si_\t^u)}+\de^{(1-\ve0)p}\|R_i^{\leq\al}v_j\|_{L^2(\Si_\t^u)}\\
 		&\left.\quad+\de^{(1-\ve0)(2p-1)}\mu_{\min}^{-\f12}(\t)\sqrt{\Et_{1,\leq|\al|+2}(\t,u)}\right).
 	\end{split}
\end{equation*}

Combining the $L^2$ estimates for all $I_i$ ($1\le i\le 5$), together with the induction hypothesis, we obtain
\begin{equation}\label{L2 estimate of Lies_LLies_Ri^al chic}
 	\begin{split}
 		\|\Lies_L\Lies_{R_i}^{\al}\chic\|_{L^2(\Si_\t^u)}
 		&\les\de^{(1-\ve0)p-[1-(1-\ve0)p]}\|\Lies_{R_i}^{\leq\al}\chic\|_{L^2(\Si_\t^u)}\\
 		&\quad+\de^{(1-\ve0)p-[1-(1-\ve0)p]}\cdot\de^{(1-\ve0)p+\f12}\\
 		&\quad+\de^{(1-\ve0)(p-1)}\left(\mu_{\min}^{-\f12}(\t)\sqrt{\Et_{1,\leq|\al|+2}(\t,u)}+\de^{(1-\ve0)p}\sqrt{\Et_{2,\leq|\al|+2}(\t,u)}\right),
 	\end{split}
\end{equation}
where the second term in the right hand side of \eqref{L2 estimate of Lies_LLies_Ri^al chic} comes from the first term in the estimate of $I_3$.

For any symmetric $(0,2)$-type tensor field $\xi$, one has
\begin{equation*}
  L\big(\rho^4|\xi|^2\big)=2\rho^4\xi^{AB}\Lies_{L}\xi_{AB}-4\rho^4\chic_A^B\xi_B^C\xi_C^A.
\end{equation*}
Substituting $\xi=\Lies_{R_i}^{\al}\chic$ into the above identity yields
\begin{equation*}
  \big|L\big(\rho^2|\Lies_{R_i}^{\al}\chic|\big)\big|\leq \rho^2|\Lies_{L}\Lies_{R_i}^{\al}\chic|
  +2|\chic|\cdot\rho^2|\Lies_{R_i}^{\al}\chic|.
\end{equation*}
According to Newton-Leibniz formula, we have
\begin{equation*}
  \rho^2|\Lies_{R_i}^{\al}\chic|(\t)=(1-u)^2|\Lies_{R_i}^{\al}\chic|(1)+\int_{1}^{\t}L\big[(\t'-u)^2|\Lies_{R_i}^{\al}\chic|\big](\t')\d\t'.
\end{equation*}
Since $\rho\sim 1$, by taking the $L^2(\Si_\t^u)$ norm on both sides of the above identity, we obtain
\begin{equation}\label{inequality of L2 norm of Lies_Ri^al chic}
  \|\Lies_{R_i}^{\al}\chic\|_{L^2(\Si_\t^u)}\les\de^{(1-\ve0)p+\f12}+\int_{1}^{\t}\left\||\chic|\cdot|\Lies_{R_i}^{\al}\chic|
  +|\Lies_{L}\Lies_{R_i}^{\al}\chic|\right\|_{L^2(\Si_{\t'}^u)}\d\t'.
\end{equation}

Substituting \eqref{L2 estimate of Lies_LLies_Ri^al chic} into \eqref{inequality of L2 norm of Lies_Ri^al chic},
 together with Gronwall inequality, yields
\begin{equation*}
  \|\Lies_{R_i}^{\al}\chic\|_{L^2(\Si_\t^u)}\les\de^{(1-\ve0)p+\f12}
  +\de^{(1-\ve0)(p-1)}\int_1^\t\left(\mu_{\min}^{-\f12}(\t')\sqrt{\Et_{1,\leq|\al|+2}(\t',u)}
  +\de^{(1-\ve0)p}\sqrt{\Et_{2,\leq|\al|+2}(\t',u)}\right)\d\t',
\end{equation*}
which also derives the analogous estimate of $L^2$ norm for $\Lies_{R_i}^{\al}{}^{(R_j)}\pis_L$.

\vskip 0.2 true cm

\textbf{Part 2. Estimates of $R_i^{\al+1}\mu$ and ${\Lies_{R_i}^{\al}}{}^{(R_j)}\pis_T$}

\vskip 0.1 true cm	

By \eqref{commute Ri al+1 with L mu}, one has
\begin{equation*}
 	\begin{split}
 		LR_i^{\al+1}\mu
 		&=\sum_{\be_1+\be_2=\al+1}R_i^{\be_1}(c^{-1}Lc)\cdot R_i^{\be_2}\mu-R_i^{\al+1}(cTc)+\sum_{\be_1+\be_2+\be_3=\al}\Lies_{R_i}^{\be_1}\gs^{-1}\cdot\Lies_{R_i}^{\be_2}{}^{(R_j)}\pis_L\cdot\ds R_i^{\be_3}\mu\\
 		&=I'_1+I'_2+I'_3,
 	\end{split}
\end{equation*}
where we have neglected the unimportant coefficient constants.

For $I'_1$, we have
\begin{equation*}
 	\begin{split}
 		\|I'_1\|_{L^2(\Si_\t^u)}
 		&\les\de^{(1-\ve0)p-[1-(1-\ve0)p]}\|R_i^{\leq\al+1}\mu\|_{L^2(\Si_\t^u)}\\
 		&\quad+\de^{(1-\ve0)p}\|\Lies_{R_i}^{\leq\al}{}^{(R_j)}\pis_L\|_{L^2(\Si_\t^u)}\\
 		&\quad+\de^{(1-\ve0)(p-1)}\left(\mu_{\min}^{-\f12}(\t)\sqrt{\Et_{1,\leq|\al|+2}(\t,u)}
 +\de^{(1-\ve0)p}\sqrt{\Et_{2,\leq|\al|+2}(\t,u)}\right).
 	\end{split}
\end{equation*}

For $I'_2$, by Lemma \ref{Lemma 1 in higher order L^2 estimates}, $[T,R_i]={}^{(R_i)}\pis_T$ (due to its expression
in \eqref{Ri pi}) and $T=\f12(\Lb-c^{-2}\mu L)$, we arrive at
\begin{equation*}
 	\begin{split}
 		\|I'_2\|_{L^2(\Si_\t^u)}
 		&\les\|R_i^{\al+1}(\f12pc^4\vp_0^{p-1}T\vp_0)\|_{L^2(\Si_\t^u)}\\
 		&\les\de^{(1-\ve0)p}\left(\de^{(1-\ve0)p}\|R_i^{\leq\al+1}\mu\|_{L^2(\Si_\t^u)}
 +\de^{(1-\ve0)p}\|\Lies_{R_i}^{\leq\al}R_j\|_{L^2(\Si_\t^u)}+\|\Lies_{R_i}^{\leq\al}\chic\|_{L^2(\Si_\t^u)}\right.\\
 		&\quad+\de^{(1-\ve0)p}\|\Lies_{R_i}^{\leq\al-1}{}^{(R_j)}\pis\|_{L^2(\Si_\t^u)}
 +\|R_i^{\leq\al}\Lc^j\|_{L^2(\Si_\t^u)}+\|R_i^{\leq\al}v_j\|_{L^2(\Si_\t^u)}\\
 		&\left.\quad+\de^{(1-\ve0)(p-1)}\mu_{\min}^{-\f12}(\t)\sqrt{\Et_{1,\leq|\al|+1}(\t,u)}\right)\\
 		&\quad+\de^{(1-\ve0)(p-1)}\left(\sqrt{\Et_{1,\leq|\al|+2}(\t,u)}+\sqrt{\Et_{2,\leq|\al|+2}(\t,u)}\right).
 	\end{split}
\end{equation*}

For $I'_3$,
\begin{equation*}
 	\|I'_3\|_{L^2(\Si_\t^u)}\les\de^{(1-\ve0)p}\|R_i^{\leq\al+1}\mu\|_{L^2(\Si_\t^u)}+\de^{(1-\ve0)p}\|\Lies_{R_i}^{\leq\al-1}{}^{(R_j)}\pis\|_{L^2(\Si_\t^u)}+\|\Lies_{R_i}^{\leq\al}{}^{(R_j)}\pis_L\|_{L^2(\Si_\t^u)}.
\end{equation*}

Combining the $L^2$ estimates for all $I'_i$ ($1\le i\le 3$), we obtain that from the induction hypothesis and
the estimates in Part 1
\begin{equation}\label{L2 estimate of L Ri al+1 mu}
 	\begin{split}
 		\|LR_i^{\al+1}\mu\|_{L^2(\Si_\t^u)}
 		&\les\de^{(1-\ve0)p-[1-(1-\ve0)p]}\|R_i^{\leq\al+1}\mu\|_{L^2(\Si_\t^u)}+\de^{(1-\ve0)p+\f12}\\
 		&\quad+\de^{(1-\ve0)(p-1)}\left(\mu_{\min}^{-\f12}(\t)\sqrt{\Et_{1,\leq|\al|+2}(\t,u)}+\sqrt{\Et_{2,\leq|\al|+2}(\t,u)}\right),
 	\end{split}
\end{equation}
where the second term in the right hand side of \eqref{L2 estimate of L Ri al+1 mu} comes from the third term in the estimate of $I'_3$.

According to Newton-Leibniz formula, one has
\begin{equation*}
 	R_i^{\al+1}\mu(\t)=R_i^{\al+1}\mu(1)+\int_1^\t LR_i^{\al+1}\mu(\t')\d\t'.
\end{equation*}
By taking the $L^2(\Si_\t^u)$ norm on both sides of the above identity and utilizing the
Minkowski inequality, then
\begin{equation}\label{inequality of L2 norm of Ri al+1 mu}
 	\|R_i^{\al+1}\mu\|_{L^2(\Si_\t^u)}\les\de^{\f12}+\int_1^\t\|LR_i^{\al+1}\mu\|_{L^2(\Si_{\t'}^u)}\d\t'.
\end{equation}
Substituting \eqref{L2 estimate of L Ri al+1 mu} into \eqref{inequality of L2 norm of Ri al+1 mu}, by Gronwall inequality, we finally arrive at
\begin{equation*}
 	\|R_i^{\al+1}\mu\|_{L^2(\Si_\t^u)}\les\de^{\f12}+\de^{(1-\ve0)(p-1)}\int_1^\t\left(\mu_{\min}^{-\f12}(\t')\sqrt{\Et_{1,\leq|\al|+2}(\t',u)}+\sqrt{\Et_{2,\leq|\al|+2}(\t',u)}\right)\d\t',		
\end{equation*}
which also gives the analogous estimate of $L^2$ norm for $\Lies_{R_i}^{\al}{}^{(R_j)}\pis_T$.

\vskip 0.2 true cm

\textbf{Part 3. Estimates of $R_i^{\al+1}\Lc^j$, $R_i^{\al+1}\Tc^j$ and $R_i^{\al+1}v_j$}

\vskip 0.1 true cm	

By \eqref{structure equation of Lc^j}, one has
\begin{equation*}
	R_i^{\al+1}\Lc^j=\Lies_{R_i}^{\al}(R_i^A\chic_{AB}\ds^Bx^j)=\sum_{\be_1+\be_2+\be_3+\be_4=\al}\Lies_{R_i}^{\be_1}R_j\cdot\Lies_{R_i}^{\be_2}\chic\cdot\Lies_{R_i}^{\be_3}\gs^{-1}\cdot\Lies_{R_i}^{\be_4}\ds x^j.
\end{equation*}
Thanks to the estimate of $\Lies_{R_i}^{\al}\chic$ in Part 1, together with the induction hypothesis, we arrive at
\begin{equation*}
	\|R_i^{\al+1}\Lc^j\|_{L^2(\Si_\t^u)}\les\de^{(1-\ve0)p+\f12}+\de^{(1-\ve0)(p-1)}\int_1^\t\left(\mu_{\min}^{-\f12}(\t')\sqrt{\Et_{1,\leq|\al|+2}(\t',u)}+\de^{(1-\ve0)p}\sqrt{\Et_{2,\leq|\al|+2}(\t',u)}\right)\d\t'.
\end{equation*}
Then the estimates of $R_i^{\al+1}\Tc^j$ and $R_i^{\al+1}v_j$ follow immediately from \eqref{YHC-1} and \eqref{v_i}.

\vskip 0.2 true cm

\textbf{Part 4. Estimates of the other quantities}

\vskip 0.1 true cm	

The estimates of $\Lies_{R_i}^{\al+1}\ds x^j$ and $\Lies_{R_i}^{\al+1}R_j$ follow directly from \eqref{R_i} and Part 3, while the estimates of the other deformation tensors follow from Lemma \ref{Lemma components of the deformation tensor} and the estimates
in the above parts.

If there are vectorfields $T$ or $\rho L$ in the string of $Z$, we could utilize the structure equations in
Lemma \ref{Lemma structure equations} to get the corresponding $L^2$ bounds as in the end of the proof for
Proposition \ref{Proposition higher order L^infty estimates}.
\end{proof}

\subsection{Top order derivative $L^2$ estimates}\label{Section 6.2}

When we try to close the energy estimate of the top order derivatives (see Section \ref{Section 7} below)
and further complete the proof of bootstrap assumptions
in \eqref{bootstrap assumptions}, it is found that the
top orders of derivatives of $\vp_\ga$, $\chi$ and $\mu$ for the energy estimates are $2N-8$, $2N-9$ and $2N-8$ respectively. However, as shown in Proposition \ref{Proposition non-top order L^2 estimates}, the $L^2$ estimates for the $(2N-9)^{\text{th}}$ order derivatives of $\chi$ and $(2N-8)^{\text{th}}$ order derivatives of $\mu$ can be controlled by the $(2N-7)^{\text{th}}$ order derivatives of $\vp_\ga$. So there
exists one order loss of  derivatives in the corresponding energy inequality. To overcome this difficulty, as in \cite{Ch2} and \cite{M-Y}-\!\cite{Sp},
we need to treat $\chi$ and $\mu$ with the related top order derivatives by
introducing some modified quantities and taking the elliptic estimates.

\vskip 0.2 true cm

\textbf{6.2.1 Estimates on $\nas\chi$}

\vskip 0.1 true cm	

Recalling \eqref{Rc_LALB}, we now define
\begin{equation}\label{mu Rc LL}
	\mu\Rc_{LL}=\mu\gs^{AB}\Rc_{LALB}=-\f12pc^4\vp_0^{p-1}\mu\Des\vp_0+R_0,
\end{equation}
where
\begin{equation}\label{R0}
	R_0=-\f32pc^3\mu\vp_0^{p-1}\ds_Ac\ds^A\vp_0-\f12p(p-1)c^4\mu\vp_0^{p-2}|\ds\vp_0|^2.
\end{equation}
By \eqref{transport equation of chi} and \eqref{YHC-1}, one has
\begin{equation}\label{transport equation of trgs chi}
	L\trgs\chi=(c^{-1}Lc-2\rho^{-1})\trgs\chi+2\rho^{-2}-|\chic|^2-\Rc_{LL},
\end{equation}
then
\begin{equation}\label{transport equation of mu trgs chi}
	\begin{split}
		L(\mu\trgs\chi)=(L\mu+c^{-1}Lc\mu-2\rho^{-1}\mu)\trgs\chi+2\rho^{-2}\mu-\mu|\chic|^2-\mu\Rc_{LL}.
	\end{split}
\end{equation}
Note that the highest order derivative of $\vp_0$ in the right hand side of \eqref{mu Rc LL}
is $\Des\vp_0$. We will replace $\mu\Des\vp_0$ in $\mu\Rc_{LL}$ of \eqref{transport equation of mu trgs chi} with $L(\p\vp)+\lot$.
Indeed, by \eqref{transport equation of Lb vp_ga}, \eqref{H_ga} and \eqref{Hc_ga}, one has
\begin{equation}\label{mu Des vp0}
	\mu\Des\vp_0=L\Lb\vp_0+\trgs\chi T\vp_0-\Hc_0+F_0.
\end{equation}
Hence,
\begin{equation}\label{mu Rc LL replaced}
	\begin{split}
		\mu\Rc_{LL}&=-LE_{\chi}-e_{\chi}-\f12pc^4\vp_0^{p-1}T\vp_0\trgs\chi\\
		           &=-LE_{\chi}-e_{\chi}+cTc\trgs\chi,
	\end{split}
\end{equation}
where
\begin{equation}\label{E chi and e chi}
	\begin{split}
		E_{\chi}&=\f12pc^4\vp_0^{p-1}\Lb\vp_0,\\
		e_{\chi}&=-L(\f12pc^4\vp_0^{p-1})\Lb\vp_0-\f12pc^4\vp_0^{p-1}(\Hc_0-F_0)-R_0.
	\end{split}
\end{equation}
Substituting \eqref{mu Rc LL replaced} into \eqref{transport equation of mu trgs chi} yields
\begin{equation}\label{modified transport equation of mu trgs chi}
	L(\mu\trgs\chi-E_{\chi})=2(L\mu-\rho^{-1}\mu)\trgs\chi+2\rho^{-2}\mu-\mu|\chic|^2+e_{\chi}.
\end{equation}

Define
\begin{equation}\label{higher order E chi}
    E_{\chi}^{\al}=\mu\ds R_i^{\al}\trgs\chi-\ds R_i^{\al}E_{\chi}.
\end{equation}
Then
\begin{equation}\label{L E_chi^0}
 	\begin{split}
 		\Lies_LE_{\chi}^0
 		&=2(\mu^{-1}L\mu-\rho^{-1})E_{\chi}^0+2(\mu^{-1}L\mu-\rho^{-1})\ds E_{\chi}-\mu\ds(|\chic|^2)+e_{\chi}^0,
 	\end{split}
\end{equation}
where
\begin{equation}\label{e_chi^0}
	\begin{split}
		e_{\chi}^0
		&=\ds e_{\chi}+(\ds L\mu-2\rho^{-1}\ds\mu)\trgs\chi-\ds\mu(L\trgs\chi+|\chic|^2-2\rho^{-2}).
	\end{split}
\end{equation}
By the induction argument on \eqref{modified transport equation of mu trgs chi}, we have
\begin{equation}\label{higher order modified transport equation of E chi}
	\begin{split}
		\Lies_LE_{\chi}^{\al}
		&=2(\mu^{-1}L\mu-\rho^{-1})E_{\chi}^{\al}+2(\mu^{-1}L\mu-\rho^{-1})\ds R_i^{\al}E_{\chi}-\mu\ds R_i^{\al}(|\chic|^2)+e_{\chi}^{\al},
	\end{split}
\end{equation}
where
\begin{equation}\label{higher order e chi}
	\begin{split}
		e_{\chi}^{\al}
		&=\Lies_{R_i}^{\al}\big[\ds e_{\chi}+(\ds L\mu-2\rho^{-1}\ds\mu)\trgs\chi-\ds\mu(L\trgs\chi+|\chic|^2-2\rho^{-2})\big]\\
		&\quad+\sum_{\be_1+\be_2=\al-1}\Lies_{R_i}^{\be_1}\big[{}^{(R_i)}\pis_L\nas E_{\chi}^{\be_2}+\nas{}^{(R_i)}\pis_LE_{\chi}^{\be_2}\big]\\
		&\quad+\sum_{\be_1+\be_2=\al-1}\Lies_{R_i}^{\be_1}\big\{(R_iL\mu-{}^{(R_i)}\pis_L^A\ds_A\mu-2\rho^{-1}R_i\mu)\ds R_i^{\be_2}\trgs\chi-R_i\mu\big[\ds R_i^{\be_2}(|\chic|^2)+\ds LR_i^{\be_2}\trgs\chi\big]\big\}.
	\end{split}
\end{equation}
Note that for any one-form $\xi$, one has
\begin{equation}\label{weighted transport equation of 1-form}
	\begin{split}
		L(|\xi|^2)
		&=\Lies_L(\gs^{AB}\xi_A\xi_B)\\
		&=-2\chic^{AB}\xi_A\xi_B-2\rho^{-1}|\xi|^2+2\Lies_L\xi\cdot\xi.
	\end{split}
\end{equation}
By taking $\xi=E_{\chi}^{\al}$ in \eqref{weighted transport equation of 1-form} and
utilizing \eqref{higher order modified transport equation of E chi}, then
\begin{equation*}
	\begin{split}
		L(|E_{\chi}^{\al}|^2)
		&=-2\chic^{AB}(E_{\chi}^{\al})_A(E_{\chi}^{\al})_B-2\rho^{-1}|E_{\chi}^{\al}|^2\\
		&\quad+\big[2(\mu^{-1}L\mu-\rho^{-1})E_{\chi}^{\al}+2(\mu^{-1}L\mu-\rho^{-1})\ds R_i^{\al}E_{\chi}-\mu\ds R_i^{\al}(|\chic|^2)+e_{\chi}^{\al}\big]\cdot E_{\chi}^{\al}.
	\end{split}
\end{equation*}
For $(\mu^{-1}L\mu-\rho^{-1})|E_{\chi}^{\al}|^2$, if $\mu\geq\f{1}{10}$, it can be bounded by $\de^{-[1-(1-\ve0)p]}|E_{\chi}^{\al}|^2$;
if $\mu<\f{1}{10}$, according to \eqref{key estimate of L mu}, the sign of $\mu^{-1}L\mu$ is negative so that
this term can be ignored. Hence,
\begin{equation}\label{transport inequality of E chi}
	L(|E_{\chi}^{\al}|)\les\de^{-[1-(1-\ve0)p]}|E_{\chi}^{\al}|+|\mu^{-1}L\mu-\rho^{-1}||\ds R_i^{\al}E_{\chi}|+|\mu\ds R_i^{\al}(|\chic|^2)|+|e_{\chi}^{\al}|.
\end{equation}
Then by applying Newton-Leibniz formula and taking the $L^2(\Si_{\t}^u)$ norm on both sides of \eqref{transport inequality of E chi}, we obtain from Gronwall inequality and  Minkowski inequality that
\begin{equation}\label{preliminary L^2 estimate of E chi}
\begin{split}
\|E_{\chi}^{\al}\|_{L^2(\Si_{\t}^u)}
&\les\de^{(1-\ve0)p-\f12}+\int_1^\t\big[\|\mu\ds R_i^{\al}(|\chic|^2)\|_{L^2(\Si_{\t'}^u)}+\|e_{\chi}^{\al}\|_{L^2(\Si_{\t'}^u)}\\
&\quad+\|\mu^{-1}L\mu-\rho^{-1}\|_{L^\infty(\Si_{\t'}^u)}\|\ds R_i^{\al}E_{\chi}\|_{L^2(\Si_{\t'}^u)}\big]\d\t'.\\
\end{split}
\end{equation}

Next, we estimate the terms in the integrand of \eqref{preliminary L^2 estimate of E chi} one by one.

\vskip 0.2 true cm
\textbf{(1-a) Estimate of $\mu\ds R_i^{\al}(|\chic|^2)$}
\vskip 0.1 true cm

According to \eqref{elliptic equation of chi}, we have
\begin{equation}\label{elliptic equation of chic}
	(\divgs\chic)_A=\ds_A\trgs\chi+\H_A,
\end{equation}
where
\begin{equation}\label{H}
	\H_A=c^{-1}\ds^Bc\chi_{AB}-c^{-1}\ds_Ac\trgs\chi.
\end{equation}
By commuting $\Lies_{R_i}^{\al}$ with $\divgs$, it follows from Lemma \ref{Lemma commutators} that
\begin{equation*}
	\begin{split}
		&\quad\Lies_{R_i}^{\al}(\divgs\chic)_A\\
		&=\Lies_{R_i}^{\al}(\gs^{BC}\nas_C\chic_{AB})\\
		&=(\divgs\Lies_{R_i}^{\al}\chic)_A+\sum_{\be_1+\be_2=\al-1}\Lies_{R_i}^{\be_1}\big[\gs^{BC}(\check{\nas}_C{}^{(R_i)}\pis_A^D)\Lies_{R_i}^{\be_2}\chic_{BD}\\
        &\quad+\gs^{BC}(\check{\nas}_C{}^{(R_i)}\pis_B^D)\Lies_{R_i}^{\be_2}\chic_{AD}+{}^{(R_i)}\pis^{BC}\nas_C\Lies_{R_i}^{\be_2}\chic_{AB}\big].
	\end{split}
\end{equation*}
Then \eqref{elliptic equation of chic} implies
\begin{equation}\label{higher order elliptic equation of chic}
	(\divgs\Lies_{R_i}^{\al}\chic)_A=\ds_AR_i^{\al}\trgs\chi+(\H^{\al})_A,
\end{equation}
where
\begin{equation}\label{higher order H}
\begin{split}	(\H^{\al})_A&=\Lies_{R_i}^{\al}\H_A-\sum_{\be_1+\be_2=\al-1}\Lies_{R_i}^{\be_1}\big[\gs^{BC}(\check{\nas}_C{}^{(R_i)}\pis_A^D)\Lies_{R_i}^{\be_2}\chic_{BD}
+\gs^{BC}(\check{\nas}_C{}^{(R_i)}\pis_B^D)\Lies_{R_i}^{\be_2}\chic_{AD}\\
&\quad+{}^{(R_i)}\pis^{BC}\nas_C\Lies_{R_i}^{\be_2}\chic_{AB}\big]
\end{split}
\end{equation}
with
\begin{equation}\label{L^2 estimate of higher order H}
	\|\H^{\al}\|_{L^2(\Si_\t^u)}\les\de^{(1-\ve0)(p-1)}\|R_i^{\leq\al+1}\vp_0\|_{L^2(\Si_\t^u)}+\de^{(1-\ve0)p}\|\Lies_{R_i}^{\leq\al}\chic\|_{L^2(\Si_\t^u)}+\de^{(1-\ve0)p}\|\Lies_{R_i}^{\leq\al}{}^{(R_j)}\pis\|_{L^2(\Si_\t^u)}.
\end{equation}
It follows from the elliptic estimate \eqref{elliptic estimate 1} and \eqref{estimate of Gaussian curvature}, \eqref{higher order elliptic equation of chic} with $\chic=(\chi-\f12\trgs\chi\gs)+\f12\trgs\chic\gs$ that
\begin{equation}\label{application of elliptic estimate 1}
	\begin{split}
		\|\mu\nas\Lies_{R_i}^{\al}\chic\|_{L^2(\Si_\t^u)}
		&\les\|\mu\divgs\Lies_{R_i}^{\al}\chic\|_{L^2(\Si_\t^u)}+\|\Lies_{R_i}^{\al}\chic\|_{L^2(\Si_\t^u)}+\|\Lies_{R_i}^{\leq\al}{}^{(R_j)}\pis\|_{L^2(\Si_\t^u)}+\|\mu\ds R_i^{\leq\al}\trgs\chi\|_{L^2(\Si_\t^u)}\\
		&\les\|\mu\ds R_i^{\leq\al}\trgs\chi\|_{L^2(\Si_\t^u)}+\|\H^{\al}\|_{L^2(\Si_\t^u)}+\|\Lies_{R_i}^{\al}\chic\|_{L^2(\Si_\t^u)}+\|\Lies_{R_i}^{\leq\al}{}^{(R_j)}\pis\|_{L^2(\Si_\t^u)}.
	\end{split}
\end{equation}
Then by \eqref{L^2 estimate of higher order H} and \eqref{application of elliptic estimate 1}, together with Lemma \ref{Lemma 1 in higher order L^2 estimates} and Proposition \ref{Proposition non-top order L^2 estimates}, we obtain
\begin{equation*}
	\begin{split}
		&\quad\|\mu\ds R_i^{\al}(|\chic|^2)\|_{L^2(\Si_\t^u)}\\
		&=\|\mu\ds\Lies_{R_i}^{\al}(\gs^{AC}\gs^{BD}\chic_{AB}\chic_{CD})\|_{L^2(\Si_\t^u)}\\
		&\les\de^{(1-\ve0)p}\|\mu\nas\Lies_{R_i}^{\al}\chic\|_{L^2(\Si_\t^u)}+\de^{(1-\ve0)p}\|\Lies_{R_i}^{\leq\al}\chic\|_{L^2(\Si_\t^u)}+\de^{(1-\ve0)2p}\|\Lies_{R_i}^{\leq\al}{}^{(R_j)}\pis\|_{L^2(\Si_\t^u)}\\
		&\les\de^{(1-\ve0)p}\left(\|\mu\ds R_i^{\al}\trgs\chi\|_{L^2(\Si_\t^u)}+\|\H^{\al}\|_{L^2(\Si_\t^u)}+\|\Lies_{R_i}^{\al}\chic\|_{L^2(\Si_\t^u)}+\|\Lies_{R_i}^{\leq\al}{}^{(R_j)}\pis\|_{L^2(\Si_\t^u)}\right)\\
		&\quad+\de^{(1-\ve0)p}\|\Lies_{R_i}^{\leq\al}\chic\|_{L^2(\Si_\t^u)}+\de^{(1-\ve0)2p}\|\Lies_{R_i}^{\leq\al}{}^{(R_j)}\pis\|_{L^2(\Si_\t^u)}\\
		&\les\de^{(1-\ve0)p}\|E_{\chi}^{\al}\|_{L^2(\Si_\t^u)}\\
		&\quad+\de^{(1-\ve0)p}\left[\de^{(1-\ve0)p+\f12}+\de^{(1-\ve0)(p-1)}\int_1^\t\left(\mu_{\min}^{-\f12}(\t')\sqrt{\Et_{1,\leq|\al|+2}(\t',u)}+\de^{(1-\ve0)p}\sqrt{\Et_{2,\leq|\al|+2}(\t',u)}\right)\d\t'\right]\\
		&\quad+\de^{(1-\ve0)(2p-1)}\left(\sqrt{\Et_{1,\leq|\al|+2}(\t,u)}+\sqrt{\Et_{2,\leq|\al|+2}(\t,u)}\right).
	\end{split}
\end{equation*}
Hence,
\begin{equation*}
	\begin{split}
		&\quad\int_1^\t\|\mu\ds R_i^{\al}(|\chic|^2)\|_{L^2(\Si_{\t'}^u)}\d\t'\\
		&\les\int_1^\t\de^{(1-\ve0)p}\|E_{\chi}^{\al}\|_{L^2(\Si_{\t'}^u)}\d\t'\\
		&\quad+\de^{(1-\ve0)p+\f32}+\de^{(1-\ve0)(2p-1)}\int_1^\t\left(\mu_{\min}^{-\f12}(\t')\sqrt{\Et_{1,\leq|\al|+2}(\t',u)}+\sqrt{\Et_{2,\leq|\al|+2}(\t',u)}\right)\d\t'.
	\end{split}
\end{equation*}

\vskip 0.2 true cm

\textbf{(1-b) Estimate of $e_{\chi}^{\al}$}

\vskip 0.1 true cm

Substituting \eqref{transport equation of mu} and \eqref{transport equation of trgs chi} into \eqref{higher order e chi},
one keeps in mind that there are two important eliminations as follows
\begin{equation*}
	\begin{split}
		&\quad(\ds L\mu-2\rho^{-1}\ds\mu)\trgs\chi-\ds\mu(L\trgs\chi+|\chic|^2-2\rho^{-2})\\
		&=\big[\mu\ds(c^{-1}Lc)-\ds(cTc)\big]\trgs\chi+\ds\mu\Rc_{LL}
	\end{split}
\end{equation*}
and
\begin{equation*}
	\begin{split}
		&\quad(R_iL\mu-{}^{(R_i)}\pis_L^A\ds_A\mu-2\rho^{-1}R_i\mu)\ds R_i^{\be_2}\trgs\chi-R_i\mu\big[\ds R_i^{\be_2}(|\chic|^2)+\ds LR_i^{\be_2}\trgs\chi\big]\\
		&=\big[\mu R_i(c^{-1}Lc)-R_i(cTc)-{}^{(R_i)}\pis_L^A\ds_A\mu\big]\ds R_i^{\be_2}\trgs\chi+R_i\mu\ds R_i^{\be_2}\Rc_{LL}+\lot.
	\end{split}
\end{equation*}
Then by \eqref{E chi and e chi}, \eqref{higher order E chi} and Proposition \ref{Proposition non-top order L^2 estimates}, one has
\begin{equation*}
	\begin{split}
		&\quad\|e_{\chi}^{\al}\|_{L^2(\Si_\t^u)}\\
		&\les\de^{(1-\ve0)p}\|E_{\chi}^{\al}\|_{L^2(\Si_\t^u)}\\
		&\quad+\de^{(1-\ve0)(p-1)-1}\|LR_i^{\al+1}\vp_0\|_{L^2(\Si_\t^u)}+\de^{(1-\ve0)(p-1)-[1-(1-\ve0)p]}\sqrt{\Et_{2,\leq|\al|+2}(\t,u)}\\
		&\quad+\de^{(1-\ve0)p-1}\left[\de^{(1-\ve0)p+\f12}+\de^{(1-\ve0)(p-1)}\int_1^\t\left(\mu_{\min}^{-\f12}(\t')\sqrt{\Et_{1,\leq|\al|+2}(\t',u)}+\de^{(1-\ve0)p}\sqrt{\Et_{2,\leq|\al|+2}(\t',u)}\right)\d\t'\right]\\
		&\quad+\de^{(1-\ve0)p-[1-(1-\ve0)p]}\left[\de^{\f12}+\de^{(1-\ve0)(p-1)}\int_1^\t\left(\mu_{\min}^{-\f12}(\t')\sqrt{\Et_{1,\leq|\al|+2}(\t',u)}+\sqrt{\Et_{2,\leq|\al|+2}(\t',u)}\right)\d\t'\right].
	\end{split}
\end{equation*}
Hence,
\begin{equation*}
	\begin{split}
		&\quad\int_1^\t\|e_{\chi}^{\al}\|_{L^2(\Si_{\t'}^u)}\d\t'\\
		&\les\int_1^\t\de^{(1-\ve0)p}\|E_{\chi}^{\al}\|_{L^2(\Si_{\t'}^u)}\d\t'\\
		&\quad+\de^{(1-\ve0)p+\f12}+\de^{(1-\ve0)(\f{p}{2}-1)-\f12}\sqrt{\int_0^u\Ft_{1,\leq|\al|+2}(\t,u')\d u'}\\
		&\quad+\de^{(1-\ve0)(p-1)-[1-(1-\ve0)p]}\int_1^\t\left(\mu_{\min}^{-\f12}(\t')\de^{1-(1-\ve0)p}\sqrt{\Et_{1,\leq|\al|+2}(\t',u)}+\sqrt{\Et_{2,\leq|\al|+2}(\t',u)}\right)\d\t'.
	\end{split}
\end{equation*}

\vskip 0.2 true cm

\textbf{(1-c) Estimate of $(\mu^{-1}L\mu-\rho^{-1})\ds R_i^{\al}E_{\chi}$}

\vskip 0.1 true cm

Thanks to \eqref{integrability of mu}, one has
\begin{equation*}
	\begin{split}
		&\quad\int_1^\t\|\mu^{-1}L\mu-\rho^{-1}\|_{L^\infty(\Si_{\t'}^u)}\|\ds R_i^{\al}E_{\chi}\|_{L^2(\Si_{\t'}^u)}\d\t'\\
		&\les\de^{-[1-(1-\ve0)p]}\int_1^{\tu}\|\ds R_i^{\al}E_{\chi}\|_{L^2(\Si_{\t'}^u)}\d\t'+\de^{-[1-(1-\ve0)p]}\int_{\tu}^{\t}\mu_{\min}^{-1}(\t')\|\ds R_i^{\al}E_{\chi}\|_{L^2(\Si_{\t'}^u)}\d\t'\\
		&\les\de^{(1-\ve0)2p+\f12}+\de^{(1-\ve0)3p-\f12}\int_{\tu}^{\t}\mu_{\min}^{-1}(\t')\d\t'+\de^{(1-\ve0)(p-1)}\mu_{\min}^{-b_{|\al|+2}}\left(\sqrt{\Et^b_{1,\leq|\al|+2}(\t,u)}+\sqrt{\Et^b_{2,\leq|\al|+2}(\t,u)}\right).
	\end{split}
\end{equation*}

Finally, substituting the estimates in \textbf{(1-a)}, \textbf{(1-b)} and \textbf{(1-c)} into \eqref{preliminary L^2 estimate of E chi}, we obtain from Gronwall inequality and \eqref{higher order E chi} that
\begin{equation*}
	\begin{split}
		&\quad\|\mu\nas\Lies_{R_i}^{\al}\chic\|_{L^2(\Si_\t^u)}+\|\mu\ds R_i^{\al}\trgs\chi\|_{L^2(\Si_\t^u)}\\
		&\les\de^{(1-\ve0)p-\f12}+\de^{(1-\ve0)3p-\f12}\int_{\tu}^{\t}\mu_{\min}^{-1}(\t')\d\t'\\
		&\quad+\de^{(1-\ve0)(\f{p}{2}-1)-\f12}\mu_{\min}^{-b_{|\al|+2}}(\t)\sqrt{\int_0^u\Ft^b_{1,\leq|\al|+2}(\t,u')\d u'}\\
		&\quad+\de^{(1-\ve0)(p-1)}\mu_{\min}^{-b_{|\al|+2}}(\t)\left(\sqrt{\Et^b_{1,\leq|\al|+2}(\t,u)}+\sqrt{\Et^b_{2,\leq|\al|+2}(\t,u)}\right).
	\end{split}
\end{equation*}
Moreover, if there is $T$ or $\rho L$ in $Z^{\al}$, by utilizing Lemma \ref{Lemma structure equations} and \eqref{[X,Y]} as in Proposition \ref{Proposition non-top order L^2 estimates}, we obtain
\begin{equation*}\label{top order L^2 estimate of chi}
	\begin{split}
		&\quad\de^{l+s[1-(1-\ve0)p]}\|\mu\nas\Lies_Z^{\al}\chic\|_{L^2(\Si_\t^u)}+\de^{l+s[1-(1-\ve0)p]}\|\mu\ds Z^{\al}\trgs\chi\|_{L^2(\Si_\t^u)}\\
		&\les\de^{(1-\ve0)p-\f12}+\de^{(1-\ve0)3p-\f12}\int_{\tu}^{\t}\mu_{\min}^{-1}(\t')\d\t'\\
		&\quad+\de^{(1-\ve0)(\f{p}{2}-1)-\f12}\mu_{\min}^{-b_{|\al|+2}}(\t)\sqrt{\int_0^u\Ft^b_{1,\leq|\al|+2}(\t,u')\d u'}\\
		&\quad+\de^{(1-\ve0)(p-1)}\mu_{\min}^{-b_{|\al|+2}}(\t)\left(\sqrt{\Et^b_{1,\leq|\al|+2}(\t,u)}+\sqrt{\Et^b_{2,\leq|\al|+2}(\t,u)}\right).
	\end{split}
\end{equation*}

\vskip 0.2 true cm

\textbf{6.2.2 Estimates on $\nas^2\mu$}

\vskip 0.1 true cm	

Similarly to the treatment for $\trgs\chi$, we use the transport equation \eqref{transport equation of mu} to estimate $\Des\mu$.
By Lemma \ref{Lemma commutators}, one has
\begin{equation}\label{transport equation of Des mu}
	L\Des\mu=(c^{-1}Lc-2\rho^{-1})\Des\mu-2\chic^{AB}\nas_{AB}^2\mu-\ds_A\trgs\chic\ds^A\mu+\mu L\Des\ln c-cT\Des c+\J,
\end{equation}
where
\begin{equation}\label{J}
	\begin{split}
		\J&=-2c\big[\ds_A(c^{-2}\mu)\chic^{AB}-\f12\ds^B(c^{-2}\mu)\trgs\chic\big]\ds_Bc-2c^{-1}\mu\chic^{AB}\nas_{AB}^2c-2\rho^{-1}c^{-1}\mu\Des c-2\ds_Ac\ds^ATc-\Des cTc\\
		&\quad+2\mu\chic^{AB}\nas_{AB}^2\ln c+2\rho^{-1}\mu\Des\ln c+2\ds_A(c^{-1}Lc)\ds^A\mu-2\H_A\ds^A\mu.
	\end{split}
\end{equation}
Hence,
\begin{equation}\label{transport equation of mu Des mu}
	\begin{split}
		L(\mu\Des\mu)
		&=L\mu\Des\mu+\mu L\Des\mu\\
		&=(L\mu+c^{-1}Lc\mu-2\rho^{-1}\mu)\uwave{\Des\mu}-2\mu\chic^{AB}\uwave{\nas_{AB}^2\mu}-\mu\ds_A\trgs\chic\ds^A\mu\\
		&\quad+\underline{L(\mu^2\Des\ln c)}-2\mu L\mu\Des\ln c-c\mu\boxed{T\Des c}+\mu\J.
	\end{split}
\end{equation}
The term with underline can be removed to the left hand side of \eqref{transport equation of mu Des mu}, while the terms with wavy line will be treated by the elliptic estimate and Gronwall inequality. The strategy to treat the boxed term in \eqref{transport equation of mu Des mu} which contains the third order derivative of the solution is to transfer it into such a form $L(\p^{\leq2}\vp)+\lot$.
Indeed, \eqref{mu Des vp0} gives that
\begin{equation}\label{mu Des c}
	\begin{split}
		\mu\Des c
		&=\mu\nas^A(-\frac{1}{2}pc^3\vp_0^{p-1}\ds_A\vp_0)\\
		&=-\frac{1}{2}pc^3\vp_0^{p-1}L\Lb\vp_0-\frac{1}{2}pc^3\vp_0^{p-1}T\vp_0\trgs\chi+\frac{1}{2}pc^3\vp_0^{p-1}(\Hc_0-F_0)+\mu\ds^A(-\frac{1}{2}pc^3\vp_0^{p-1})\ds_A\vp_0.
	\end{split}
\end{equation}
Combining \eqref{mu Des c} and \eqref{transport equation of trgs chi along T} yields
\begin{equation}\label{T Des c}
	\begin{split}
		-c\mu T\Des c
		&=-cT(\mu\Des c)+cT\mu\Des c\\
		&=\underline{L(\frac{1}{2}pc^4\vp_0^{p-1}T\Lb\vp_0)}-cTc\uwave{\Des\mu}+\K,		
	\end{split}
\end{equation}
where
\begin{equation}\label{K}
	\begin{split}
		\K&=-L(\frac{1}{2}pc^4\vp_0^{p-1})T\Lb\vp_0-\frac{1}{2}pc^4\vp_0^{p-1}{}^{(T)}\pis_L\cdot\ds\Lb\vp_0+\frac{1}{2}pc^4\vp_0^{p-1}(T\vp_0)\mathcal{I}\\
		&\quad+cT(\frac{1}{2}pc^3\vp_0^{p-1})L\Lb\vp_0+cT(\frac{1}{2}pc^3\vp_0^{p-1}T\vp_0)\trgs\chi\\
		&\quad-cT\big[\frac{1}{2}pc^3\vp_0^{p-1}(\Hc_0-F_0)+\mu\ds(-\frac{1}{2}pc^3\vp_0^{p-1})\cdot\ds\vp_0\big]+cT\mu\Des c.
	\end{split}
\end{equation}
Substituting \eqref{T Des c} into \eqref{transport equation of mu Des mu}, we arrive at
\begin{equation}\label{modified transport equation of mu Des mu}
	L(\mu\Des\mu-E_{\mu})=2(L\mu-\rho^{-1}\mu)\Des\mu-2\mu\chic^{AB}\nas_{AB}^2\mu-\mu\ds_A\trgs\chic\ds^A\mu+e_{\mu},
\end{equation}
where
\begin{equation}\label{E mu and e mu}
	\begin{split}
		E_{\mu}&=\frac{1}{2}pc^4\vp_0^{p-1}T\Lb\vp_0+\mu^2\Des\ln c,\\
		e_{\mu}&=\K+\mu\J-2\mu L\mu\Des\ln c.
	\end{split}
\end{equation}
Note that $e_{\mu}$ is composed of lower order derivative terms such as $\p^{\leq 2}\vp$, $\p^{\leq 1}\mu$ and $\check\chi$.

Define
\begin{equation}\label{higher order E mu}
	E_{\mu}^{\al}=\mu\Zu^{\al}\Des\mu-\Zu^{\al}E_{\mu}
\end{equation}
with $\Zu\in\{T,R_1,R_2,R_3\}$. Then
\begin{equation}\label{L E_mu^0}
	\begin{split}
		LE_{\mu}^0
		&=L(\mu\Des\mu-E_{\mu})\\
		&=2(\mu^{-1}L\mu-\rho^{-1})E_{\mu}^0-2\mu\chic^{AB}\nas_{AB}^2\mu-\mu\ds_A\trgs\chic\ds^A\mu+e_{\mu}^0,
	\end{split}
\end{equation}
where
\begin{equation}\label{e_mu^0}
	\begin{split}
		e_{\mu}^0
		&=e_{\mu}+2(\mu^{-1}L\mu-\rho^{-1})E_{\mu}.
	\end{split}
\end{equation}
By the induction argument on \eqref{modified transport equation of mu Des mu}, we have
\begin{equation}\label{higher order modified transport equation of E mu}
	\begin{split}
		LE_{\mu}^{\al}=2(\mu^{-1}L\mu-\rho^{-1})E_{\mu}^{\al}-2\mu\chic^{AB}\Lies_{\Zu}^{\al}\nas_{AB}^2\mu-\mu\ds_A\Zu^{\al}\trgs\chic\ds^A\mu+e_{\mu}^{\al},
	\end{split}
\end{equation}
where
\begin{equation}\label{higher order e mu}
	\begin{split}
		e_{\mu}^{\al}
		&=\Zu^{\al}e_{\mu}+2(\mu^{-1}L\mu-\rho^{-1})\Zu^{\al}E_{\mu}-2\sum_{\be_1+\be_2=\al-1}\Zu^{\be_1}\big[\Lies_{\Zu}(\mu\chic^{AB})\Lies_{\Zu}^{\be_2}\nas_{AB}^2\mu\big]\\
		&\quad+\sum_{\be_1+\be_2=\al-1}\Zu^{\be_1}[L,\Zu]E_{\mu}^{\be_2}+\sum_{\be_1+\be_2=\al-1}\Zu^{\be_1}\big\{\big[\Zu L\mu-{}^{(\Zu)}\pis_L^A\ds_A\mu-2\Zu(\rho^{-1}\mu)\big]\Zu^{\be_2}\Des\mu\big\}\\
		&\quad-\sum_{\be_1+\be_2=\al-1}\Zu^{\be_1}\big[(\Zu\mu L\Zu^{\be_2}\Des\mu)+\ds_A\Zu^{\be_2}\trgs\chic\Lies_{\Zu}(\mu\ds^A\mu)\big].
	\end{split}
\end{equation}

Note that
\begin{equation}\label{YHC-2}
	\begin{split}
		L\big(|E_{\mu}^{\al}|^2\big)=2E_{\mu}^{\al}\big[2(\mu^{-1}L\mu-\rho^{-1})E_{\mu}^{\al}
-2\mu\chic^{AB}\Lies_{\Zu}^{\al}\nas_{AB}^2\mu-\mu\ds_A\Zu^{\al}\trgs\chic\ds^A\mu+e_{\mu}^{\al}\big].
	\end{split}
\end{equation}
As in \eqref{transport inequality of E chi} and \eqref{preliminary L^2 estimate of E chi},
we have
\begin{equation}\label{preliminary L2 estimate of E mu}
	\begin{split}
		\de^l\|E_{\mu}^{\al}\|_{L^2(\Si_\t^u)}
		\les\de^{(1-\ve0)p-\f32}+\de^l\int_1^\t\left(\de^{(1-\ve0)p}\|\mu\Lies_{\Zu}^{\al}\nas^2\mu\|_{L^2(\Si_{\t'}^u)}+\|e_{\mu}^{\al}\|_{L^2(\Si_{\t'}^u)}+\|\mu\ds\Zu^{\al}\trgs\chic\|_{L^2(\Si_{\t'}^u)}\right)\d\t'.
	\end{split}
\end{equation}

Next, we estimate the terms in the integrand of \eqref{preliminary L2 estimate of E mu} one by one.

\vskip 0.2 true cm

\textbf{(2-a) Estimate of $\mu\Lies_{\Zu}^{\al}\nas^2\mu$}

\vskip 0.1 true cm

With the help of Lemma \ref{Lemma commutators}, we have from \eqref{estimate of Gaussian curvature} and the elliptic estimate \eqref{elliptic estimate 2} that
\begin{equation*}
	\begin{split}
		&\quad\de^l\|\mu\Lies_{\Zu}^{\al}\nas^2\mu\|_{L^2(\Si_{\t}^u)}\\
		&\le\de^l\|\mu\nas^2(\Zu^{\al}\mu)\|_{L^2(\Si_{\t}^u)}+\de^l\|\mu[\Lies_{\Zu}^{\al},\nas^2]\mu\|_{L^2(\Si_{\t}^u)}\\
		&\les\de^l\|\mu\Des(\Zu^{\al}\mu)\|_{L^2(\Si_{\t}^u)}+\de^l\|\ds\Zu^{\al}\mu\|_{L^2(\Si_{\t}^u)}+\de^l\|\mu[\Lies_{\Zu}^{\al},\nas^2]\mu\|_{L^2(\Si_{\t}^u)}\\
	    &\les\de^l\|E_{\mu}^{\al}\|_{L^2(\Si_{\t}^u)}+\de^{(1-\ve0)(p-1)-1}\left(\sqrt{\Et_{1,\leq|\al|+2}(\t,u)}+\sqrt{\Et_{2,\leq|\al|+2}(\t,u)}\right)\\
	    &\quad+\de^{(1-\ve0)p-1}\left[\de^{(1-\ve0)p+\f12}+\de^{(1-\ve0)(p-1)}\int_1^\t\left(\mu_{\min}^{-\f12}(\t')\sqrt{\Et_{1,\leq|\al|+2}(\t',u)}+\de^{(1-\ve0)p}\sqrt{\Et_{2,\leq|\al|+2}(\t',u)}\right)\d\t'\right]\\
	    &\quad+\de^{\f12}+\de^{(1-\ve0)(p-1)}\int_1^\t\left(\mu_{\min}^{-\f12}(\t')\sqrt{\Et_{1,\leq|\al|+2}(\t',u)}+\sqrt{\Et_{2,\leq|\al|+2}(\t',u)}\right)\d\t',
    \end{split}
\end{equation*}
where $\al'$ means that the number of $T$ appearing in $\Lies_{\Zu}^{\leq\al'}$ is at most $l-1$ and $|\al'|\leq|\al|$.

Similarly to the treatment for $\chi$, utilizing Lemma \ref{Lemma 3 in higher order L^2 estimates} again, we arrive at
\begin{equation*}
	\begin{split}
		&\quad\de^l\int_1^\t\de^{(1-\ve0)p}\|\mu\Lies_{\Zu}^{\al}\nas^2\mu\|_{L^2(\Si_{\t'}^u)}\d\t'\\
		&\les\int_1^\t\de^l\|E_{\mu}^{\al}\|_{L^2(\Si_{\t'}^u)}\d\t'+\de^{(1-\ve0)2p+\f12}+\de^{\f32}\\
		&\quad+\de^{(1-\ve0)(p-1)}\mu_{\min}^{-b_{|\al|+2}}(\t)\left(\sqrt{\Et^b_{1,\leq|\al|+2}(\t,u)}+\sqrt{\Et^b_{2,\leq|\al|+2}(\t,u)}\right).
	\end{split}
\end{equation*}

\vskip 0.2 true cm

\textbf{(2-b) Estimate of $e_{\mu}^{\al}$}

\vskip 0.1 true cm

In view of \eqref{higher order e mu}, we deal with $2(\mu^{-1}L\mu-\rho^{-1})\Zu^{\al}E_{\mu}$, $\Zu^{\al}e_{\mu}$ and the other
left terms one by one.

For $2(\mu^{-1}L\mu-\rho^{-1})\Zu^{\al}E_{\mu}$, one has from Lemma \ref{Lemma 3 in higher order L^2 estimates} and Proposition \ref{Proposition non-top order L^2 estimates} that
\begin{equation*}
	\begin{split}
		&\quad\int_1^\t\|\mu^{-1}L\mu-\rho^{-1}\|_{L^\infty(\Si_{\t'}^u)}\cdot\de^l\|\Zu^{\al}E_{\mu}\|_{L^2(\Si_{\t'}^u)}\d\t'\\
		&\les\de^{-[1-(1-\ve0)p]}\int_1^\t\mu_{\min}^{-1}(\t')\left[\de^{(1-\ve0)(p-1)-1}\left(\sqrt{\Et_{1,\leq|\al|+2}(\t',u)}+\sqrt{\Et_{2,\leq|\al|+2}(\t',u)}\right)\right.\\
		&\left.\quad+\de^{(1-\ve0)p-1}\|\Lies_{\Zu}^{\leq\al-1}{}^{(R_j)}\pis_T\|_{L^2(\Si_{\t'}^u)}\right]\d\t'\\
		&\les\de^{(1-\ve0)2p-\f12}+\de^{(1-\ve0)3p-\f32}\int_{\tu}^{\t}\mu_{\min}^{-1}(\t')\d\t'\\
		&\quad+\de^{(1-\ve0)(p-1)-1}\mu_{\min}^{-b_{|\al|+2}}(\t)\left(\sqrt{\Et^b_{1,\leq|\al|+2}(\t,u)}+\sqrt{\Et^b_{2,\leq|\al|+2}(\t,u)}\right).
	\end{split}
\end{equation*}

For $\Zu^{\al}e_{\mu}$, substituting \eqref{K} and \eqref{J} into \eqref{E mu and e mu} yields
\begin{equation*}
	e_{\mu}=-L(\f12pc^4\vp_0^{p-1})T\Lb\vp_0+cT(\f12pc^3\vp_0^{p-1}T\vp_0)\trgs\chi+cT\mu\Des c-\f12pc^4\vp_0^{p-1}{}^{(T)}\pis_L^A\ds_A\Lb\vp_0+\text{better}\ \text{terms}.
\end{equation*}
Here and below ``$\text{better}\ \text{terms}$'' stands for the terms with either lower order derivatives or higher smallness orders of $\de$,
which can be neglected in the related estimates. Then
\begin{equation*}
	\begin{split}
		&\quad\int_1^\t\de^l\|\Zu^{\al}e_{\mu}\|_{L^2(\Si_{\t'}^u)}\d\t'\\
		&\les\de^{(1-\ve0)2p-\f12}+\de^{(1-\ve0)(p-1)-1}\mu_{\min}^{-b_{|\al|+2}}(\t)\left(\sqrt{\Et^b_{1,\leq|\al|+2}(\t,u)}+\sqrt{\Et^b_{2,\leq|\al|+2}(\t,u)}\right).
	\end{split}
\end{equation*}

For the other left terms in \eqref{higher order e mu}, with the help of \eqref{transport equation of mu} and \eqref{transport equation of Des mu}, one has
\begin{equation}\label{YHC-7}
	\begin{split}
		&\quad\int_1^\t\de^l\|\text{the}\ \text{other left}\ \text{terms}\|_{L^2(\Si_{\t'}^u)}\d\t'\\
		&\les\int_1^\t\de^{(1-\ve0)p}\cdot\de^l\|E_{\mu}^{\al}\|_{L^2(\Si_{\t'}^u)}\d\t'+\text{better}\ \text{terms},
	\end{split}
\end{equation}
where the first term on the right hand side of \eqref{YHC-7}
comes from the contribution of $\dse\sum_{\be_1+\be_2=\al-1}\Zu^{\be_1}[L,\Zu]E_{\mu}^{\be_2}$
in \eqref{higher order e mu}.

Therefore,
\begin{equation*}
	\begin{split}
		&\quad\int_1^\t\de^l\|e_{\mu}^{\al}\|_{L^2(\Si_{\t'}^u)}\d\t'\\
		&\les\int_1^\t\de^{(1-\ve0)p}\cdot\de^l\|E_{\mu}^{\al}\|_{L^2(\Si_{\t'}^u)}\d\t'\\
		&\quad+\de^{(1-\ve0)2p-\f12}+\de^{(1-\ve0)3p-\f32}\int_{\tu}^{\t}\mu_{\min}^{-1}(\t')\d\t'\\
		&\quad+\de^{(1-\ve0)(p-1)-1}\mu_{\min}^{-b_{|\al|+2}}(\t)\left(\sqrt{\Et^b_{1,\leq|\al|+2}(\t,u)}+\sqrt{\Et^b_{2,\leq|\al|+2}(\t,u)}\right).
	\end{split}
\end{equation*}

\vskip 0.2 true cm

\textbf{(2-c) Estimate of $\mu\ds\Zu^{\al}\trgs\chic$}

\vskip 0.1 true cm

By \eqref{top order L^2 estimate of chi}, one has
\begin{equation*}
	\begin{split}
		&\quad\int_1^\t\de^l\|\mu\ds\Zu^{\al}\trgs\chic\|_{L^2(\Si_{\t'}^u)}\d\t'\\
		&\les\de^{1-(1-\ve0)p}\left[\de^{(1-\ve0)p-\f12}+\de^{(1-\ve0)3p-\f12}\int_{\tu}^{\t}\mu_{\min}^{-1}(\t')\d\t'\right.\\
		&\quad+\de^{(1-\ve0)(\f{p}{2}-1)-\f12}\mu_{\min}^{-b_{|\al|+2}}(\t)\sqrt{\int_0^u\Ft^b_{1,\leq|\al|+2}(\t,u')\d u'}\\
		&\left.\quad+\de^{(1-\ve0)(p-1)}\mu_{\min}^{-b_{|\al|+2}}(\t)\left(\sqrt{\Et^b_{1,\leq|\al|+2}(\t,u)}+\sqrt{\Et^b_{2,\leq|\al|+2}(\t,u)}\right)\right].
	\end{split}
\end{equation*}

Finally, substituting the estimates in \textbf{(2-a)}, \textbf{(2-b)} and \textbf{(2-c)} into \eqref{preliminary L2 estimate of E mu}, we obtain from Gronwall inequality and \eqref{higher order E mu} that
\begin{equation*}
	\begin{split}
		&\quad\de^l\|\mu\Zu^{\al}\Des\mu\|_{L^2(\Si_\t^u)}+\de^l\|\mu\Lies_{\Zu}^{\al}\nas^2\mu\|_{L^2(\Si_\t^u)}\\
		&\les\de^{(1-\ve0)p-\f32}+\de^{(1-\ve0)3p-\f32}\int_{\tu}^{\t}\mu_{\min}^{-1}(\t')\d\t'\\
		&\quad+\de^{(1-\ve0)(-\f{p}{2}-1)+\f12}\mu_{\min}^{-b_{|\al|+2}}(\t)\sqrt{\int_0^u\Ft^b_{1,\leq|\al|+2}(\t,u')\d u'}\\
		&\quad+\de^{(1-\ve0)(p-1)-1}\mu_{\min}^{-b_{|\al|+2}}(\t)\left(\sqrt{\Et^b_{1,\leq|\al|+2}(\t,u)}+\sqrt{\Et^b_{2,\leq|\al|+2}(\t,u)}\right).
	\end{split}
\end{equation*}
Moreover, if there is $\rho L$ in $Z^{\al}$, by utilizing Lemma \ref{Lemma structure equations} and \eqref{[X,Y]} as in Proposition \ref{Proposition non-top order L^2 estimates}, we arrive at
\begin{equation}\label{top order L^2 estimate of mu}
	\begin{split}
		&\quad\de^{l+s[1-(1-\ve0)p]}\|\mu Z^{\al}\Des\mu\|_{L^2(\Si_\t^u)}+\de^{l+s[1-(1-\ve0)p]}\|\mu\Lies_Z^{\al}\nas^2\mu\|_{L^2(\Si_\t^u)}\\
		&\les\de^{(1-\ve0)p-\f32}+\de^{(1-\ve0)3p-\f32}\int_{\tu}^{\t}\mu_{\min}^{-1}(\t')\d\t'\\
		&\quad+\de^{(1-\ve0)(-\f{p}{2}-1)+\f12}\mu_{\min}^{-b_{|\al|+2}}(\t)\sqrt{\int_0^u\Ft^b_{1,\leq|\al|+2}(\t,u')\d u'}\\
		&\quad+\de^{(1-\ve0)(p-1)-1}\mu_{\min}^{-b_{|\al|+2}}(\t)\left(\sqrt{\Et^b_{1,\leq|\al|+2}(\t,u)}+\sqrt{\Et^b_{2,\leq|\al|+2}(\t,u)}\right).
	\end{split}
\end{equation}

\section{Error estimates}\label{Section 7}

In this section, based on the higher order derivative $L^2$ estimates for related quantities in Section \ref{Section 6}, we are ready to deal with the last two error terms of \eqref{energy inequality}, and hence the energy estimates on \eqref{covariant wave equations} can be completed.

For convenience, we set
\begin{equation}\label{M}
	\begin{split}	
		\M_k^{\al}&=\sum_{j=0}^{|\al|-1}\big(Z_{|\al|+1}+{}^{(Z_{|\al|+1})}\la\big)\cdots\big(Z_{j+2}+{}^{(Z_{j+2})}\la\big){}^{(Z_{j+1})}\N_k^{j}+{}^{(Z_{|\al|+1})}\N_k^{|\al|},\ k=1,2,3,\\
		\M_0^{\al}&=\big(Z_{|\al|+1}+{}^{(Z_{|\al|+1})}\la\big)\cdots\big(Z_1+{}^{(Z_1)}\la\big)\Phi_{\ga}^0
	\end{split}
\end{equation}
and use $\Mc_1^{\al}$ to stand for the summation in $\M_1^{\al}$ excluding the top order derivative terms.

\subsection{Error estimates on non-top order derivative terms in \eqref{higher order commuted covariant wave equation}}\label{Section 7.1}

\vskip 0.2 true cm

\textbf{$\bullet$ Estimate of $\Mc_1^{\al}$}

\vskip 0.1 true cm

Utilizing \eqref{N1 pi Psi rho L}, \eqref{N1 pi Psi T} and \eqref{N1 pi Psi Ri} with Propositions \ref{Proposition higher order L^infty estimates}, \ref{Proposition non-top order L^2 estimates}, one has
\begin{equation*}
	\begin{split}
		&\quad\de^{l+s[1-(1-\ve0)p]}\|\Mc_1^{\al}\|_{L^2(\Si_\t^u)}\\
		&\les\de^{\f32-\ve0-[1-(1-\ve0)p]}\big(1+\de^{(1-\ve0)p-[1-(1-\ve0)p]}\big)\\
		&\quad+\big(1+\de^{(1-\ve0)p-[1-(1-\ve0)p]}\big)\mu_{\min}^{-b_{|\al|+2}}(\t)\sqrt{\Et^b_{1,\leq|\al|+2}(\t,u)}\\
&\quad +\de^{(1-\ve0)p-[1-(1-\ve0)p]}\mu_{\min}^{-b_{|\al|+2}}(\t)\sqrt{\Et^b_{2,\leq|\al|+2}(\t,u)}
	\end{split}
\end{equation*}
and hence,
\begin{equation}\label{error estimate of Mc_1}
	\begin{split}
		&\quad\de^{2l+2s[1-(1-\ve0)p]}\int_{D^{\t,u}}\Mc_1^{\al}\big(\de^{1-(1-\ve0)p}L\Psi_{\ga}^{|\al|+1}+\de\Lb\Psi_{\ga}^{|\al|+1}\big)\\
		&\les\de^{2-2\ve0}\big(\de^2+\de^{(1-\ve0)4p}\big)+\mu_{\min}^{-b_{|\al|+2}}(\t)\de^{-1}\int_0^u\de^{1-(1-\ve0)p}\Ft^b_{1,\leq|\al|+2}(\t,u')\d u'\\
		&\quad+\mu_{\min}^{-2b_{|\al|+2}}(\t)\f{\de^{2-(1-\ve0)p}+\de^{(1-\ve0)3p}}{2b_{|\al|+2}-1}\de^{1-(1-\ve0)p}\Et^b_{1,\leq|\al|+2}(\t,u)\\
		&\quad+\mu_{\min}^{-2b_{|\al|+2}}(\t)\big(\de^{2-(1-\ve0)p}+\de^{(1-\ve0)3p}\big)\de^{-[1-(1-\ve0)p]}\int_1^\t\de^{1-(1-\ve0)p}\Et^b_{1,\leq|\al|+2}(\t',u)\d\t'\\
		&\quad+\mu_{\min}^{-2b_{|\al|+2}}(\t)\f{1}{2b_{|\al|+2}-1}\de\Et^b_{2,\leq|\al|+2}(\t,u)+\mu_{\min}^{-2b_{|\al|+2}}(\t)\de^{-[1-(1-\ve0)p]}\int_1^\t\de\Et^b_{2,\leq|\al|+2}(\t',u)\d\t'.
	\end{split}
\end{equation}

\vskip 0.2 true cm

\textbf{$\bullet$ Estimate of $\M_2^{\al}$}

\vskip 0.1 true cm

It follows from \eqref{N2 pi Psi rho L}, \eqref{N2 pi Psi T} and \eqref{N2 pi Psi Ri},
Proposition \ref{Proposition higher order L^infty estimates} and \ref{Proposition non-top order L^2 estimates} that
\begin{equation*}
	\begin{split}
		&\quad\de^{l+s[1-(1-\ve0)p]}\|\M_2^{\al}\|_{L^2(\Si_\t^u)}\\
		&\les\de^{\f32-\ve0-2[1-(1-\ve0)p]}+\de^{l+s[1-(1-\ve0)p]}\|\ds\Psi_{\ga}^{|\al|+1}\|_{L^2(\Si_\t^u)}
+\de^{-[1-(1-\ve0)p]}\de^{l+s[1-(1-\ve0)p]}\|L\Psi_{\ga}^{|\al|+1}\|_{L^2(\Si_\t^u)}\\
		&\quad+\de^{(1-\ve0)p-[1-(1-\ve0)p]}\mu_{\min}^{-b_{|\al|+2}}(\t)\left(\sqrt{\Et^b_{1,\leq|\al|+2}(\t,u)}
+\sqrt{\Et^b_{2,\leq|\al|+2}(\t,u)}\right)
	\end{split}
\end{equation*}
and hence,
\begin{equation}\label{error estimate of M_2}
	\begin{split}
		&\quad\de^{2l+2s[1-(1-\ve0)p]}\int_{D^{\t,u}}\M_2^{\al}\big(\de^{1-(1-\ve0)p}L\Psi_{\ga}^{|\al|+1}+\de\Lb\Psi_{\ga}^{|\al|+1}\big)\\
		&\les\de^{2-2\ve0+(1-\ve0)2p}+\mu_{\min}^{-2b_{|\al|+2}}(\t)\de^{1+[1-(1-\ve0)p]}\Kt^b_{\leq|\al|+2}(\t,u)+\mu_{\min}^{-2b_{|\al|+2}}(\t)\de^{-1}\int_0^u\de^{1-(1-\ve0)p}\Ft^b_{1,\leq|\al|+2}(\t,u')\d u'\\
		&\quad+\mu_{\min}^{-2b_{|\al|+2}}(\t)\f{\de^{(1-\ve0)3p}}{2b_{|\al|+2}-1}\de^{1-(1-\ve0)p}\Et^b_{1,\leq|\al|+2}(\t,u)+\de^{(1-\ve0)3p-[1-(1-\ve0)p]}\int_1^\t\de^{1-(1-\ve0)p}\Et^b_{1,\leq|\al|+2}(\t',u)\d\t'\\
		&\quad+\mu_{\min}^{-2b_{|\al|+2}}(\t)\f{1}{2b_{|\al|+2}-1}\de\Et^b_{2,\leq|\al|+2}(\t,u)+\mu_{\min}^{-2b_{|\al|+2}}(\t)\de^{-[1-(1-\ve0)p]}\int_1^\t\de\Et^b_{2,\leq|\al|+2}(\t',u)\d\t'.
	\end{split}
\end{equation}

\vskip 0.2 true cm

\textbf{$\bullet$ Estimate of $\M_3^{\al}$}

\vskip 0.1 true cm

By \eqref{N3 pi Psi rho L}, \eqref{N3 pi Psi T} and \eqref{N3 pi Psi Ri},
Proposition \ref{Proposition higher order L^infty estimates} and \ref{Proposition non-top order L^2 estimates}, we have
\begin{equation*}
	\begin{split}
		&\quad\de^{l+s[1-(1-\ve0)p]}\|\M_3^{\al}\|_{L^2(\Si_\t^u)}\\
		&\les\de^{\f32-\ve0-[1-(1-\ve0)p]}+\de^{(1-\ve0)p-[1-(1-\ve0)p]}\mu_{\min}^{-b_{|\al|+2}}(\t)\left(\sqrt{\Et^b_{1,\leq|\al|+2}(\t,u)}+\sqrt{\Et^b_{2,\leq|\al|+2}(\t,u)}\right)
	\end{split}
\end{equation*}
and hence,
\begin{equation}\label{error estimate of M_3}
	\begin{split}
		&\quad\de^{2l+2s[1-(1-\ve0)p]}\int_{D^{\t,u}}\M_3^{\al}\big(\de^{1-(1-\ve0)p}L\Psi_{\ga}^{|\al|+1}+\de\Lb\Psi_{\ga}^{|\al|+1}\big)\\
		&\les\de^{4-2\ve0}+\mu_{\min}^{-2b_{|\al|+2}}(\t)\de^{-1}\int_0^u\de^{1-(1-\ve0)p}\Ft^b_{1,\leq|\al|+2}(\t,u')\d u'\\
		&\quad+\mu_{\min}^{-2b_{|\al|+2}}(\t)\f{\de^{(1-\ve0)3p}}{2b_{|\al|+2}-1}\de^{1-(1-\ve0)p}\Et^b_{1,\leq|\al|+2}(\t,u)+\de^{(1-\ve0)3p-[1-(1-\ve0)p]}\int_1^\t\de^{1-(1-\ve0)p}\Et^b_{1,\leq|\al|+2}(\t',u)\d\t'\\
		&\quad+\mu_{\min}^{-2b_{|\al|+2}}(\t)\f{1}{2b_{|\al|+2}-1}\de\Et^b_{2,\leq|\al|+2}(\t,u)+\mu_{\min}^{-2b_{|\al|+2}}(\t)\de^{-[1-(1-\ve0)p]}\int_1^\t\de\Et^b_{2,\leq|\al|+2}(\t',u)\d\t'.
	\end{split}
\end{equation}

\vskip 0.2 true cm

\textbf{$\bullet$ Estimate of $\M_0^{\al}$}

\vskip 0.1 true cm

Applying \eqref{covariant wave equations}, \eqref{F_ga}, Proposition \ref{Proposition higher order L^infty estimates}
and \ref{Proposition non-top order L^2 estimates}, one has
\begin{equation*}
	\begin{split}
		&\quad\de^{l+s[1-(1-\ve0)p]}\|\M_0^{\al}\|_{L^2(\Si_\t^u)}\\
		&\les\de^{\f32-\ve0+(1-\ve0)p-2[1-(1-\ve0)p]}+\de^{(1-\ve0)p-1}\de^{l+s[1-(1-\ve0)p]}\|L\Psi_{\ga}^{|\al|+1}\|_{L^2(\Si_\t^u)}\\
		&\quad+\de^{(1-\ve0)p-[1-(1-\ve0)p]}\mu_{\min}^{-b_{|\al|+2}}(\t)\left(\sqrt{\Et^b_{1,\leq|\al|+2}(\t,u)}+\de^{(1-\ve0)p}\sqrt{\Et^b_{2,\leq|\al|+2}(\t,u)}\right)
	\end{split}
\end{equation*}
and hence,
\begin{equation}\label{error estimate of M_0}
	\begin{split}
		&\quad\de^{2l+2s[1-(1-\ve0)p]}\int_{D^{\t,u}}\M_0^{\al}\big(\de^{1-(1-\ve0)p}L\Psi_{\ga}^{|\al|+1}+\de\Lb\Psi_{\ga}^{|\al|+1}\big)\\
		&\les\de^{2-2\ve0+(1-\ve0)2p}+\mu_{\min}^{-2b_{|\al|+2}}(\t)\de^{-1}\int_0^u\de^{1-(1-\ve0)p}\Ft^b_{1,\leq|\al|+2}(\t,u')\d u'\\
		&\quad+\mu_{\min}^{-2b_{|\al|+2}}(\t)\f{\de^{(1-\ve0)3p}}{2b_{|\al|+2}-1}\de^{1-(1-\ve0)p}\Et^b_{1,\leq|\al|+2}(\t,u)+\de^{(1-\ve0)3p-[1-(1-\ve0)p]}\int_1^\t\de^{1-(1-\ve0)p}\Et^b_{1,\leq|\al|+2}(\t',u)\d\t'\\
		&\quad+\mu_{\min}^{-2b_{|\al|+2}}(\t)\f{1}{2b_{|\al|+2}-1}\de\Et^b_{2,\leq|\al|+2}(\t,u)+\mu_{\min}^{-2b_{|\al|+2}}(\t)\de^{-[1-(1-\ve0)p]}\int_1^\t\de\Et^b_{2,\leq|\al|+2}(\t',u)\d\t'.
	\end{split}
\end{equation}

\subsection{Error estimates on top order derivative terms in \eqref{higher order commuted covariant wave equation}}\label{Section 7.2}

Note that for the most difficult term ${}^{(Z)}\mathscr{N}_1^{0}$, the number of the top order derivatives is $|\al|$,
which means that there will be some terms containing the $(|\al|+1)^{\text{th}}$ order derivatives of the deformation tensors.
In this case, $\Et_{i,\leq|\al|+3}$ will appear in the right hand side of \eqref{energy inequality} if one only adopts
Proposition \ref{Proposition non-top order L^2 estimates}. This leads to that the direct estimate on ${}^{(Z)}\mathscr{N}_1^{0}$
can not be absorbed by the left hand
side of \eqref{energy inequality}. To overcome this difficulty, we will carefully examine the expression of ${}^{(Z)}\mathscr{N}_1^{j}$
and apply the estimates in Section \ref{Section 6} to deal with the top order derivatives of $\chi$ and $\mu$. In fact,
by substituting ${}^{(Z)}\mathscr{N}_1^{j}$ into \eqref{higher order commuted covariant wave equation}, the resulting terms $R_i^A(\ds_AZ^{\al}\trgs\chi)T\Psi_{\ga}^0$ from \eqref{principal term of N1 Ri} and $Z^{\al}\Des\mu T\Psi_{\ga}^0$
from \eqref{principal term of N1 T} need to be treated especially.

\vskip 0.2 true cm

\textbf{(1) Treatment on $\int_{D^{\t,u}}R_i^{\al+1}\trgs\chic T\vp_{\ga}\big(\de^{1-(1-\ve0)p}LR_i^{\al+1}\vp_{\ga}+\de\Lb R_i^{\al+1}\vp_{\ga}\big)$}

\vskip 0.1 true cm
Let
\begin{equation*}\label{YHC-3}
\begin{split}
&\int_{D^{\t,u}}R_i^{\al+1}\trgs\chic T\vp_{\ga}\big(\de^{1-(1-\ve0)p}LR_i^{\al+1}\vp_{\ga}+\de\Lb R_i^{\al+1}\vp_{\ga}\big)
=A_1+A_2.
\end{split}
\end{equation*}

At first, we treat $A_1$. Integrating by parts yields
\begin{equation*}
\begin{split}
A_1&=\de^{1-(1-\ve0)p}\int_{D^{\t,u}}R_i^{\al+1}\trgs\chic T\vp_{\ga}\cdot LR_i^{\al+1}\vp_{\ga}\\
&=\de^{1-(1-\ve0)p}\int_1^\t\int_0^u\int_{S_{\t',u'}}(L+\trgs\chi)(R_i^{\al+1}\trgs\chic T\vp_\ga R_i^{\al+1}\vp_\ga)\\
&\quad-\de^{1-(1-\ve0)p}\int_1^\t\int_0^u\int_{S_{\t',u'}}(L+\trgs\chi)(R_i^{\al+1}\trgs\chic)T\vp_\ga R_i^{\al+1}\vp_\ga\\
&\quad-\de^{1-(1-\ve0)p}\int_1^\t\int_0^u\int_{S_{\t',u'}}R_i^{\al+1}\trgs\chic LT\vp_\ga R_i^{\al+1}\vp_\ga\\
&=I_{\chi}^1+I_{\chi}^2+I_{\chi}^3.
\end{split}
\end{equation*}

\vskip 0.2 true cm

\textbf{$\bullet$ Estimate of $I_{\chi}^1$}

\vskip 0.1 true cm

It follows from the integration by parts that
\begin{equation*}
	\begin{split}
		I_{\chi}^1
		&=\de^{1-(1-\ve0)p}\int_1^\t\f{\p}{\p\t'}\int_{\Si_{\t'}^u}R_i^{\al+1}\trgs\chic T\vp_\ga R_i^{\al+1}\vp_\ga\\
		&=-\de^{1-(1-\ve0)p}\int_{\Si_\t^u}R_i^{\al}\trgs\chic T\vp_\ga R_i^{\al+2}\vp_\ga\\
		&\quad-\de^{1-(1-\ve0)p}\int_{\Si_\t^u}R_i^{\al}\trgs\chic\big(R_iT\vp_\ga+\f12T\vp_\ga\trgs{}^{(R_i)}\pis\big)R_i^{\al+1}\vp_\ga\\
		&\quad-\de^{1-(1-\ve0)p}\int_{\Si_1^u}R_i^{\al+1}\trgs\chic T\vp_\ga R_i^{\al+1}\vp_\ga\\
		&=I_{\chi}^{11}+I_{\chi}^{12}+I_{\chi}^{13}.
	\end{split}
\end{equation*}

For $I_{\chi}^{11}$, by Proposition \ref{Proposition non-top order L^2 estimates} and Lemma \ref{Lemma 3 in higher order L^2 estimates},
we arrive at
\begin{equation}\label{error estimate of I_chi^11}
	\begin{split}
		|I_{\chi}^{11}|
		&\les\mu_{\min}^{-2b_{|\al|+2}}(\t)\ep\de^{1-(1-\ve0)p}\Et^b_{1,\leq|\al|+2}(\t,u)+\mu_{\min}^{-1}(\t)\f{\de^{(1-\ve0)p}}{\ep}\de^{2-2\ve0}\\
		&\quad+\mu_{\min}^{-2b_{|\al|+2}}(\t)\f{1}{b_{|\al|+2}-\f12}\de^{1-(1-\ve0)p}\Et^b_{1,\leq|\al|+2}(\t,u)\\
		&\quad+\mu_{\min}^{-2b_{|\al|+2}}(\t)\ep\de^{1-(1-\ve0)p}\Et^b_{1,\leq|\al|+2}(\t,u)+\mu_{\min}^{-2b_{|\al|+2}}(\t)\f{1}{\ep}\de^{-[1-(1-\ve0)p]}\int_1^\t\de^{1-(1-\ve0)p}\Et^b_{1,\leq|\al|+2}(\t',u)\d\t'\\
		&\quad+\mu_{\min}^{-2b_{|\al|+2}}(\t)\f{1}{b_{|\al|+2}-1}\de^{1-(1-\ve0)p}\Et^b_{1,\leq|\al|+2}(\t,u)+\mu_{\min}^{-2b_{|\al|+2}}(\t)\f{\de^{(1-\ve0)p}}{b_{|\al|+2}-1}\de\Et^b_{2,\leq|\al|+2}(\t,u)\\
		&\quad+\mu_{\min}^{-2b_{|\al|+2}}(\t)\ep\de^{1-(1-\ve0)p}\Et^b_{1,\leq|\al|+2}(\t,u)+\mu_{\min}^{-2b_{|\al|+2}}(\t)\f{\de^{(1-\ve0)p}}{\ep}\de^{-[1-(1-\ve0)p]}\int_1^\t\de\Et^b_{2,\leq|\al|+2}(\t',u)\d\t',
	\end{split}
\end{equation}
where $\ep>0$ is a small constant arising from the inequality $ab\les\ep a^2+\f{1}{\ep}b^2$.

For $I_{\chi}^{12}$, it is a lower order derivative term compared with $I_{\chi}^{11}$, which can be estimated similarly.

For $I_{\chi}^{13}$, due to $\t=1$, one has
\begin{equation}\label{error estimate of I_chi^13}
	\begin{split}
		|I_{\chi}^{13}|
		&\les\de^{1-(1-\ve0)p}\de^{-\ve0}\|R_i^{\al+1}\trgs\chic\|_{L^2(\Si_1^u)}\|R_i^{\al+1}\vp_\ga\|_{L^2(\Si_1^u)}\\
		&\les\de^{1-(1-\ve0)p}\de^{-\ve0}\cdot\de^{(1-\ve0)p}\de^{\f12}\cdot\de^{1-\ve0}\de^{\f12}\\
		&=\de^{3-2\ve0}.
	\end{split}
\end{equation}

\vskip 0.2 true cm

\textbf{$\bullet$ Estimate of $I_{\chi}^2$}

\vskip 0.1 true cm

By integration by parts, we have that
\begin{equation*}
	\begin{split}
		I_{\chi}^2
		&=\de^{1-(1-\ve0)p}\int_1^\t\int_0^u\int_{S_{\t',u'}}(L+\trgs\chi)(R_i^{\al}\trgs\chic)T\vp_{\ga}R_i^{\al+2}\vp_{\ga}\\
		&\quad+\de^{1-(1-\ve0)p}\int_1^\t\int_0^u\int_{S_{\t',u'}}(L+\trgs\chi)(R_i^{\al}\trgs\chic)\big(R_iT\vp_{\ga}+\f12T\vp_{\ga}\trgs{}^{(R_i)}\pis\big)R_i^{\al+1}\vp_{\ga}\\
		&\quad-\de^{1-(1-\ve0)p}\int_1^\t\int_0^u\int_{S_{\t',u'}}{}^{(R_i)}\pis_L^A\ds_AR_i^{\al}\trgs\chic T\vp_{\ga}R_i^{\al+1}\vp_{\ga}\\
		&\quad+\de^{1-(1-\ve0)p}\int_1^\t\int_0^u\int_{S_{\t',u'}}R_i\trgs\chi R_i^{\al}\trgs\chic T\vp_{\ga}R_i^{\al+1}\vp_{\ga}\\
		&=I_{\chi}^{21}+I_{\chi}^{22}+I_{\chi}^{23}+I_{\chi}^{24}.
	\end{split}
\end{equation*}

For $I_{\chi}^{21}$, by Proposition \ref{Proposition non-top order L^2 estimates}, Lemma \ref{Lemma 3 in higher order L^2 estimates}, \eqref{[X,Y]} and \eqref{transport equation of trgs chi}, one has
\begin{equation}\label{error estimate of I_chi^21}
	\begin{split}
		|I_{\chi}^{21}|
		&\les\mu_{\min}^{-2b_{|\al|+2}}(\t)\f{1}{b_{|\al|+2}-\f12}\de^{1-(1-\ve0)p}\Et^b_{1,\leq|\al|+2}(\t,u)+\f{\de^{1+[1-(1-\ve0)p]}}{b_{|\al|+2}-\f12}\de^{2-2\ve0}\\
		&\quad+\mu_{\min}^{-2b_{|\al|+2}}(\t)\f{1}{b_{|\al|+2}-\f12}\de^{1-(1-\ve0)p}\Et^b_{1,\leq|\al|+2}(\t,u)+\f{\de^{(1-\ve0)3p}}{b_{|\al|+2}-\f12}\de^{2-2\ve0}\\
		&\quad+\de^{2-2\ve0}+\mu_{\min}^{-2b_{|\al|+2}}(\t)\de\int_1^\t\de^{1-(1-\ve0)p}\Et^b_{1,\leq|\al|+2}(\t',u)\d\t'\\
		&\quad+\de^{2-2\ve0}+\mu_{\min}^{-2b_{|\al|+2}}(\t)\de^{(1-\ve0)3p-[1-(1-\ve0)p]}\int_1^\t\de^{1-(1-\ve0)p}\Et^b_{1,\leq|\al|+2}(\t',u)\d\t'\\
		&\quad+\mu_{\min}^{-2b_{|\al|+2}}(\t)\f{1}{2b_{|\al|+2}}\de^{1-(1-\ve0)p}\Et^b_{1,\leq|\al|+2}(\t,u)\\
		&\quad+\mu_{\min}^{-2b_{|\al|+2}}(\t)\de^{-[1-(1-\ve0)p]}\int_1^\t\de^{1-(1-\ve0)p}\Et^b_{1,\leq|\al|+2}(\t',u)\d\t'\\
		&\quad+\mu_{\min}^{-2b_{|\al|+2}}(\t)\f{\de^{(1-\ve0)\f12p}}{2b_{|\al|+2}-\f12}\de\Et^b_{2,\leq|\al|+2}(\t,u)\\
		&\quad+\mu_{\min}^{-2b_{|\al|+2}}(\t)\de^{(1-\ve0)\f12p-[1-(1-\ve0)p]}\int_1^\t\de\Et^b_{2,\leq|\al|+2}(\t',u)\d\t'.
	\end{split}
\end{equation}
Note that  $I_{\chi}^{22}$ is a lower order derivative term compared with $I_{\chi}^{21}$,
$I_{\chi}^{23}$ is a better term with higher smallness order of $\de$ compared with $I_{\chi}^3$, and $I_{\chi}^{24}$ is a
lower order derivative term compared with $I_{\chi}^{23}$. Then these terms can be treated conveniently.

\vskip 0.2 true cm

\textbf{$\bullet$ Estimate of $I_{\chi}^3$}

\vskip 0.1 true cm

For $I_{\chi}^3$, by \eqref{top order L^2 estimate of chi} and Lemma \ref{Lemma 3 in higher order L^2 estimates}, one has
\begin{equation}\label{error estimate of I_chi^3}
    \begin{split}
|I_{\chi}^3|
&\les\mu_{\min}^{-2b_{|\al|+2}}(\t)\f{1}{b_{|\al|+2}+1}\de^{1-(1-\ve0)p}\Et^b_{1,\leq|\al|+2}(\t,u)
+\mu_{\min}^{-2}(\t)\f{\de^{(1-\ve0)p}}{b_{|\al|+2}+1}\de^{2-2\ve0}\\
&\quad+\de^{2-2\ve0}
+\mu_{\min}^{-2b_{|\al|+2}}(\t)\de^{(1-\ve0)p-[1-(1-\ve0)p]}\int_1^\t\de^{1-(1-\ve0)p}\Et^b_{1,\leq|\al|+2}(\t',u)\d\t'\\
&\quad+\mu_{\min}^{-2b_{|\al|+2}}(\t)\f{1}{b_{|\al|+2}+1}\de\Et^b_{2,\leq|\al|+2}(\t,u)+\mu_{\min}^{-2}(\t)\f{1}{b_{|\al|+2}+1}\de^{2-2\ve0}\\
&\quad+\de^{2-2\ve0}+\mu_{\min}^{-2b_{|\al|+2}}(\t)\de^{-[1-(1-\ve0)p]}\int_1^\t\de\Et^b_{2,\leq|\al|+2}(\t',u)\d\t'\\
&\quad+\mu_{\min}^{-2b_{|\al|+2}}(\t)\de^{-1}\int_0^u\de^{1-(1-\ve0)p}\Ft^b_{1,\leq|\al|+2}(\t,u')\d u'\\
&\quad+\mu_{\min}^{-2b_{|\al|+2}}(\t)\f{\de^{(1-\ve0)p}}{2b_{|\al|+2}}\de^{1-(1-\ve0)p}\Et^b_{1,\leq|\al|+2}(\t,u)\\
&\quad +\mu_{\min}^{-2b_{|\al|+2}}(\t)\de^{(1-\ve0)p-[1-(1-\ve0)p]}\int_1^\t\de^{1-(1-\ve0)p}\Et^b_{1,\leq|\al|+2}(\t',u)\d\t'\\
&\quad+\mu_{\min}^{-2b_{|\al|+2}}(\t)\f{1}{2b_{|\al|+2}}\de\Et^b_{2,\leq|\al|+2}(\t,u)
+\mu_{\min}^{-2b_{|\al|+2}}(\t)\de^{-[1-(1-\ve0)p]}\int_1^\t\de\Et^b_{2,\leq|\al|+2}(\t',u)\d\t'.
	\end{split}
\end{equation}

\vskip 0.2 true cm

Secondly, we start to treat $A_2$.
Since $\|\Lb R_i^{\al+1}\vp_{\ga}\|_{L^2(\Si_\t^u)}\les\mu_{\min}^{-b_{|\al|+2}}(\t)\sqrt{\Et^b_{2,\leq|\al|+2}(\t,u)}$ and $|T\vp_{\ga}|\les\de^{-\ve0}$, $A_2$ is a better term with higher smallness order of $\de$ compared with $I_{\chi}^3$, which can be analogously estimated.

\vskip 0.2 true cm

\textbf{(2) Treatment on $\de^{2l}\int_{D^{\t,u}}R_i^{\be}T^{l-1}\Des\mu T\vp_{\ga}\big(\de^{1-(1-\ve0)p}LR_i^{\al+1}\vp_{\ga}+\de\Lb R_i^{\al+1}\vp_{\ga}\big),\ |\be|+l=|\al|+1$}

\vskip 0.1 true cm
Let
\begin{equation*}\label{YHC-3}
\begin{split}
&\de^{2l}\int_{D^{\t,u}}R_i^{\be}T^{l-1}\Des\mu T\vp_{\ga}\big(\de^{1-(1-\ve0)p}LR_i^{\al+1}\vp_{\ga}+\de\Lb R_i^{\al+1}\vp_{\ga}\big)=B_1+B_2.
\end{split}
\end{equation*}

At first, we treat the troublesome term $B_1$. By integration by parts, one has
\begin{equation*}
	\begin{split}
		B_1&=\de^{2l+1-(1-\ve0)p}\int_{D^{\t,u}}R_i^{\be}T^{l-1}\Des\mu T\vp_{\ga}\cdot LR_i^{\be}T^l\vp_{\ga}\\
		&=\de^{2l+1-(1-\ve0)p}\int_1^\t\int_0^u\int_{S_{\t',u}}(L+\trgs\chi)(R_i^{\be}T^{l-1}\Des\mu T\vp_{\ga}R_i^{\be}T^l\vp_{\ga})\\
		&\quad-\de^{2l+1-(1-\ve0)p}\int_1^\t\int_0^u\int_{S_{\t',u}}(L+\trgs\chi)(R_i^{\be}T^{l-1}\Des\mu)T\vp_{\ga}R_i^{\be}T^l\vp_{\ga}\\
		&\quad-\de^{2l+1-(1-\ve0)p}\int_1^\t\int_0^u\int_{S_{\t',u}}R_i^{\be}T^{l-1}\Des\mu LT\vp_{\ga}R_i^{\be}T^l\vp_{\ga}\\
		&=I_{\mu}^1+I_{\mu}^2+I_{\mu}^3.
	\end{split}
\end{equation*}

\vskip 0.1 true cm

\textbf{$\bullet$ Estimate of $I_{\mu}^1$}

\vskip 0.1 true cm
It follows from the integration by parts that
\begin{equation*}
	\begin{split}
		I_{\mu}^1
		&=\de^{2l+1-(1-\ve0)p}\int_1^\t\f{\p}{\p\t'}\int_{\Si_{\t'}^u}R_i^{\be}T^{l-1}\Des\mu T\vp_{\ga}R_i^{\be}T^l\vp_{\ga}\\
		&=-\de^{2l+1-(1-\ve0)p}\int_{\Si_\t^u}R_i^{\be-1}T^{l-1}\Des\mu T\vp_{\ga}R_i^{\be+1}T^l\vp_{\ga}\\
		&\quad-\de^{2l+1-(1-\ve0)p}\int_{\Si_\t^u}R_i^{\be-1}T^{l-1}\Des\mu\big(R_iT\vp_{\ga}+\f12T\vp_{\ga}\trgs{}^{(R_i)}\pis\big)R_i^{\be}T^l\vp_{\ga}\\
		&\quad-\de^{2l+1-(1-\ve0)p}\int_{\Si_1^u}R_i^{\be}T^{l-1}\Des\mu T\vp_{\ga}R_i^{\be}T^l\vp_{\ga}\\
		&=I_{\mu}^{11}+I_{\mu}^{12}+I_{\mu}^{13}.
	\end{split}
\end{equation*}

For $I_{\mu}^{11}$, by \eqref{top order L^2 estimate of chi} and Lemma \ref{Lemma 3 in higher order L^2 estimates}, we have
\begin{equation}\label{error estimate of I_mu^11}
	\begin{split}
		|I_{\mu}^{11}|
		&\les\mu_{\min}^{-1}(\t)\de^{2-2\ve0}+\mu_{\min}^{-2b_{|\al|+2}}(\t)\de^{1+[1-(1-\ve0)p]}\de^{1-(1-\ve0)p}\Et^b_{1,\leq|\al|+2}(\t,u)\\
		&\quad+\mu_{\min}^{-2b_{|\al|+2}}(\t)\f{\de}{b_{|\al|+2}-\f12}\de^{1-(1-\ve0)p}\Et^b_{1,\leq|\al|+2}(\t,u)\\
		&\quad+\mu_{\min}^{-2b_{|\al|+2}}(\t)\de^{1-[1-(1-\ve0)p]}\int_1^\t\de^{1-(1-\ve0)p}\Et^b_{1,\leq|\al|+2}(\t',u)\d\t'
+\mu_{\min}^{-2b_{|\al|+2}}(\t)\de\de^{1-(1-\ve0)p}\Et^b_{1,\leq|\al|+2}(\t,u)\\
		&\quad+\mu_{\min}^{-2b_{|\al|+2}}(\t)\f{\de^{\f12+\f12[1-(1-\ve0)p]}}{b_{|\al|+2}-1}\de\Et^b_{2,\leq|\al|+2}(\t,u)
+\mu_{\min}^{-2b_{|\al|+2}}(\t)\f{\de^{\f12+\f12[1-(1-\ve0)p]}}{b_{|\al|+2}-1}\de^{1-(1-\ve0)p}\Et^b_{1,\leq|\al|+2}(\t,u)\\
		&\quad+\mu_{\min}^{-2b_{|\al|+2}}(\t)\de^{\f12}\int_1^\t\de\Et^b_{2,\leq|\al|+2}(\t',u)\d\t'
+\mu_{\min}^{-2b_{|\al|+2}}(\t)\de^{\f12}\de^{1-(1-\ve0)p}\Et^b_{1,\leq|\al|+2}(\t,u).
	\end{split}
\end{equation}

Note that $I_{\mu}^{12}$ is a lower order derivative term compared with $I_{\mu}^{11}$, which can be treated similarly.

For $I_{\mu}^{13}$, due to $\t=1$, one has
\begin{equation}\label{error estimate of I_mu^13}
	\begin{split}
		|I_{\mu}^{13}|
		&\les\de^{1-\ve0+1-(1-\ve0)p}\|\de^{l-1}R_i^{\be}T^{l-1}\Des\mu\|_{L^2(\Si_1^u)}\|\de^lR_i^{\be}T^l\vp_\ga\|_{L^2(\Si_1^u)}\\
		&\les\de^{1-\ve0+1-(1-\ve0)p}\cdot\de^{\f12}\cdot\de^{1-\ve0}\de^{\f12}\\
		&=\de^{3-2\ve0+1-(1-\ve0)p}.
	\end{split}
\end{equation}

\vskip 0.1 true cm

\textbf{$\bullet$ Estimate of $I_{\mu}^2$}

\vskip 0.1 true cm

It follows from the integration by parts that
\begin{equation*}
	\begin{split}
		I_{\mu}^2
		&=\de^{2l+1-(1-\ve0)p}\int_1^\t\int_0^u\int_{S_{\t',u'}}(L+\trgs\chi)(R_i^{\be-1}T^{l-1}\Des\mu)T\vp_{\ga}R_i^{\be+1}T^l\vp_{\ga}\\
		&\quad+\de^{2l+1-(1-\ve0)p}\int_1^\t\int_0^u\int_{S_{\t',u'}}(L+\trgs\chi)(R_i^{\be-1}T^{l-1}\Des\mu)\big(R_iT\vp_{\ga}+\f12T\vp_{\ga}\trgs{}^{(R_i)}\pis\big)R_i^{\be}T^l\vp_{\ga}\\
		&\quad-\de^{2l+1-(1-\ve0)p}\int_1^\t\int_0^u\int_{S_{\t',u'}}{}^{(R_i)}\pis_L^A\ds_AR_i^{\be-1}T^{l-1}\Des\mu T\vp_{\ga}R_i^{\be}T^l\vp_{\ga}\\
		&\quad+\de^{2l+1-(1-\ve0)p}\int_1^\t\int_0^u\int_{S_{\t',u'}}R_i\trgs\chi R_i^{\be-1}T^{l-1}\Des\mu T\vp_{\ga}R_i^{\be}T^l\vp_{\ga}\\
		&=I_{\mu}^{21}+I_{\mu}^{22}+I_{\mu}^{23}+I_{\mu}^{24}.
	\end{split}
\end{equation*}

For $I_{\mu}^{21}$, by Proposition \ref{Proposition non-top order L^2 estimates}, Lemma \ref{Lemma 3 in higher order L^2 estimates}, \eqref{[X,Y]} and \eqref{transport equation of Des mu}, the resulting estimate is the same as $I_{\chi}^{21}$.

In addition, note that $I_{\mu}^{22}$ is a lower order derivative term compared with $I_{\mu}^{21}$,
$I_{\mu}^{23}$ is a better term with higher smallness order of $\de$ compared with $I_{\mu}^3$, and $I_{\mu}^{24}$ is a lower order derivative term compared with $I_{\mu}^{23}$.
Then these terms can be estimated easily.

\vskip 0.2 true cm

\textbf{$\bullet$ Estimate of $I_{\mu}^3$}

\vskip 0.1 true cm

For $I_{\mu}^3$, by \eqref{top order L^2 estimate of mu} and Lemma \ref{Lemma 3 in higher order L^2 estimates},
the resulting estimate is the same as $I_{\chi}^3$.

\vskip 0.2 true cm

Secondly, we treat $B_2$.
Since $\|\Lb R_i^{\al+1}\vp_{\ga}\|_{L^2(\Si_\t^u)}\les\mu_{\min}^{-b_{|\al|+2}}(\t)\sqrt{\Et^b_{2,\leq|\al|+2}(\t,u)}$ and $|T\vp_{\ga}|\les\de^{-\ve0}$, $B_2$ is a better term with higher smallness order of $\de$ compared with $I_{\mu}^3$, which can be estimated easily.

All in all, combining \eqref{error estimate of Mc_1}-\eqref{error estimate of M_0} and \eqref{error estimate of I_chi^11}-\eqref{error estimate of I_mu^13}, together with Gronwall inequality (choosing $b_{k}$ sufficiently large), we finally arrive at
\begin{equation}\label{top order energy inequality}
	\begin{split}
	    &\quad\de^{1-(1-\ve0)p}\Et^b_{1,\leq|\al|+2}(\t,u)+\de^{1-(1-\ve0)p}\Ft^b_{1,\leq|\al|+2}(\t,u)\\
	    &\quad+\de\Et^b_{2,\leq|\al|+2}(\t,u)+\de\Ft^b_{2,\leq|\al|+2}(\t,u)+\Kt^b_{\leq|\al|+2}(\t,u)\\
	    &\les\de^{2-2\ve0}.
	\end{split}	
\end{equation}

\section{Shock formation}\label{Section 8}

\subsection{Descent scheme}\label{Section 8.1}

Note that the modified energies and fluxes in \eqref{higher order energy and flux} go to zero
when $\mu_{\min}\rightarrow 0$, which could not be utilized directly to close the bootstrap assumptions \eqref{bootstrap assumptions}
by Sobolev embedding theorem. As in \cite{Ch2} and \cite{M-Y}-\cite{Sp}, by the descent scheme, when the orders of derivatives decrease,
the powers of $\mu_{\min}$ needed in \eqref{higher order energy and flux}  also decrease.
Therefore, through some finite steps, the related power of $\mu_{\min}$ could
become zero. This implies that the standard energy inequality without any $\mu_{\min}$-weights can be
eventually obtained.

Indeed, by \eqref{top order L^2 estimate of chi}, \eqref{top order L^2 estimate of mu} and
Lemma \ref{Lemma 3 in higher order L^2 estimates}, we have
\begin{equation}\label{next-to-top order error estimate 1}
	\begin{split}
		&\quad\de^{2l+2s[1-(1-\ve0)p]}\int_{D^{\t,u}}T\vp_{\ga}Z^{\al}\trgs\chic\big(\de^{1-(1-\ve0)p}LZ^{\al}\vp_{\ga}+\de\Lb Z^{\al}\vp_{\ga}\big)\\
		&\les\de^{4-2\ve0}+\mu_{\min}^{-2b_{|\al|+1}}(\t)\de^{2-2\ve0}+\mu_{\min}^{-2b_{|\al|+1}}(\t)\de^{-1}\int_0^u\de^{1-(1-\ve0)p}\Ft^b_{1,\leq|\al|+1}(\t,u')\d u'\\
		&\quad+\mu_{\min}^{-2b_{|\al|+1}}(\t)\f{1}{2b_{|\al|+1}-1}\de\Et^b_{2,\leq|\al|+1}(\t,u)+\mu_{\min}^{-2b_{|\al|+1}}(\t)\de^{-[1-(1-\ve0)p]}\int_1^\t\de\Et^b_{2,\leq|\al|+1}(\t',u)\d\t'
	\end{split}
\end{equation}
and
\begin{equation}\label{next-to-top order error estimate 2}
	\begin{split}
		&\quad\de^{2l+2s[1-(1-\ve0)p]}\int_{D^{\t,u}}T\vp_{\ga}Z^{\al-1}\Des\mu\big(\de^{1-(1-\ve0)p}LZ^{\al}\vp_{\ga}+\de\Lb Z^{\al}\vp_{\ga}\big)\\
		&\les\de^{2-2\ve0+2[1-(1-\ve0)p]}+\mu_{\min}^{-2b_{|\al|+1}}(\t)\de^{2-2\ve0}+\mu_{\min}^{-2b_{|\al|+1}}(\t)\de^{-1}\int_0^u\de^{1-(1-\ve0)p}\Ft^b_{1,\leq|\al|+1}(\t,u')\d u'\\
		&\quad+\mu_{\min}^{-2b_{|\al|+1}}(\t)\f{1}{2b_{|\al|+1}-1}\de\Et^b_{2,\leq|\al|+1}(\t,u)+\mu_{\min}^{-2b_{|\al|+1}}(\t)\de^{-[1-(1-\ve0)p]}\int_1^\t\de\Et^b_{2,\leq|\al|+1}(\t',u)\d\t',
	\end{split}
\end{equation}
where the key point is to choose $b_{|\al|+1}=b_{|\al|+2}-1$ in \eqref{next-to-top order error estimate 1} and
\eqref{next-to-top order error estimate 2}.

Collecting \eqref{next-to-top order error estimate 1} and \eqref{next-to-top order error estimate 2}, together with Gronwall inequality, the next-to-top order energy inequality can be obtained as follows
\begin{equation}\label{next-to-top order energy inequality}
	\begin{split}
		&\quad\de^{1-(1-\ve0)p}\Et^b_{1,\leq|\al|+1}(\t,u)+\de^{1-(1-\ve0)p}\Ft^b_{1,\leq|\al|+1}(\t,u)\\
		&\quad+\de\Et^b_{2,\leq|\al|+1}(\t,u)+\de\Ft^b_{2,\leq|\al|+1}(\t,u)+\Kt^b_{\leq|\al|+1}(\t,u)\\
		&\les\de^{2-2\ve0}.
	\end{split}	
\end{equation}

Due to $b_{|\al|+2-n}=b_{|\al|+2}-n$, then there exists an $n_0\in\mathbb{N}$ such that
\begin{equation*}
	\Et^b_{i,\leq|\al|+2-n_0}(\t,u)=\sup_{1\leq\t'\leq\t}\Et_{i,\leq|\al|+2-n_0}(\t',u),\ i=1,2.
\end{equation*}
In this case, we have
\begin{equation*}
	\de^{1-(1-\ve0)p}\Et^b_{1,\leq 2N-8-n_0}(\t,u)+\de\Et^b_{2,\leq 2N-8-n_0}(\t,u)\les\de^{2-2\ve0}.
\end{equation*}
Together with the following Sobolev embedding formula on $S_{\t,u}$
\begin{equation}\label{Sobolev embedding on S_tu}
	\|f\|_{L^{\infty}(S_{\t,u})}\les\frac{1}{\t}\sum_{|\be|\leq 2}\|R_i^{\be}f\|_{L^2(S_{\t,u})},
\end{equation}
one has  that for $|\al|\leq 2N-11-n_0$,
\begin{equation}\label{close bootstrap assumptions}
	\begin{split}
		\de^{l+s[1-(1-\ve0)p]}|Z^{\al}\vp_{\ga}|
		&\les\sum_{|\be|\leq 2}\de^l\|R_i^{\be}Z^{\al}\vp_{\ga}\|_{L^2(S_{\t,u})}\\
		&\les\de^{\f12}\left(\sqrt{\Et^b_{1,\leq 2N-8-n_0}}+\sqrt{\Et^b_{2,\leq 2N-8-n_0}}\right)\\
		&\les\de^{\f12}\sqrt{\de^{1-2\ve0}}\\
		&\les\de^{1-\ve0},
	\end{split}
\end{equation}
which closes the bootstrap assumptions \eqref{bootstrap assumptions} by choosing $N$ such that $2N-11-n_0\geq N$.

\subsection{Shock formation}\label{Section 8.2}

According to \eqref{transport equation of mu}, \eqref{transport inequality of rho Lb vp_ga} and Newton-Leibniz formula,
we have
\begin{equation*}
 	\begin{split}
 		&\quad L\mu(\t,u,\vt)\\
 		&=\f12pc^4\big[\f{1}{\t}\de^{1-\ve0}\phi_1+O\big(\de^{1-\ve0+(1-\ve0)p}\big)\big]^{p-1}\big[-\f{1}{\t}\de^{-\ve0}\p_s\phi_1+O\big(\de^{-\ve0+(1-\ve0)p}\big)\big]+O\big(\de^{(1-\ve0)p-[1-(1-\ve0)p]}\big)\\
 		&=-\f12p\phi_1^{p-1}\p_s\phi_1\f{1}{\t^p}\de^{-[1-(1-\ve0)p]}+O\big(\de^{(1-\ve0)p-[1-(1-\ve0)p]}\big).
 	\end{split}
\end{equation*}
Hence,
\begin{equation*}
 	\begin{split}
 		\mu(\t,u,\vt)
 		&=\mu(1,u,\vt)+\int_1^{\t}\f{\tau^pL\mu(\tau,u,\vt)}{\tau^p}\d\tau\\
 		&=1-\f{p}{2}\de^{-[1-(1-\ve0)p]}\phi_1^{p-1}\p_s\phi_1\f{1}{p-1}\left(1-\f{1}{\t^{p-1}}\right)+O\big(\de^{(1-\ve0)p}\big).
 	\end{split}
\end{equation*}
Therefore, for sufficiently small $\de>0$, $\t$ can not be greater than $t^*$, otherwise $\mu$ would
be negative in terms of \eqref{shock-formation assumption of data}. In other words, for \eqref{main equation}
with \eqref{initial data},
under assumption \eqref{shock-formation assumption of data},
the shock is formed  before $t^*$.

\begin{remark}\label{second order derivatives blow up}
Near the blow-up time of \eqref{main equation} with \eqref{initial data}, i.e., $\mu\rightarrow0$, then $\mu<\f{1}{10}$ holds.
By \eqref{key estimate of L mu}, one has
\begin{equation*}
 	\begin{split}
 		-cTc&=L\mu-\mu c^{-1}Lc\\
 		    &\les-\de^{-[1-(1-\ve0)p]}+O\big(\de^{(1-\ve0)p-[1-(1-\ve0)p]}\big)\\
 		    &\les-\de^{-[1-(1-\ve0)p]}.
 	\end{split}
\end{equation*}
This derives
\begin{equation*}
 	c^{-1}\mu|\Tr\vp_0|=|T\vp_0|=\left|\f{-cTc}{\f12pc^4\vp_0^{p-1}}\right|\gtrsim\de^{-[1-(1-\ve0)p]}.
\end{equation*}	
It follows from $g(\Tr,\Tr)=1$ that
\begin{equation*}
 	\de^{-[1-(1-\ve0)p]}\mu^{-1}\les|\Tr^i\p_i\vp_0|\les\left(\sum_{i=1}^3(\Tr^i)^2\right)^{\f12}\left(\sum_{i=1}^3(\p_i\vp_0)^2\right)^{\f12}=|\na_x\vp_0|.
\end{equation*}
Therefore, for the solution $\phi$ of \eqref{main equation} with \eqref{initial data},
\begin{equation*}
 	|\p^2\phi|\rightarrow+\infty\ \text{as}\ \mu\rightarrow0.
\end{equation*}
\end{remark}

\section*{Appendix}

\appendix

\section{The existence of short pulse initial data}\label{Section A}

In this section, we give the existence of short pulse initial data in \eqref{initial data}
which satisfy the outgoing constraint condition \eqref{outgoing constraint condition}. Indeed, for any fixed $\phi_0\in C_0^{\infty}\big((-1,0)\times\mathbb{S}^2\big)$, motivated by Appendix in \cite{DLY}, we choose $\phi_1$ as
\begin{equation}\label{phi_1}
	\phi_1=-\p_s\phi_0-\phi_0\de-\f{1}{2(p+1)}(-\p_s\phi_0)^{p+1}\de^{(1-\ve0)p}.
\end{equation}
It is pointed out that although the selection of $\phi_1$ depends on $\de$, the $C^\infty$-norm of $\phi_1$ is actually independent of $\de$.

Due to $\p_t^2\phi=c^2\De\phi=c^2(\p_r^2\phi+\f{2}{r}\p_r\phi+\f{1}{r^2}\Des\phi)$
from \eqref{main equation} and $r=1+s\de$, one then has
\begin{equation}\label{k=2}
	\begin{split}
		&\quad(\p_t+\p_r)^2\phi(1,x)\\
		&=2(\p_s\phi_1+\p_s^2\phi_0)\de^{-\ve0}+2\p_s\phi_0\de^{1-\ve0}-\phi_1^p\p_s^2\phi_0\de^{(1-\ve0)p-\ve0}+O\big(\de^{2-\ve0-2\max\{0,1-(1-\ve0)p\}}\big)\\
		&=2\p_s\big[\phi_1+\p_s\phi_0+\phi_0\de+\f{1}{2(p+1)}(-\p_s\phi_0)^{p+1}\de^{(1-\ve0)p}\big]\de^{-\ve0}+O\big(\de^{2-\ve0-2\max\{0,1-(1-\ve0)p\}}\big)\\
		&=O\big(\de^{2-\ve0-2\max\{0,1-(1-\ve0)p\}}\big)
	\end{split}
\end{equation}
and
\begin{equation}\label{k=1}
	\begin{split}
		&\quad(\p_t+\p_r)\phi(1,x)\\		
		&=(\phi_1+\p_s\phi_0)\de^{1-\ve0}\\
		&=O\big(\de^{\min\{1,(1-\ve0)p\}}\big)\de^{1-\ve0}\\
		&=O\big(\de^{2-\ve0-\max\{0,1-(1-\ve0)p\}}\big).
	\end{split}
\end{equation}


\begin{thebibliography}{99}
	
\bibitem{Ali1} S. Alinhac, \textit{Blowup of small data solutions for a class of quasilinear wave equations in two space dimensions. II}, Acta Math. 182, no. 1, 1-23 (1999)
	
	
	\bibitem{Ali2} S. Alinhac, \textit{The null condition for quasilinear wave equations in two space dimensions. II}, Amer. J. Math. 123, no. 6, 1071-1101 (2001)
	


\bibitem{Alt1} D.Alterman, J. Rauch, \textit{Nonlinear geometric optics for short pulses},
 J. Differential Equations 178, no. 2, 437-465  (2002)


\bibitem{Alt2} D.Alterman, J. Rauch, \textit{Diffractive nonlinear geometric optics for short pulses},
SIAM J. Math. Anal. 34, no. 6, 1477-1502  (2003)

	
	\bibitem{Ch1} D. Christodoulou, \textit{The formation of shocks in 3-dimensional fluids}, Monographs in Mathematics,
European Mathematical Soc., 2007.
	
	
	\bibitem{Ch2} D. Christodoulou, \textit{The formation of black holes in general relativity}, Monographs in Mathematics,
European Mathematical Soc., 2009.
	
	
	\bibitem{Ding1} Ding Bingbing, Witt Ingo, Yin Huicheng, \textit{The small data solutions of general 3-D quasilinear wave equations. II,} J. Differential Equations 261, no. 2, 1429-1471 (2016)
	
	
	\bibitem{Ding2} Ding Bingbing, Xin Zhouping, Yin Huicheng, \textit{Global smooth large solutions to the 2D isentropic and irrotational Chaplygin gases}, arXiv:2204.08181, 18 Apr 2022, Preprint (2022)
	
	
	\bibitem{Ding3} Ding Bingbing, Xin Zhouping, Yin Huicheng, \textit{Global smooth large data solutions of 4-D quasilinear wave equations}, Preprint (2022)
	
	
\bibitem{DLY} Ding Bingbing, Lu Yu, Yin Huicheng, \textit{On the critical exponent $p_c$ of the 3D quasilinear wave equation $-\big(1+(\partial_t\phi)^p\big)\partial_t^2\phi+\Delta\phi=0$ with short pulse initial data. I, global existence}, Preprint (2022)
	
	
	\bibitem{H} L. H\"ormander, \textit{Lectures on nonlinear hyperbolic equations}, Mathematiques \& Applications, Vol. 26, Springer-Verlag, Heidelberg, 1997.
	
\bibitem{Hunter}	J.K. Hunter, A. Majda, R. Rosales, \textit{Resonantly interacting, weakly nonlinear hyperbolic waves. II.
Several space variables,} Stud. Appl. Math. 75, no. 3, 187-226  (1986)



	\bibitem{K-P} S. Klainerman, G. Ponce,\textit{ Global small amplitude solutions to nonlinear evolution equations,}
Comm. Pure Appl. Math. 36, no. 1, 133-141 (1983)
	
	
	\bibitem{K-R} S. Klainerman, I. Rodnianski, \textit{On the formation of trapped surfaces},
	Acta Math. 208, no. 2, 211-333 (2012)
	
	
\bibitem{Majda}	A. Majda, R. Rosales, \textit{Resonantly interacting weakly nonlinear hyperbolic waves. I.
A single space variable.} Stud. Appl. Math. 71, no. 2, 149-179  (1984)
	
	\bibitem{M-Y} Miao Shuang, Yu Pin, \textit{On the formation of shocks for quasilinear wave equations},
Invent. Math. 207, no. 2, 697-831 (2017)
	
	
	\bibitem{Sp} J. Speck, \textit{Shock formation in small-data solutions to 3D quasilinear wave equations},
Mathematical Surveys and Monographs, 214. Amer. Math. Soc., Providence, RI, 2016.
	
	
	
\end{thebibliography}
\end{document}